\documentclass[11pt]{article} 
\usepackage[utf8]{inputenc} 

\usepackage{biblatex}
\usepackage[margin=1.25in]{geometry} 
\usepackage{graphicx}

\usepackage{booktabs} 
\usepackage{array} 
\usepackage{paralist} 
\usepackage{verbatim} 
\usepackage{mathrsfs} 
\usepackage{physics}
\usepackage{amssymb}
\usepackage{amsthm}
\usepackage{amsmath,amsfonts,amssymb}
\usepackage{esint}
\usepackage{graphics}
\usepackage{svg}
\usepackage{caption}
\usepackage{enumerate}
\usepackage{mathtools}
\usepackage{multirow}
\usepackage{multicol}
\usepackage[protrusion=true,final]{microtype}

\usepackage{physics}
\usepackage{xfrac}
\usepackage{bbm}
\usepackage{subcaption}
\usepackage{stmaryrd} 
\usepackage{mathtools}

\let\oldsquare\square 

\usepackage{mathabx}

\renewcommand{\square}{\oldsquare}

\usepackage{tocloft}

\usepackage[usenames,dvipsnames]{xcolor}
\usepackage[colorlinks=true, pdfstartview=FitV, linkcolor=blue, citecolor=blue, urlcolor=blue]{hyperref}

\usepackage[normalem]{ulem}
\usepackage{aliascnt}
\usepackage{cleveref}

\numberwithin{equation}{section}

\newtheorem{theorem}{Theorem}[section]

\newtheorem{corollary}[theorem]{Corollary}
\newtheorem{proposition}[theorem]{Proposition}
\newtheorem{lemma}[theorem]{Lemma}
\theoremstyle{definition}

\Crefname{corollary}{Corollary}{Corollaries}
\Crefname{proposition}{Proposition}{Propositions}

\let\originalleft\left
\let\originalright\right
\renewcommand{\left}{\mathopen{}\mathclose\bgroup\originalleft}
\renewcommand{\right}{\aftergroup\egroup\originalright}

\newcommand{\vertiii}{\vert\kern-0.3ex\vert\kern-0.25ex\vert}

\newcommand*{\R}{\ensuremath{\mathbb{R}}}

\DeclareSymbolFont{boldoperators}{OT1}{cmr}{bx}{n}
\SetSymbolFont{boldoperators}{bold}{OT1}{cmr}{bx}{n}

\renewcommand{\R}{\mathbb{R}}

\DeclareMathOperator*{\argmin}{argmin}

\makeatletter
\newcommand\avsuminner[2]{
	{\sbox0{$\m@th#1\sum$}
		\vphantom{\usebox0}
		\ooalign{
			\hidewidth
			\smash{\,\rule[.23em]{8.8pt}{1.1pt} \relax}
			\hidewidth\cr
			$\m@th#1\sum$\cr
		}
	}
}
\makeatother

\makeatletter
\newcommand\avsuminnerr[2]{
	{\sbox0{$\m@th#1\sum$}
		\vphantom{\usebox0}
		\ooalign{
			\hidewidth
			\smash{\,\rule[.23em]{6pt}{0.7pt} \relax}
			\hidewidth\cr
			$\m@th#1\sum$\cr
		}
	}
}
\makeatother

\def\XXint#1#2#3{{\setbox0=\hbox{$#1{#2#3}{\int}$}
		\vcenter{\hbox{$#2#3$}}\kern-.5\wd0}}

\theoremstyle{definition}
\newtheorem{appendixproof}{Proof}[section]

\crefname{appendixproof}{Proof}{Proofs}
\Crefname{appendixproof}{Proof}{Proofs}

\newcommand{\proofinappendix}[1]{
\par\smallskip
\noindent\textit{See \Cref{pf:#1} in the Appendix.}
}

\makeatletter 
\newcommand{\negphantom}{\v@true\h@true\negph@nt} 
\newcommand{\neghphantom}{\v@false\h@true\negph@nt} 
\newcommand{\negph@nt}{\ifmmode\expandafter\mathpalette 
	\expandafter\mathnegph@nt\else\expandafter\makenegph@nt\fi} 
\newcommand{\makenegph@nt}[1]{
	\setbox\z@\hbox{\color@begingroup#1\color@endgroup}\finnegph@nt} 
\newcommand{\finnegph@nt}{
	\setbox\tw@\null 
	\ifv@ \ht\tw@\ht\z@\dp\tw@\dp\z@\fi \ifh@\wd\tw@-\wd\z@\fi\box\tw@} 
\newcommand{\mathnegph@nt}[2]{
	\setbox\z@\hbox{$\m@th #1{#2}$}\finnegph@nt} 
\makeatother

\newcommand{\Ai}{\mathrm{Ai}}
\newcommand{\Bi}{\mathrm{Bi}}

\usepackage{titlesec}

\newcommand{\addperiod}[1]{#1.}
\titleformat{\section}
{\normalfont\Large}{\thesection.}{0.5em}{}
\titleformat*{\subsection}{\normalfont\large}
\titleformat{\subsubsection}[runin]
{\bfseries}
{\thesubsubsection.}
{0.5em}
{\addperiod}

\titleformat*{\subsubsection}{\bfseries}
\titleformat*{\paragraph}{\bfseries}
\titleformat*{\subparagraph}{\large\bfseries}

\title{\bf \Large Magnetic Stabilization of the Greenhill Instability and Hysteretic Switching in Slender Rods}

\author{Vivekanand Dabade
	\thanks{Department of Aerospace Engineering, Indian Institute of Science.
		{\footnotesize \href{mailto:dabade@iisc.ac.in}{dabade@iisc.ac.in}.}
	}
    \and Likhit  Ganedi
	\thanks{Department of Mathematics, University of Utah.
		{\footnotesize \href{mailto:lganedi@math.utah.edu}{lganedi@math.utah.edu}.}
	}
	\and 
   Mohd Tahseen
	\thanks{Department of Aerospace Engineering, Indian Institute of Science.
		{\footnotesize \href{mailto:tahseenmohd@iisc.ac.in}{tahseenmohd@iisc.ac.in}.}
	}
	\and Raghavendra Venkatraman
	\thanks{Department of Mathematics, University of Utah.
		{\footnotesize \href{mailto:raghav@math.utah.edu}{raghav@math.utah.edu}.}
	}
}
\date{\today}

\usepackage[nottoc,notlot,notlof]{tocbibind}

\begin{document}

\maketitle

\begin{center}
    {\itshape This article is dedicated to the memory of Professor Robert V. Kohn, with respect and admiration.}
\end{center}

\vspace{0.5em}

\begin{abstract}
    In a bulk ferromagnet, magnetic energy is six orders of magnitude weaker than
    elastic energy. Slenderness removes this disparity. In a thin ferromagnetic rod, the two energies
    scale differently with aspect ratio, and an ordinary nickel wire deforms nonlinearly even
    under modest remote fields. We develop a three-dimensional variational theory of a
    ferromagnetic Kirchhoff rod under gravity and magnetic loading and study it through
    bifurcation analysis, computation, and experiment. A uniform field can stabilize a vertical rod
    against buckling under its own weight, the classical Greenhill instability. We give sharp
    criteria for clamped and pinned supports in terms of an Airy-type principal
    eigenvalue, and give a fine description of the character of the bifurcation landscape, which has qualitatively different features than its non-magnetic counterpart. When subject to the field generated by a
    permanent magnet, we identify a rich bifurcation landscape as the magnet is  rotated quasistatically along a circle of radius slightly larger  than the rest length of the rod. We prove that instability occurs only through a
    planar mode and that each snap is a nondegenerate saddle-node. We also locate the two
    folds that produce hysteresis.

    \smallskip \noindent Experiments on $50\,\mu$m nickel wires confirm the
    Greenhill threshold, restabilization at~$30$ mT for both clamped and pinned supports, as well as the salient features of the hysteresis loop such as switching angles, width, and geometry of folds accurately.
    Our analysis and experiments offer principled design rules for field-stabilized filaments and remotely switched bistable
    actuators. Our analysis also offers a well-conditioned numerical method is valid near the switching events, where the direct numerical minimization is strongly ill-conditioned.
\end{abstract}

\setcounter{tocdepth}{2}
\tableofcontents

\section{Introduction and motivation}

Geometric slenderness fundamentally alters the interplay between elasticity and magnetism in ferromagnetic solids. In bulk ferromagnets, the elastic energy density, set by the Young's modulus $E \sim 10^{11}\,\mathrm{J/m^3}$, exceeds the magnetostatic energy density, set by the constant $K_d \sim 10^{5}\,\mathrm{J/m^3}$, by roughly six orders of magnitude. Magnetically induced strains are therefore correspondingly infinitesimal, on the order of $10^{-5}$. In slender geometries, however, the elastic and magnetic energies scale differently with the aspect ratio, and this disparity in scaling allows their magnitudes to become comparable \cite{AvatarDabade2024,DeSimone2002}. As a consequence, rods, ribbons, and filaments of ordinary ferromagnetic materials such as nickel can undergo large, geometrically nonlinear deformations under modest magnetic fields applied remotely. This capability  is inaccessible to bulk magnetostriction, but underlies the growing interest in magnetically actuated soft robots, shape-programmable magnetic matter, sensors, and deployable structures \cite{kim2019ferromagnetic,ren2024design}. Slender ferromagnetic bodies  provide a setting in which elasticity and magnetism compete on equal terms, and in which classical stability problems admit qualitatively new solutions.\\

In this article we develop a variational theory of a slender, inextensible, unshearable rod (that is, a Kirchoff rod) subject to gravity and study the impact of different types of magnetic loading. We validate its predictions quantitatively against numerical computations and laboratory experiments on ferromagnetic nickel wires. Our methods of analysis are quite general, but we pay special attention to two canonical types of externally applied magnetic fields that are natural from the viewpoint of applications and are easy to arrange experimentally:
\begin{enumerate}
    \item[(a)] the stabilization of a gravity-opposed ferromagnetic rod by a spatially uniform magnetic field, and
    \item[(b)] the equilibrium and stability of a rod driven by the  non-uniform field due to a nearby permanent magnet.
\end{enumerate}
Slender magnetic structures and their interplay with elasticity has been considered before, and our literature review section surveys this material. Departing from these studies, the \textbf{central goal of our paper} is to pursue a holistic understanding of the magneto-elastic behavior of slender rods that combines (a) modeling and variational analysis, (b) bifurcation theory, (c) matching laboratory experiments, and (d) numerical simulation. Indeed, we will make the case that each of these viewpoints informs and compliments the others. \\

The model that we work with combines a dimensionally reduced micromagnetics model from~\cite{SS}, modeling the deformation of the spatial rod using Kirchoff rod theory. Referring the reader to Section~\ref{sec:model} for a detailed explanation of this model which is associated with local minimizers of an energy~$E.$ To explain this energy, we let~$r:[0,L] \to \R^3$ denote the centerline of the Kirchoff rod that is arc-length parametrized. The model we work with is one-dimensional, and it is conventional to associate to the rod~$r$ an orthonormal director frame~$R(s) = (d_1(s),d_2(s),d_3(s)) \in SO(3),$ and take~$d_3(s) = r'(s).$ The arclength coordinate runs from~$s=0$ at the supported end of the rod (which maybe either clamped or pinned), through~$s=L$ which is the free end of the rod. Associated with the orthonormality constraint of~$R(s)$ the skew-symmetric matrix relation
\begin{equation*}
    \begin{pmatrix}
        d_1'(s)\\
        d_2'(s)\\
        d_3'(s)
    \end{pmatrix} = \begin{pmatrix}
        0 &-\tau(s) & \kappa_2(s)\\
        \tau(s) & 0 & -\kappa_1(s)\\
        -\kappa_2(s) & \kappa_1(s) & 0
    \end{pmatrix}\begin{pmatrix}
        d_1(s)\\
        d_2(s)\\
        d_3(s)
    \end{pmatrix}\,,
\end{equation*}
for bending strains~$\kappa_1,\kappa_2$ and twist strain~$\tau$.
After a careful non-dimensionalization procedure outlined in Section~\ref{sec:model}, we arrive at the energy
\begin{multline}
    \mathcal E[R,m]
    =
    \frac12
    \int_0^1
    \left[
        \kappa_1^2+\kappa_2^2
        +
        \mathsf C \tau^2
        \right]\,ds
    +
    \mathsf G
    \int_0^1
    (1-s)d_3\cdot e_3\,ds
    \\ +
    \frac{\mathsf A}{2}
    \int_0^1
    \left[
        1-(m\cdot d_3)^2
        \right]\,ds
    -
    \mathsf M
    \int_0^1
    m\cdot H_{\mathrm{ext}}(r(s))\,ds\,.
    \label{eq:nondimensional_master_energy-intro}
\end{multline}
Here~$\mathsf{C},\mathsf{G},\mathsf{A},$ and~$\mathsf{M}$ are positive constants whose meaning is explained in Section~\ref{sec:model}. Briefly, nondimensionalizing by the bending elastic energy,~$\mathsf{C}$ denotes the relative twist elastic modulus,~$\mathsf{G}$ denotes the gravitational potential energy relative to the bending energy,~$\mathsf{A}$ denotes the stray-field energy relative to the bending energy, and finally,~$\mathsf{M}$ denotes the ratio of Zeeman energy to the bending energy.  The vector field~$m:[0,L] \to \R^3$ denotes the magnetization, which, after saturation, satisfies the pointwise constraint~$|m|=1.$ The externally applied field is~$H_{\mathrm{ext}},$ and is evaluated along the rod~$r(s).$ \\

Evidently, in~\eqref{eq:nondimensional_master_energy-intro}, the first term is the elastic energy of the Kirchoff rod~$r,$ the second term represents non-dimensional gravitational potential energy. The last two terms are the result of the dimensional reduction procedure from three-dimensional micromagnetics energy functional and correspond to the stray-field energy and the Zeeman energy respectively. Usual models of micromagnetics also include an exchange and anisotropy energy terms as well. As we work with a magnetically soft material under relatively large external fields, we ignore both these contributions (see~\cite{AvatarDabade2024} for a discussion in a planar context). \\

Two noteworthy features are worth highlighting at this juncture: (i) common treatments of Kirchoff rods are planar, whereas the analysis of the present paper is \textit{fully three-dimensional}, and (ii) the coupling with magnetism produces new and novel physics that is not present in the classical Euler elastica. Explaining these new observations analytically in the setting of experiments (a) and (b) above with a remarkable agreement between analysis and experiment is the main contribution of our paper. We turn to describing each of these experiments next.

\subsection{Magnetic stabilization of the Greenhill Instability.} A canonical example of elastic instability is the Greenhill problem: a slender vertical rod supported from below eventually loses stability under its own weight when its length exceeds a critical value \cite{Greenhill1881,frisch1961analysis}. Unlike Euler buckling, in which compression is imposed through an external end load, the destabilizing force in Greenhill buckling is generated by the self-weight of the rod. This leads to an Airy spectral problem and a critical length determined by the competition between bending stiffness and gravity. The gravity-opposed straight configuration is therefore a natural benchmark for asking: \emph{What is the effect of magnetism on the Greenhill Instability?}\\

When a sufficiently strong external magnetic field is applied to a soft ferromagnet, its magnetization aligns with the field. As explained earlier, the nonlocal stray-field energy reduces to a local magnetic anisotropy, and acts as an effective restoring mechanism that can stabilize vertical configurations. This occurs for both clamped and pinned supports, although the underlying stability mechanisms and bifurcation structures differ substantially. This is made precise in the following theorem.

\begin{theorem}\label{t.intro-unif}
    \begin{enumerate}
        \item (Clamped boundary conditions) A slender ferromagnetic rod is stable provided
              $\mathsf{A} > -\beta_0^{\rm cl}(\mathsf{G})$, where $\beta_0^{\rm cl}(\mathsf{G})$
              is the principal eigenvalue of the clamped rod held against gravity in zero field,
              \begin{equation*}
                  \beta_0^{\rm cl}(\mathsf{G}) := \inf\left\{
                  \frac{\int_0^1 \big[(u')^2 - \mathsf{G}(1-s)u^2\big]\,ds}{\int_0^1 u^2\,ds}
                  \;:\; u \in H^1(0,1)\setminus\{0\},\ u(0)=0 \right\},
              \end{equation*}
              which is controlled by $\mathsf{G}$ and the largest zero of the Airy function
              $\mathrm{Ai}$. One has $\beta_0^{\rm cl}(\mathsf{G}) > 0$ exactly when
              $\mathsf{G} < \mathsf{G}_{\rm crit} \approx 7.8373$, so below the Greenhill
              threshold no field is needed. If $\beta_0^{\rm cl}(\mathsf{G})$ is close to $0$,
              then the onset of bifurcation is of supercritical pitchfork type. This branch
              continuously deforms, eventually to a subcritical bifurcation when $\beta_0^{\rm cl}(\mathsf{G})\leq -\frac{\mathsf{G}}{4}$ (see Theorem~\ref{thm:uniform_physical_clamped} for details).

        \item (Pinned boundary conditions) Let $\beta_0^{\rm p}(\mathsf{G})$ denote the
              same infimum taken over all of $H^1(0,1)\setminus\{0\}$, without the constraint
              $u(0)=0$. In the absence of an applied field, the Kirchoff rod is
              \emph{unstable} against gravity for any length, since
              $\beta_0^{\rm p}(\mathsf{G}) \le -\mathsf{G}/2 < 0$. On the other hand, with a
              uniformly applied magnetic field in the $e_3$ direction, a slender rod is
              \emph{stable} provided $\mathsf{A} > -\beta_0^{\rm p}(\mathsf{G})$, a threshold
              controlled by $\mathsf{G}$ and the largest zero of the derivative of the Airy
              function, $\mathrm{Ai}'$. At the onset of instability, the bifurcation is of
              subcritical pitchfork type. (see Theorem~\ref{thm:uniform_physical_pinned} for details).

        \item
              In either case, a sufficient length for stability is $L < K_d/(\rho g)$,
              and notably, is independent of the diameter of the rod
              (see Corollary~\ref{thm:stabilization} and Figure~\ref{fig:pinned_H_not_0}).
    \end{enumerate}
\end{theorem}

The predictions of Theorem~\ref{t.intro-unif} are confirmed by experiment. Clamped nickel wires of diameter~$d = 50\,\mu\mathrm{m}$ and of lengths 8, 10, and 12 cm match well with the theoretically predicted critical length $L_c \in [11.15, 14.05] $ cm, depending  on the variability of the Young modulus~$\mathsf{E}$ of the Nickel wire. As seen in~\Cref{fig:clamped2}, in the gravity-opposed configuration without a field, the 8 and 10 cm wires remain straight while the 12 cm wire buckles, in agreement with the Greenhill threshold. Upon application of a uniform 30 mT field parallel to the~$e_3$ axis, all three wires, including the previously unstable 12 cm specimen, recover a stable vertical configuration. This directly confirms the predicted magnetic stabilization of the Greenhill instability.\\

\noindent For the pinned case, realized experimentally with a low-friction floating spherical support, an 11 cm wire is unstable against gravity in zero field, exactly as the theory suggests, and is stabilized by the same 30 mT field, consistent with the sufficient condition in Theorem~\ref{t.intro-unif}, part 3 (see also Corollary~\ref{thm:stabilization} for a more detailed condition). Moreoever, for our material parameters, Corollary~\ref{thm:stabilization} guarantees stability for all lengths below 1.70 m, regardless of the diameter of the rod.

\subsection{Hysteretic switching by a localized magnetic source}
The second experiment (Figs.~\ref{fig:radial_setup} and~\ref{fig:radial_deformed}) concerns a clamped ferromagnetic rod subject to the field generated by small, uniformly magnetized spherical permanent magnet. The magnet is constrained to move along a circle centered at the clamp, of radius slightly greater than the length of the rod. We track the deformation of the rod, in particular the position of its free end, as the magnet is moved along this circle. In contrast to the uniform field of the preceding section, the applied field is now strongly non-uniform, but it is known in closed form. The field exterior to a uniformly magnetized sphere is exactly that of a point dipole located at its center. This observation renders the problem analytically tractable. In a slender ferromagnetic rod the magnetization aligns with the tangent of the rod. Indeed, this orientation is favored because it minimizes the stray-field energy at negligible exchange cost. Also, the dipole field derives from a scalar potential. The Zeeman integrand is therefore an exact derivative along the centerline, and the Zeeman energy reduces to an endpoint term that depends on the configuration only through the position of the free end. In effect, the distributed magnetic loading reduces to a single force acting at the tip, which makes the tip position the natural observable.\\

\noindent In the experiment, the spherical NdFeB permanent magnet is constrained to move along a circular track of radius $11.25~\mathrm{cm}$ centered at the clamped end of the wire. The magnetization of the  spherical permanent magnet maintained radially outward. The magnet is initially placed at $\alpha = -50^\circ$ and advanced quasi-statically along the track in the clockwise sense to $\alpha = +50^\circ$, where $\alpha$ denotes the angular position of the magnet on the track, measured from the vertical. The response of the rod tip is marked by two sharp transitions. Over an initial range of magnet angles the rod remains nearly vertical, essentially undisturbed by the magnet. At a first critical angle the rod undergoes a large, sudden deflection toward the magnet, after which its tip tracks the magnet as it moves along the circle. This tracking persists until a second critical angle, of magnitude different from the first, at which the rod abruptly releases and returns to a nearly vertical configuration. Beyond this point the tip is again independent of the magnet position. When the protocol is reversed, the magnet starting at $\alpha = +50^\circ$ and rotated counterclockwise, the same sequence of capture, tracking, and release recurs, but the forward and reverse paths do not coincide. The response is therefore path dependent and exhibits hysteresis. See Fig. \ref{fig:theta_plot}. \\

\noindent This experiment, too, turns out to be analytically tractable due to three important observations:
\begin{enumerate}
    \item[(i)] The first simplification is a planar reduction. Configurations lying in the plane spanned
          by the clamp axis and the magnet solve the full three-dimensional Euler-Lagrange
          system, and moreover, we show that if an equilibrium loses stability, it must do so through a planar mode. Making a planar reduction, we can parametrize the rod geometry through its tangent angle $\theta(s)$, measured from the
          vertical, with $\theta(0) = 0$ at the clamp.
    \item[(ii)] The second observation is a splitting of $\theta$ into one dominant mode and a
          correction. Let $\psi(s)$ be the profile that minimizes the quadratic energy of the
          straight rod, with the magnet directly above it ($\alpha = 0$), among all tangent
          angle profiles with mean value one. It solves a linear boundary value problem with an
          explicit potential and is positive on $(0,1]$ (see Lemma~\ref{lem:mode_decomposition}) . By construction, $\psi$ is
          energetically decoupled at the straight state from every mean-zero perturbation.
          Every planar configuration then admits the decomposition
          \begin{equation*}
              \theta = q\psi + w, \qquad q = \int_0^1 \theta\,ds, \qquad \int_0^1 w\,ds = 0.
          \end{equation*}
          The scalar $q$ is the mean tangent angle and measures how far the rod has deflected.
          For $(q,\alpha)$ in an explicit range, the correction $w = w(q,\alpha)$ is fixed by
          a contraction mapping argument. Planar equilibria are then described by the zero set of a function of two real variables:
          \begin{equation*}
              F(q,\alpha) := \delta\mathcal{E}_{{\rm pl},\alpha}[q\psi + w(q,\alpha)](\psi),
          \end{equation*}
          the first variation of the planar energy at $q\psi + w$ in the direction $\psi$.
    \item[(iii)] The third observation is that we can read stability off $F$. The derivative $\partial_q F$ is the planar
          second variation in the critical direction. An equilibrium is stable when
          $\partial_q F > 0$ and unstable when $\partial_q F < 0$. Stability is lost where
          $\partial_q F = 0$ on the curve $F = 0$, which is where the branch of equilibria
          turns back in $\alpha$. This is a fold, and it is the mechanism of the snap.
\end{enumerate}

Throughout, lengths are measured in units of $L$, the magnet center moves on the
circle of radius $D>1$ about the clamp, and $\delta = D-1$ is the clearance between
this circle and the tip of the straight rod. In the notation of the experiment,
$\alpha = \theta_m$ and $\Theta_{\rm tip} = \theta_{\rm tip}$.

\begin{theorem}\label{permanent-intro}
    Assume that the magnetic pull at the tip dominates the weight of the rod,
    $2\mathsf{M}/\delta^3 > \mathsf{G}$, and that the rod is strictly stable against gravity without an applied field: $\beta_0^{\rm cl}(\mathsf{G}) > \sqrt{6}\,\mathsf{M}/(9\delta^3)$,
    where $-\beta_0^{\rm cl}(\mathsf{G})$ is the critical value of $\mathsf{A}$ in
    Theorem~\ref{t.intro-unif}. Then, on the region of parameter space where the
    reduction is justified ( characterized by explicit, checkable inequalities), the following hold:
    \begin{enumerate}
        \item (Planar branch) A unique branch of planar equilibria emanates from the
              straight vertical rod at $\alpha = 0$, and along it the tip deflects toward the
              magnet. This branch of planar equilibria yield equilibria of the full three-dimensional
              problem. (see Corollary~\ref{cor:planar_branch}).

        \item (Incipient mode is planar) As the magnet is rotated through the upper
              semicircle, every equilibrium on this branch is strictly stable with respect to
              out-of-plane perturbations. If an equilibrium loses stability, it does so through
              a planar mode. (see Theorem~\ref{thm:permanent_physical}, part 1).

        \item (One-dimensional stability criterion) At each equilibrium on the branch, the
              planar second variation is positive definite on mean-zero perturbations, with an
              explicit margin. Its sign on all perturbations is therefore decided by the single
              number $\partial_q F(q,\alpha)$, the stiffness in the direction of the corrected
              mode $\xi = \psi + \partial_q w$: the equilibrium is strictly stable if
              $\partial_q F > 0$, degenerate with a one-dimensional kernel if $\partial_q F = 0$,
              and unstable with Morse index one if $\partial_q F < 0$.
              (see Theorem~\ref{thm:permanent_physical}).

        \item (Snap-through is a nondegenerate fold) Where $\partial_q F$ vanishes on the
              branch and $\partial_\alpha F\,\partial_{qq}F \neq 0$, the branch has a
              nondegenerate saddle-node $(q_*,\alpha_*)$: locally a parabola, with a
              stable and an unstable equilibrium on one side of $\alpha_*$ and none on the
              other, so the rod snaps. If $T_* := \partial_q \Theta_{\rm tip}(q_*,\alpha_*) \neq 0$,
              the measured tip angle satisfies
              \begin{equation*}
                  \alpha - \alpha_* = -\,\frac{\partial_{qq}F(q_*,\alpha_*)}
                  {2\,\partial_\alpha F(q_*,\alpha_*)\,T_*^{2}}\,
                  (\Theta_{\rm tip} - \Theta_{{\rm tip},*})^2
                  + o\big((\Theta_{\rm tip} - \Theta_{{\rm tip},*})^2\big),
              \end{equation*}
              a square-root departure of the tip angle from its fold value. The folds are the
              solutions of the well-conditioned two-dimensional system
              $F = \partial_q F = 0$. (see Theorem~\ref{thm:planar_fold},
              Corollary~\ref{cor:tip_angle_fold}, and Theorem~\ref{thm:permanent_physical},
              parts 2 and 3).
    \end{enumerate}
\end{theorem}

Theorem~\ref{permanent-intro} identifies the mechanism but does not by itself locate
the folds.
A direct Newton solve of the full equilibrium equations degenerates at a fold, since
the linearized operator is the planar second variation and acquires a zero eigenvalue
there. In the reduced formulation the critical direction is carried by the scalar
$q$, the complementary problem for $w$ stays coercive, and the folds are the
solutions of the two-dimensional system $F(q,\alpha) = \partial_q F(q,\alpha) = 0$.
Cheap initial guesses $(q_0,\alpha_0)$ come from the one-mode approximation
$\theta = q\psi$, that is, from dropping $w$. For the experimental parameters of
Table~\ref{tab:parameters_radial} the result is the following.

\begin{corollary}[Numerical predictions at the experimental parameters] \label{cor:intronum}
    Numerical solution of the reduced equations $F(q,\alpha) = \partial_q F(q,\alpha) = 0$
    gives the values below.
    \begin{center}
        \begin{tabular}{llccccc}
            \toprule
            $E$ & Fold & $q_0$ & $\alpha_0$ & $q_*$ & $\alpha_*$ & $\Theta_{{\rm tip},*}$ \\
            \midrule
            $200$ GPa & Lower & $0.142206$ & $29.85402^\circ$ & $0.143157$ & $30.12789^\circ$ & $8.20398^\circ$ \\
            & Upper & $0.539770$ & $39.73885^\circ$ & $0.528278$ & $39.30739^\circ$ & $30.35258^\circ$ \\[2pt]
            $100$ GPa & Lower & $0.173427$ & $39.45516^\circ$ & $0.178723$ & $40.95699^\circ$ & $10.24345^\circ$ \\
            & Upper & $0.919017$ & $63.23606^\circ$ & $0.872981$ & $61.50409^\circ$ & $50.41667^\circ$ \\
            \bottomrule
        \end{tabular}
    \end{center}

    At every computed fold, $\partial_\alpha F > 0$, $\partial_{qq} F < 0$ at each lower
    fold, and $\partial_{qq} F > 0$ at each upper fold. Thus the saddle-nodes are
    non-degenerate, the lower folds open toward increasing $\alpha$, and the upper folds
    open toward decreasing $\alpha$.

    All of the hypotheses of Theorem~\ref{thm:permanent_physical} are satisfied except
    for the contraction condition $K < 1/2$ at the upper fold for $E = 100$ GPa, where
    $K \approx 1.28 > 1/2$.

    The computed hysteresis width is
    \begin{equation}\label{eq:intro-hysteresis}
        \alpha_{\rm upper} - \alpha_{\rm lower} \approx
        \begin{cases}
            9.1795^\circ,  & E = 200 \text{ GPa},\\
            20.5471^\circ, & E = 100 \text{ GPa}.
        \end{cases}
    \end{equation}
    Reflection over the $e_3$ axis in the $e_1$--$e_3$ plane gives the corresponding
    negative-angle folds.
\end{corollary}

\noindent The two folds organize the response. The branch through the straight rod is the
\emph{engaged} branch, on which the tip tracks the magnet; it terminates at the upper
fold $\alpha_{\rm upper}$. The \emph{disengaged} branch, on which the rod hangs nearly
vertical, exists only for $|\alpha| \geq \alpha_{\rm lower}$ and terminates at the
lower fold. For $\alpha_{\rm lower} < |\alpha| < \alpha_{\rm upper}$ both branches
are stable, and the state the rod occupies depends on its history. A magnet swept
through the vertical therefore captures the rod at $|\theta_m| = \alpha_{\rm lower}$
on one side and releases it at $|\theta_m| = \alpha_{\rm upper}$ on the other.
Reversing the sweep reflects the picture. The unequal magnitudes of the capture and
release angles, and the resulting loop in Figure~\ref{fig:theta_plot}, reflect the double saddle-node structure producing a hysteresis loop with width~\eqref{eq:intro-hysteresis}.\\

The predictions of Theorem~\ref{permanent-intro} and
Corollary~\ref{cor:intronum} are confirmed by experiment (see Fig.~\ref{fig:theta_plot}). A clamped
nickel wire of diameter $d = 50\,\mu$m and length $9.2$ cm is driven by a $15$ mm
NdFeB sphere whose center moves on a circle of radius $11.25$ cm, so that
$D \approx 1.22$ and $\delta \approx 0.22$ (Figure~\ref{fig:radial}). The
corresponding nondimensional parameters are $(\mathsf{G},\mathsf{M}) = (2.2,\,0.0236)$
for $E = 200$ GPa and $(4.4,\,0.0472)$ for $E = 100$ GPa. As seen in
Figure~\ref{fig:theta_plot}, in either direction of rotation the wire stays nearly
vertical, snaps toward the magnet at
$|\alpha| \approx 25^\circ$,
tracks it through the vertical along the computed engaged branch, and releases at
$|\alpha| \approx 40^\circ$
on the opposite side. The measured switching angles and the measured loop width of
roughly $15^\circ$
lie within a few degrees of the fold angles $\alpha_{\rm lower} = 30.13^\circ$ and
$\alpha_{\rm upper} = 39.31^\circ$ computed for $E = 200$ GPa, the value closest to
the bulk modulus of nickel. The residual discrepancy is concentrated at the
transitions, where the response is most sensitive to imperfections in the experiment. Even at these transitions, as is evident from Fig.~\ref{fig:theta_plot}, the experimentally plotted lower folds open toward increasing~$\alpha,$ whereas the upper folds open toward decreasing~$\alpha,$ in agreement with Corollary~\ref{cor:intronum}.

\subsection{Review of Literature}\label{s.litreview}

The stability of slender elastic structures under body forces has a long history.  The
classical Greenhill problem considers the question of when a vertical elastic rod loses stability under its
own weight \cite{Greenhill1881,frisch1961analysis}.  The magnetoelastic buckling of slender
ferromagnetic structures has also been studied since the experiments of Moon and Pao
\cite{MoonPao1968}, who showed that a transverse magnetic field can destabilize a
ferromagnetic beam-plate.  \\

For bulk magnetoelastic materials,
DeSimone and James \cite{DeSimone2002} formulated a variational theory of
magnetoelasticity. Rigorous thin-wire dimension reduction limits of micromagnetics show that the
magnetization of a sufficiently slender soft ferromagnetic wire becomes effectively
one-dimensional and that the nonlocal demagnetization energy reduces to a local
shape-anisotropy term \cite{SS}.\\

For slender ferromagnetic bodies, the interaction between demagnetization and elasticity
has been studied both analytically and experimentally.  Gerbal et al.\
\cite{GerbalEtAl2015} revisited the classical magnetoelastic buckling experiment and
obtained quantitative agreement between a reduced theory and experiments on
ferromagnetic and superparamagnetic rods.  Singh and Onck \cite{SinghOnck2018}
derived slender-body descriptions retaining the demagnetizing field and used them to
study magnetic buckling and post-buckling of ferromagnetic beams.  More recently,
Avatar and Dabade \cite{AvatarDabade2024} studied large deformation of planar
ferromagnetic elastic ribbons.    \\

A closely related, but physically distinct, literature concerns hard-magnetic
elastomers, in which a prescribed remanent magnetization is embedded in a compliant
elastic matrix.  Reduced elastica and beam theories have been developed to describe
large magnetic deformations in this setting \cite{WangEtAl2020}, and Kirchhoff-type
rod theories have been derived for fully three-dimensional geometrically nonlinear
deformations and quantitatively compared with experiments
\cite{SanoPezzullaReis2022}.  Yan, Abbasi, and Reis
\cite{YanAbbasiReis2022} developed a complementary framework combining
one-dimensional reduction, three-dimensional simulation, and experiment for beams in
both uniform and nonuniform magnetic fields.  Magnetic coupling has also been used to
control classical elastic instabilities, such as the buckling strength
of shells \cite{YanEtAl2021}, and hard-magnetic bistable beams have been designed to
undergo remotely actuated snap-through between pre-existing stable configurations
\cite{AbbasiEtAl2023}.  These developments have important applications in remotely
actuated structures and soft robotics; see, for example,
\cite{kim2019ferromagnetic,ren2024design}.\\

The phenomena considered here are different in two respects.  First, our material is an
ordinary ferromagnetic wire rather than a hard-magnetic elastomer with a programmed
remanent magnetization.  We study two limiting magnetic regimes arising from the same
variational model: a strong uniform field which fixes the magnetization direction, and
a shape-anisotropy-dominated regime in which the magnetization follows the rod tangent.
Second, the two problems concern how magnetic coupling changes the \emph{stability
    landscape} of the elastica.  In the uniform-field problem, magnetism suppresses the
Greenhill instability and stabilizes gravity-opposed states which are unstable in the
purely elastic problem.  When instead driven by a permanent magnet, the conservative dipole field
instead creates saddle-node bifurcations and hysteretic switching in a rod which is not
pre-buckled or mechanically constrained to be bistable.

\subsection{Statement on AI use} AI was used for editing and proofreading purposes. Codex was also used to debug errors in the numerical root-finding procedures. \\

No generative AI tools were used in the development of the mathematical ideas, the
derivation of the theoretical results, or the proofs. The authors are fully responsible for the
content of the manuscript and for the accuracy, validity, and integrity of all results presented.

\section{Materials and methods}
\subsection{Variational model for a ferromagnetic Kirchhoff rod}
\label{sec:model}

We model the ferromagnetic wire as an inextensible and unshearable
Kirchhoff rod of length \(L\). There are many descriptive frames used to model such a rod, but we choose the ``director frame" approach which is particularly well-suited for describing the orientation of the rod. The dimensional centerline is denoted by
\[
    r:[0,L]\to\mathbb R^3,
\]
and the orientation of each cross-section is described by an orthonormal
director frame
\[
    R(s)
    :=
    \big(d_1(s),d_2(s),d_3(s)\big)
    \in SO(3).
\]
The arclength coordinate \(s\) is measured from the supported end of the
rod, so that \(s=0\) denotes the end with the clamp or pin and \(s=L\) the free end.
The inextensibility and unshearability constraints give
\[
    r'(s)=d_3(s).
\]

The Darboux vector \(\Omega\) is defined by the relation
\[
    d_i'(s)=\Omega(s)\times d_i(s),
    \qquad i=1,2,3.
\]
Writing
\[
    \Omega
    =
    \kappa_1 d_1+\kappa_2d_2+\tau d_3,
\]
the functions \(\kappa_1,\kappa_2\) are the two bending strains and
\(\tau\) is the twist strain.\\

Our primary assumption which simplifies the analysis is to suppose that the rod has circular cross section with radius \(b\llless L\), and thus is well-approximated by a one dimensional model, obtained from a 3D-1D dimension reduction. Now we describe each portion of the dimensionally reduced energy starting with the elastic energy.\\

We assume that the rod is made of a linear elastic material with Young's modulus \(E\) and shear modulus \(G\). The bending and torsional stiffnesses are then given by
\[
    B := EI,
    \qquad
    C := GJ,
\]
where \(I\) is the second moment of area and \(J\) is the torsional constant.
For a circular rod of radius \(b\), the cross-sectional area \(A\), second moment of area \(I\), and the torsional constant \(J\) are
\[
    A=\pi b^2,
    \qquad
    I=\frac{\pi b^4}{4}, \qquad J = 2I.
\]

Then the elastic energy is
\begin{equation}
    \label{eq:elastic_energy}
    E_{\mathrm{el}}[R]
    =
    \frac12
    \int_0^L
    \left[
        B(\kappa_1^2+\kappa_2^2)
        +
        C\tau^2
        \right]\,ds.
\end{equation}

Next we consider the gravitational potential energy of the rod, which accounts for the self-weight of the wire. We take \(e_3\) to be the upward vertical direction and let \(\rho\) denote
the mass density of material which we assume to be uniform. The gravitational potential energy is
\begin{equation}
    \label{eq:gravity_centerline}
    E_{\mathrm{grav}}[r]
    =
    \rho A g
    \int_0^L r(s)\cdot e_3\,ds.
\end{equation}

After fixing the origin of the coordinate system at the supported end of the rod,
\(r(0)=0\) and using the relation \(r'(s)=d_3(s)\), we can rewrite the gravitational energy in terms of the director frame by changing the order of integration in \eqref{eq:gravity_centerline},
\begin{equation}
    \label{eq:gravity_director}
    E_{\mathrm{grav}}[d_3]
    =
    \rho A g
    \int_0^L
    (L-s)d_3(s)\cdot e_3\,ds.
\end{equation}

Finally, we describe the magnetic energy of the rod. The magnetic state of the wire is described by the magnetization direction
\[
    m:[0,L]\to\mathbb S^2
\]
and take \(M_r\) to be the magnitude of the magnetization of the rod.
Two competing magnetic mechanisms are retained in the one-dimensional
model after doing a 3D-1D dimension reduction of the micromagnetic model.
First, the stray field term produces an effective
anisotropy term which favors alignment of the magnetization with the tangent
of the wire. We represent this by
\begin{equation}
    \label{eq:demag_energy}
    E_{\mathrm{demag}}[R,m]
    =
    \frac{K_dA}{2}
    \int_0^L
    \left[
        1-(m(s)\cdot d_3(s))^2
        \right]\,ds,
\end{equation}
where \(K_d>0\) is the demagnetization constant.\\

Second, let
\[
    H_{\mathrm{ext}}:\mathbb R^3\to\mathbb R^3
\]
be an externally imposed magnetic field. Its interaction with the
magnetization is described by the Zeeman energy
\begin{equation}
    \label{eq:zeeman_energy}
    E_{\mathrm{Z}}[r,m]
    =
    -\mu_0 A M_r
    \int_0^L
    m(s)\cdot H_{\mathrm{ext}}(r(s))\,ds.
\end{equation}

The total dimensional energy of the magnetic rod is therefore
\begin{multline}
    \label{eq:dimensional_master_energy}
    E[r,R,m]
    =
    \frac12
    \int_0^L
    \left[
        B(\kappa_1^2+\kappa_2^2)
        +
        C\tau^2
        \right]\,ds
    +
    \rho A g
    \int_0^L
    (L-s)d_3(s)\cdot e_3\,ds
    \\ +
    \frac{K_dA}{2}
    \int_0^L
    \left[
        1-(m(s)\cdot d_3(s))^2
        \right]\,ds
    -
    \mu_0 A M_r
    \int_0^L
    m(s)\cdot H_{\mathrm{ext}}(r(s))\,ds.
\end{multline}
We nondimensionalize the energy by rescaling the arc length \(s\) by the rod length \(L\) and the energy by the bending energy scale \(B/L\). We also nondimensionalize the external field by some scale \(H_0\). The resulting dimensionless energy (without relabeling the fields and variables) is
\begin{multline}
    \label{eq:nondimensional_master_energy}
    \mathcal{E} [r,R,m]
    =
    \frac12
    \int_0^1
    \left[
        \kappa_1^2+\kappa_2^2
        +
        \mathsf C \tau^2
        \right]\,ds
    +
    \mathsf G
    \int_0^1
    (1-s)d_3\cdot e_3\,ds
    \\+
    \frac{\mathsf A}{2}
    \int_0^1
    \left[
        1-(m\cdot d_3)^2
        \right]\,ds
    -
    \mathsf M
    \int_0^1
    m\cdot H_{\mathrm{ext}}(r(s))\,ds,
\end{multline}
where
\[
    \mathsf C:=\frac{C}{B},\qquad \mathsf G:=\frac{\rho A gL^3}{B},\qquad \mathsf A:=\frac{K_d A L^2}{B},\qquad \mathsf M:=\frac{\mu_0 A M_r H_0L^2}{B}.
\]

Finally, in this paper we will consider two different types of boundary conditions: clamped and pinned at the supported end. For a clamped rod we prescribe both the position and orientation at
\(s=0\):
\begin{equation}
    r(0)=0,
    \qquad
    d_i(0)=e_i,
    \qquad i=1,2,3.
    \label{eq:clamped_bc}
\end{equation}
The end \(s=L\) is free, and the corresponding boundary conditions arise
naturally from the first variation. For the pinned configuration, the endpoint position remains fixed,
\[
    r(0)=0,
\]
but the director frame is free to rotate at \(s=0\). The resulting natural
moment boundary condition models the ball-and-socket joint.\\

Equilibrium configurations are critical points of
\eqref{eq:dimensional_master_energy} subject to the appropriate boundary
conditions and the constraints
\[
    R(s)\in SO(3),
    \qquad
    r'(s)=d_3(s),
    \qquad
    m(s)\in\mathbb S^2.
\]

Existence of minimizers is standard due to the coercivity and lower semicontinuity of the energy. An important consequence of the equilibrium criterion from the first variation is that $\tau \equiv 0$. Then we can simplify the energy to only depend on $m$ and $d_3$ as
\begin{multline}
    \label{eq:nondimensional_master_energy_reduced}
    \mathcal E[m,d_3]
    =
    \frac12
    \int_0^1
    |d_3'|^2\,ds
    +
    \mathsf G
    \int_0^1
    (1-s)d_3\cdot e_3\,ds
    \\+
    \frac{\mathsf A}{2}
    \int_0^1
    \left[
        1-(m\cdot d_3)^2
        \right]\,ds
    -
    \mathsf M
    \int_0^1
    m\cdot H_{\mathrm{ext}}\left(\int_0^sd_3\right)\,ds,
\end{multline}
with the relevant boundary conditions and the constraints
\[d_3,m \in \mathbb S^2.\]
This serves as the starting point for our model in the different magnetic regimes under consideration.

\subsection{Magnetic regimes}
\label{sec:magnetic_regimes}

The main novelty of the paper is the inclusion of the magnetic terms in the energy functional. The two magnetic terms favor different alignments: the stray field term
drives \(m\) toward the rod tangent \(d_3\), whereas the Zeeman term drives
\(m\) toward the applied field. The two experimental regimes considered
in this paper correspond to the regime where one of these mechanisms
dominate and leads to situations in which theory, numerics, and experiments can harmonize.

\subsubsection{Uniform applied field}
The first experimental regime is characterized by a strong spatially uniform applied magnetic field. Since the external field is sufficiently strong to saturate the magnetization of the
elastica in the field direction, we assume \(m(s) \equiv m_0\). In this case, the Zeeman term reduces to a constant
and may be omitted from the energy. In the experiments, the field is directed along the
vertical axis, so that \(m_0:=e_3\).

\noindent The resulting reduced energy is
\begin{equation}
    \label{eq:uniform_field_energy}
    \mathcal E_{\mathrm{unif}}[d_3]
    =
    \frac12
    \int_0^1
    |d_3'|^2\,ds
    +
    \frac{\mathsf A}{2}
    \int_0^1
    \left[
        1-(e_3\cdot d_3)^2
        \right]\,ds
    +
    \mathsf G
    \int_0^1
    (1-s)d_3\cdot e_3\,ds.
\end{equation}
This is the model used to study magnetic stabilization of a rod
against self-weight in Section \ref{sec:unifrom_field}.

\paragraph{2.2.1.1 Experiments in uniform applied field.}

\noindent Experiments were conducted to examine the stability of a gravity-opposed vertical configuration of a ferromagnetic elastica subjected to a spatially uniform magnetic field for both clamped and pinned boundary conditions. The experiments were designed to test the theoretical predictions for the critical bifurcation length in the absence of a magnetic field and the magnetic field strength required to stabilize the elastica.\\

\noindent Nickel was chosen as the wire material owing to its ferromagnetic and elastic properties. No material parameter was measured experimentally; all values were instead taken from the literature. In particular,  Young's modulus was taken in the range $E = 100-200$ GPa, together with a mass density of $\rho = 9 \times10^{3}~\mathrm{kg/m^{3}}$, and a magnetic anisotropy constant of $K_d = 1.5 \times10^{5}~\mathrm{J/m^{3}}$, see \autoref{tab:material_para}. These values were employed in estimating the critical lengths and magnetic field strengths associated with the stability transitions. \\

\noindent A GMW Associates Model 5971-160 electromagnet with a 16 cm pole gap was used to generate a spatially uniform field magnetic field. Separate experiments were designed to realize the clamped and pinned boundary conditions considered in the theoretical analysis. The observed equilibrium configurations and stability behavior were then compared with the corresponding theoretical predictions. Further details of the experimental setups and procedures are provided in \Cref{sec:unifrom_field}.

\subsubsection{Conservative field from a permanent magnet}

We next consider the opposite regime in which the stray field term is
sufficiently strong that the magnetization remains locked to the tangent
of the wire. Motivated by the experiments, we assume
\[
    m(s)=d_3(s), \qquad H_{\mathrm{ext}}(x)= -\nabla\Phi(x)
\]
where $\Phi$ is the scalar potential of the permanent magnet generating the external field. We treat the magnet as a spherical dipole with radius \(a\) and strength \(M_s\), placed at a position \(Dp\), and so the dimensional form of \(H_{\mathrm{ext}}\) at a point satisfying \(|x-Dp|>a\) is
\[
    H_{\mathrm{ext}}(x):= \frac{M_sa^3}{3} \left( \frac{3 (p \cdot [x - Dp])(x-Dp)}{|x-Dp|^5}-\frac{p}{|x-Dp|^3}\right)
\]
Nondimensionalizing \(x\) by the length scale \(L\)  (without relabeling), we get:
\[
    H_{\mathrm{ext}}(x):= \frac{M_sa^3}{3L^3} \left( \frac{3 (p \cdot [x - \mathcal{D}p])(x-\mathcal{D}p)}{|x-\mathcal{D}p|^5}-\frac{p}{|x-\mathcal{D}p|^3}\right),
\]
where \(\mathcal{D} = \frac{D}{L}\).
The prefactor sets the external field scale \(H_0\) and plugging into the formula for \(\mathsf M\) we get
\[
    H_0 := \frac{M_sa^3}{3L^3} \implies \mathsf M:=\frac{\mu_0 A M_r H_0L^2}{B} = \frac{\mu_0 A M_r M_s a^3}{3BL}
\]
In this case,
the stray field term vanishes identically and the magnetic contribution
reduces to
\[
    \mathcal E_{\mathrm{mag}}
    =
    \mathsf M
    \int_0^1
    d_3(s)\cdot \nabla \Phi (r(s))\,ds.
\]
Using \(r'=d_3\), we can rewrite this as
\[
    \mathcal E_{\mathrm{mag}}=
    \mathsf M
    \left[
        \Phi(r(1))-\Phi(r(0))
        \right].
\]
Since \(r(0)=0\) is fixed, the second term is an additive constant.
Thus, in this regime, the magnetic interaction reduces exactly
to an endpoint potential depending only on the position of the rod tip and the full reduced energy is

\begin{equation}
    \label{eq:permanent_magnet_energy}
    \mathcal E_{\mathrm{perm}}[d_3]
    =
    \frac12
    \int_0^1
    |d_3'|^2\,ds
    +
    \mathsf G
    \int_0^1
    (1-s)d_3\cdot e_3\,ds
    +
    \mathsf M
    \Phi\left(\int_0^1 d_3\right).
\end{equation}
\[
    \Phi(x)
    =
    \frac{
        p\cdot(x-\mathcal D p)
    }{
        |x-\mathcal D p|^3
    },
    \qquad
    \mathcal D=\frac{D}{L}.
\]

\paragraph{2.2.2.1 Experiments with a permanent magnet.}

Experiments were performed using commercially available spherical NdFeB permanent magnets to investigate the deformation behavior of a clamped nickel wire under a localized magnetic actuation. A spherical magnet was moved along a circular arc centered at the clamped end of the wire, with its magnetization directed outward along the radial direction. This configuration produced a strongly non-uniform magnetic field around the wire. The nominal saturation magnetization was taken as $M_s = 8 \times 10^{5}\,\mathrm{A/m}$.The resulting wire deformations were recorded for different magnet positions, and the observed critical transitions were compared with the theoretical predictions and numerical simulations. Further details of the experimental setup and procedure are provided in \Cref{subsec:permanent_magnet}.

\paragraph{Numerical procedure for permanent magnet field.}

A numerical analysis was performed to predict the deformation of the clamped nickel wire under the non-uniform magnetic field generated by the permanent magnet and to compare the results with the experimental observations. In the numerical method, the wire deformation was restricted to the plane of the circular arc traced by the magnet, and the governing energy functional \eqref{eq:permanent_magnet_energy} was minimized using the Rayleigh--Ritz method.

\noindent Assuming a planar deformation,
\[
    d_3(s)=\sin\theta(s)\,e_1 + \cos\theta(s)\,e_3,\qquad  |d_3'(s)|^2=\theta'(s)^2, \qquad d_3.e_3=\cos{\theta(s)}.
\]
Also,
\[
    r(s)= \int_0^s d_3(\zeta) \,d\zeta= \left(\int_0^s\sin{\theta(\zeta)} \, d\zeta \right)\,e_1 + \left(\int_0^s\cos{\theta(\zeta)}\,d\zeta\right)\,e_3,
\]
The total energy becomes
\begin{equation}
    \label{eq:permanent_magnet_energy_reduced}
    \mathcal E_{\mathrm{perm}}[\theta(s)]
    =
    \frac12
    \int_0^1
    \theta(s)'^2\,ds
    +
    \mathsf G
    \int_0^1
    (1-s)\cos{\theta(s)}\,ds
    +
    \mathsf M
    \Phi\left(\int_0^1 d_3\,ds\right).
\end{equation}
The tangent angle $\theta(s)$ was approximated using a Rayleigh--Ritz expansion of the form,

\begin{equation*}
    \theta(s)
    =
    \sum_{n=1}^{N}
    A_n
    \sin
    \left(
    \frac{(2n-1)\pi s}{2}
    \right),
\end{equation*}
which automatically satisfies the clamped-free boundary conditions
\begin{equation*}
    \theta(0)=0,
    \qquad
    \theta'(1)=0.
\end{equation*}
Substituting the expansion into the energy functional converts the infinite-dimensional variational problem into a finite-dimensional optimization problem,
\[
    \mathcal E_{\mathrm{perm}}[\theta(s)] \longrightarrow \mathcal E_{\mathrm{perm}}[A_1,A_2,\ldots, A_N]
\]
The equilibrium configuration was obtained by solving
\begin{equation*}
    \min_{A_1,\ldots,A_N}
    \mathcal E_{\mathrm{perm}}[A_1,A_2,\ldots, A_N]
\end{equation*}

Once the optimal coefficients were determined, the corresponding equilibrium shape was reconstructed using $r(s)$. The minimization was performed in \textsc{Mathematica} using \texttt{FindMinimum} routine. A path-continuation strategy was employed, with the converged solution at each magnet position used as the initial guess for the next, allowing the stable equilibrium response to be followed efficiently. Near the folds, however, direct minimization of energy becomes ill-conditioned as the Hessian approaches singularity. A treatment of these points via Lyapunov–Schmidt reduction is discussed in \Cref{subsec:permanent_magnet}.

\begin{table}[ht]
    \centering
    \caption{Material parameters used in calculations.}
    \label{tab:material_para}
    \begin{tabular}{lcc}
        \hline
        \textbf{Parameter} & \textbf{Symbol} & \textbf{Value} \\
        \hline
        Young's modulus
        & $E$
        & $100\text{--}200~\mathrm{GPa}$ \\

        Density
        & $\rho$
        & $9\times10^3~\mathrm{kg/m^3}$ \\

        Demagnetizing energy density
        & $K_d$
        & $1.5\times10^5~\mathrm{J/m^3}$ \\

        Magnetic moment density (Ni)
        & $M_r$
        & $4.8\times10^5~\mathrm{A/m}$ \\

        Saturation magnetization (NdFeB)
        & $M_s$
        & $8\times10^5~\mathrm{A/m}$ \\

        \hline
    \end{tabular}
\end{table}

\section{Results}
\subsection{Uniform Magnetic Field} \label{sec:unifrom_field}
In this section, we study how to stabilize a straight magnetic rod against gravity. We first characterize the equilibria of the energy.
\begin{lemma}[First and Second Variations]
    \label{thm:uniform_variations}
    Let \(d_3\in H^1((0,1);\mathbb S^2)\), and let \(\eta\in H^1((0,1);\mathbb R^3)\) be a variation satisfying \(\eta\cdot d_3=0\) a.e. in \((0,1)\). In the case of clamped boundary conditions, we also assume \(\eta(0)=0\). Then the first variation of the energy \(\mathcal E_{\mathrm{unif}}\) is
    \begin{equation}
        \label{eq:uniform_first_variation}
        \delta\mathcal E_{\mathrm{unif}}[d_3](\eta)=\int_0^1 d_3'\cdot\eta'\,ds-\mathsf A\int_0^1(e_3\cdot d_3)(e_3\cdot\eta)\,ds+\mathsf G\int_0^1(1-s)e_3\cdot\eta\,ds.
    \end{equation}
    Furthermore, \(d_3\) is an equilibrium if and only if
    \begin{equation}
        \label{eq:uniform_EL_cross}
        d_3\times\left[-d_3''-\left(\mathsf A(e_3\cdot d_3)-\mathsf G(1-s)\right)e_3\right]=0.
    \end{equation}
    which is equivalent to there existing a scalar function \(\lambda\) such that
    \begin{equation}
        \label{eq:uniform_EL}
        -d_3''-\left(\mathsf A(e_3\cdot d_3)-\mathsf G(1-s)\right)e_3=\lambda d_3,
    \end{equation}
    where
    \begin{equation}
        \label{eq:uniform_lambda}
        \lambda=|d_3'|^2-\mathsf A(e_3\cdot d_3)^2+\mathsf G(1-s)e_3\cdot d_3.
    \end{equation}
    In the clamped case, we have the boundary cnditions \(d_3(0)=e_3\) and \(d_3'(1)=0\), while in the pinned case we have \(d_3'(0)=d_3'(1)=0\).

    At an equilibrium, the second variation in an admissible variation \(\eta\) is
    \begin{equation}
        \label{eq:uniform_second_variation}
        \delta^2\mathcal E_{\mathrm{unif}}[d_3](\eta)=\int_0^1\left[|\eta'|^2-\mathsf A(e_3\cdot\eta)^2-\lambda|\eta|^2\right]\,ds.
    \end{equation}
\end{lemma}
\begin{proof}
    This proof is a direct computation of the variations. \proofinappendix{uniform_variations}
\end{proof}

It is easy to see that the only straight line solution to the interior Euler-Lagrange equations are \(d_3 = \pm e_3\). We are interested in the gravity-opposed state, \(d_3 = e_3\) and we record the second variation in the following Corollary.

\begin{corollary}[Second variation at the gravity-opposed straight state]
    \label{cor:uniform_straight_second_variation}
    The gravity-opposed straight configuration
    \[
        d_3=e_3
    \]
    is a straight equilibrium of \(\mathcal E_{\mathrm{unif}}\) for every \(\mathsf A,\mathsf G\geq0\), for both the clamped and pinned boundary conditions. At this equilibrium,
    \[
        \lambda(s)=-\mathsf A+\mathsf G(1-s),
    \]
    and the second variation is
    \begin{equation}
        \label{eq:uniform_straight_second_variation}
        \delta^2\mathcal E_{\mathrm{unif}}[e_3](\eta)
        =
        \int_0^1\left[|\eta'|^2+\left(\mathsf A-\mathsf G(1-s)\right)|\eta|^2\right]\,ds
    \end{equation}
    for every admissible tangential variation \(\eta\) which satisfies \(\eta \cdot e_3 = 0\) and, in the clamped case, also \(\eta(0)=0\) .
\end{corollary}

First, we discuss the stability of this quadratic form in the case of no magnetic effects \(\mathsf A = 0\). In the clamped case, there is a classical theory due to Greenhill, which we recall in the following Lemma.
\begin{lemma}[Greenhill stability threshold \cite{frisch1961analysis, Greenhill1881}]
    \label{lem:greenhill}
    Consider the quadratic form
    \[
        Q_{\mathsf G}[u]:=\int_0^1\left[(u')^2-\mathsf G(1-s)u^2\right]\,ds
    \]
    on
    \[
        V :=\{u\in H^1((0,1)):u(0)=0\}.
    \]
    There exists a unique critical value \(\mathsf G_{\mathrm{crit}}>0\) such that
    \[
        Q_{\mathsf G}[u]>0
        \qquad\text{for every }u\in V\setminus\{0\}
    \]
    if and only if
    \[
        \mathsf G<\mathsf G_{\mathrm{crit}}.
    \]

    The critical value is characterized by
    \[
        \Ai(-\mathsf G_{\mathrm{crit}}^{1/3})
        +\frac{1}{\sqrt3}\Bi(-\mathsf G_{\mathrm{crit}}^{1/3})=0,
    \]
    and numerically
    \[
        \mathsf G_{\mathrm{crit}}\approx 7.8373.
    \]
\end{lemma}
This means that in the clamped case, the straight wire against gravity is stable for \(\mathsf G<\mathsf G_{\mathrm{crit}}\) and unstable for \(\mathsf G>\mathsf G_{\mathrm{crit}}\). In contrast, the pinned case is always unstable, as we record in the following Lemma. This can be proven by considering the constant mode \(u(s) = 1\) in the quadratic form \(Q_{\mathsf G}\).
\begin{lemma}[Pinned Stability]
    The quadratic form for \(u \in H^1((0,1))\)
    \[
        Q_{\mathsf G}[u]:=\int_0^1\left[(u')^2-\mathsf G(1-s)u^2\right]\,ds
    \]
    has a negative mode for every \(\mathsf G > 0\).
\end{lemma}

The experimental question is for which values of \(\mathsf A\) does a rod which is unstable without a magnetic field, become stable. Furthermore, we want to determine the critical length and characterize the bifurcation when the rod becomes unstable. There is an easy answer to this question in terms of eigenvalues of the operator which is summarized in the following Proposition.

\begin{proposition}[Stability of the gravity-opposed straight configuration]
    \label{thm:uniform_straight_stability}
    Define the space
    \[
        V :=\{u\in H^1((0,1)):u(0)=0\}
    \]
    and the principal eigenvalues
    \[
        \beta_0^{\mathrm{cl}}(\mathsf G)
        :=
        \inf_{u\in V\setminus\{0\}}
        \frac{\displaystyle\int_0^1\left[(u')^2-\mathsf G(1-s)u^2\right]\,ds}
        {\displaystyle\int_0^1u^2\,ds}.
    \]
    \[
        \beta_0^{\mathrm{p}}(\mathsf G)
        :=
        \inf_{u\in H^1((0,1))\setminus\{0\}}
        \frac{\displaystyle\int_0^1\left[(u')^2-\mathsf G(1-s)u^2\right]\,ds}
        {\displaystyle\int_0^1u^2\,ds}.
    \]
    Then the gravity-opposed straight configuration \(d_3=e_3\) is strictly stable under the boundary condition \(\mathrm{bc} \in {\{\mathrm{cl},\mathrm{p}}\}\) if and only if
    \[
        \mathsf A+\beta_0^{\mathrm{bc}}(\mathsf G)>0.
    \]
\end{proposition}
\begin{proof}
    The proof is direct by definition of the principal eigenvalue and the second variation formula in \Cref{cor:uniform_straight_second_variation}. \proofinappendix{uniform_straight_stability}
\end{proof}

Since the principal application of this type of result would be to understand when the wire becomes stabilized, it is useful to get algebraic bounds for the principal eigenvalues which we can calculate from the relevant experimental parameters.

\begin{proposition}[Explicit bounds for the principal eigenvalues]
    \label{thm:uniform_eigenvalue_bounds}
    Let \(\beta_0^{\mathrm p}(\mathsf G)\) and \(\beta_0^{\mathrm{cl}}(\mathsf G)\) denote the principal eigenvalues defined in \Cref{thm:uniform_straight_stability}. Let \(a_1<0\) and \(a_1'<0\) denote the largest zeros of \(\Ai\) (Airy function) and \(\Ai'\) (derivative of the Airy function), respectively, so that
    \[
        |a_1|\approx2.3381,\qquad |a_1'|\approx1.0188.
    \]
    Then, for every \(\mathsf G>0\), the pinned eigenvalue satisfies
    \[
        -\mathsf G\leq\beta_0^{\mathrm p}(\mathsf G)\leq\min\left\{-\frac{\mathsf G}{2},-\mathsf G+|a_1'|\mathsf G^{2/3}\right\}.
    \]
    If we assume further that \(\mathsf G>\mathsf G_{\mathrm{crit}}\), then
    \[
        -\mathsf G+c_{\mathrm p}\mathsf G^{2/3}\leq\beta_0^{\mathrm p}(\mathsf G)<-\mathsf G+|a_1'|\mathsf G^{2/3},
    \]
    for
    \[
        c_{\mathrm p}:=\frac{\beta_0^{\mathrm p}(\mathsf G_{\mathrm{crit}})+\mathsf G_{\mathrm{crit}}}{\mathsf G_{\mathrm{crit}}^{2/3}}\approx0.8686.
    \]

    The clamped eigenvalue satisfies for \(\mathsf G>\mathsf G_{\mathrm{crit}}\)
    \[
        -\mathsf G+\mathsf G_{\mathrm{crit}}^{1/3}\mathsf G^{2/3}\leq\beta_0^{\mathrm{cl}}(\mathsf G)<-\mathsf G+|a_1|\mathsf G^{2/3}.
    \]
\end{proposition}
\begin{proof}
    The main idea of the proof is to transform the eigenvalues by
    \[
        \frac{\beta_0^{\mathrm{bc}}(\mathsf G)+\mathsf G}{\mathsf G^{2/3}},
    \]
    and write it as a new eigenvalue problem \(\nu_{\mathrm{bc}}\). The upper bound comes from testing with a shifting of the Airy function. The lower bound comes from monotonicity of \(\nu_{\mathrm{bc}}\) with respect to \(\mathsf G\).
    \proofinappendix{uniform_eigenvalue_bounds}
\end{proof}

We can also characterize the bifurcation as a pitchfork bifurcation in the following Proposition.

\begin{proposition}[Character of the first bifurcation]
    \label{thm:uniform_bifurcation_criticality}
    Assume \(\mathsf G>0\) in the pinned case and \(\mathsf G>\mathsf G_{\mathrm{crit}}\) in the clamped case. Set \(\mathsf A_{\mathrm{crit}}:=-\beta_0^{\mathrm{bc}}(\mathsf G)>0\), and let \(u\) be the principal eigenfunction, positive on \((0,1)\) and normalized by \(\int_0^1u^2\,ds=1\), satisfying
    \begin{equation}\label{eq:uniform_critical_eigenfunction}
        -u''+\bigl(\mathsf A_{\mathrm{crit}}-\mathsf G(1-s)\bigr)u=0,
    \end{equation}
    with the appropriate boundary conditions.

    There exists a smooth even function \(\mathsf A(r)\), defined for \(r\) near zero, such that for every unit vector \(p\in\operatorname{span}\{e_1,e_2\}\) there is a smooth family of equilibria satisfying
    \begin{equation}\label{eq:uniform_bifurcation_branch}
        \mathsf A(r)=\mathsf A_{\mathrm{crit}}+\mathsf A_2r^2+o(r^2),
    \end{equation}
    where
    \begin{equation}\label{eq:uniform_bifurcation_coefficient-prop}
        \mathsf A_2
        =\frac16\int_0^1\bigl(4\mathsf A_{\mathrm{crit}}-\mathsf G(1-s)\bigr)u^4\,ds
        =\frac12\left(\mathsf A_{\mathrm{crit}}\int_0^1u^4\,ds-\int_0^1u^2(u')^2\,ds\right).
    \end{equation}

    In the pinned case, \(\mathsf A_2>0\) for every \(\mathsf G>0\). In the clamped case, \(\mathsf A_2<0\) for \(\mathsf G>\mathsf G_{\mathrm{crit}}\) sufficiently close to \(\mathsf G_{\mathrm{crit}}\), but \(\mathsf A_2>0\) for sufficiently large \(\mathsf G\). A sufficient condition for \(\mathsf A_2>0\) is \(\beta_0^{\mathrm{cl}}(\mathsf G) \leq -\frac{\mathsf G}{4}\)
\end{proposition}
\begin{proof}
    The proof is by a Lyupanov-Schmidt reduction and using the eigenvalue equation to simplify the sign of \(\mathsf A_2\). \proofinappendix{uniform_bifurcation_criticality}
\end{proof}

We summarize the previous results into a theorem for each type of boundary condition.
\begin{theorem}[Physical Meaning for Clamped Rods]
    \label{thm:uniform_physical_clamped}
    Assume \(\mathsf G>0\) and \(\mathsf A \geq 0\). Assume the rod is held with the clamped boundary condition. Then, the stability of the straight rod held against gravity \(d_3 = e_3\) depends on values of \(\mathsf G, \mathsf A\) in the following way:
    \begin{enumerate}
        \item If \(\mathsf G<\mathsf G_{\mathrm{crit}} \approx 7.8373\), then the rod is strictly stable for \(\mathsf A \geq 0\).
        \item If \(\mathsf G\geq \mathsf G_{\mathrm{crit}} \), then the rod is only strictly stable if \(\mathsf A > -\beta_0^{\mathrm{cl}}(\mathsf G)\) as defined in \Cref{thm:uniform_straight_stability}.
        \item For any~$\mathsf{G}>0,$ a sufficient condition for stability is
              \[\mathsf A > \mathsf G-1.9864\,\mathsf G^{2/3}\,,\]
              and a sufficient condition for instability is
              \[\mathsf A < \mathsf G-2.3381\,\mathsf G^{2/3}\,. \]
    \end{enumerate}
    Furthermore, the character of the bifurcation changes from supercritical pitchfork to subcritical pitchfork depending on the value of \(\beta_0^{\mathrm{cl}}\). If \(\beta_0^{\mathrm{cl}}(\mathsf G_{\mathrm{crit}}) \sim 0\), the character is of \emph{supercritical} pitchfork type. When \(\beta_0^{\mathrm{cl}}(\mathsf G)< -\frac{\mathsf G}{4}\), the bifurcation switches to a \emph{subcritical pitchfork} type.
\end{theorem}

For pinned rods, we have the following theorem:

\begin{theorem}[Physical Meaning for Pinned Rods]
    \label{thm:uniform_physical_pinned}
    Assume \(\mathsf G>0\) and \(\mathsf A \geq 0\). Assume the rod is held with the pinned boundary condition. Then, the stability of the straight rod held against gravity \(d_3 = e_3\)  depends on values of \(\mathsf G, \mathsf A\) in the following way:
    \begin{enumerate}
        \item The rod is only strictly stable if \(\mathsf A > -\beta_0^{\mathrm{p}}(\mathsf G)\) as defined in \Cref{thm:uniform_straight_stability}.
        \item If \(\mathsf G > \mathsf G_{\mathrm{crit}}\), a sufficient condition for strict stability is
              \[\mathsf A > \mathsf G-0.8686\,\mathsf G^{2/3}\]
              and a sufficient condition for instability is
              \[\mathsf A < \mathsf G-1.0188\,\mathsf G^{2/3}\]
        \item If \(\mathsf G < \mathsf G_{\mathrm{crit}}\),
              a sufficient condition for strict stability is
              \[\mathsf A > \mathsf G\]
              and a sufficient condition for instability is
              \[\mathsf A < \max\left\{\frac{\mathsf G}{2},\mathsf G-1.0188\,\mathsf G^{2/3}\right\}\]
    \end{enumerate}
    The character of the bifurcation is always a \emph{subcritical pitchfork}.
\end{theorem}

A simple corollary of this is that we can get a sufficient condition for stability on the length of the wire with either boundary condition that is independent of the wire radius \(b\) and Young's modulus \(E\). This is given in the following Corollary.
\begin{corollary}[Magnetic Stabilization]
    \label{thm:stabilization}
    For a circular wire as in the experiment,
    using
    \[ \mathsf G=\frac{4\rho gL^3}{Eb^2}, \qquad \mathsf A=\frac{4K_dL^2}{Eb^2}, \]
    a sufficient condition for strict stability becomes
    \[ L<L_{\mathrm{stab}}:=\frac{K_d}{\rho g}. \]
    In particular, \(L_{\mathrm{stab}}\) is independent of both the wire radius \(b\) and Young's modulus \(E\).

    In the clamped case, an explicit sufficient condition for \emph{subcriticality} of the pitchfork bifurcation at the onset of instability is
    \[
        L_{\mathrm{stab}}>\frac{|a_1|}{3}
        \left(\frac{Eb^2}{4\rho g}\right)^{1/3}.
    \]
\end{corollary}

\begin{proof}
    This comes directly from the algebraic bounds on \(\beta_0\). \proofinappendix{stabilization}
\end{proof}

\subsubsection{Experimental results for Clamped Boundary Condition (Uniform Field) and Comparison with Theory}

The clamped boundary condition was realized using 3D-printed L-shaped brackets that rigidly fixed one end of the wire. Ferromagnetic wires of diameter $d=50~\mu$m and lengths $L=8$, $10$, and $12$ cm were employed. \\

The theoretical analysis predicts that the gravity-opposed straight configuration is stable in the absence of a magnetic field for $\mathsf{G}<\mathsf{G}_{\mathrm{crit}}\approx7.8373$, and unstable for $\mathsf{G}>\mathsf{G}_{\mathrm{crit}}$. Furthermore, for $\mathsf{G}>\mathsf{G}_{\mathrm{crit}}$, the theory predicts that stability can be recovered through the application of a magnetic field, provided that
$\mathsf A > \mathsf G-0.8686\,\mathsf G^{2/3}$, see \autoref{thm:uniform_physical_clamped}. For the material and geometric parameters employed in the experiments, the condition $\mathsf{G}_{\mathrm{crit}}=7.8373$ corresponds to a critical wire length of approximately $L_c = 11.15-14.05$ cm, obtained by considering the extreme values of $E=100-200$ GPa. The inequality $\mathsf A > \mathsf G-0.8686\,\mathsf G^{2/3}$ is satisfied for all three lengths considered, see \autoref{tab:clamped}.

\begin{table}[ht]
    \centering
    \caption{Parameters for the clamped boundary condition}
    \label{tab:clamped}
    \begin{tabular}{|c|c|c|c|c|}
        \hline
        $E$ & $L$ & $\mathsf{G}$ & $\mathsf{A}$ &
        $\mathsf{A}-\mathsf{G}+0.8686\,\mathsf{G}^{2/3}$ \\
        \hline
        \multirow{3}{*}{$100~\mathrm{GPa}$}
        & $8~\mathrm{cm}$  & 2.893 & 61.44  & 62.58  \\
        & $10~\mathrm{cm}$ & 5.650 & 96.00  & 96.651 \\
        & $12~\mathrm{cm}$ & 9.764 & 138.24 & 137.55 \\
        \hline
        \multirow{3}{*}{$200~\mathrm{GPa}$}
        & $8~\mathrm{cm}$  & 1.446 & 30.72 & 31.814 \\
        & $10~\mathrm{cm}$ & 2.825 & 48.00 & 49.144 \\
        & $12~\mathrm{cm}$ & 4.882 & 69.12 & 69.954 \\
        \hline
    \end{tabular}
\end{table}

The wires of lengths $L=8$, $10$, and $12$ cm were first suspended vertically with the free end directed downward, corresponding to the stable equilibrium configuration under gravity, see Figure~\ref{fig:clamped1}\subref{fig:Clamped_along_g}. The assembly was then inverted such that the clamped end was below the free end, placing the elastica in the gravity-opposed configuration. In the absence of an external magnetic field, the wires of lengths $8$ and $10$ cm remained vertically straight and stable, whereas the $12$ cm wire exhibited a clear loss of stability and deflected from the vertical configuration, see Figure~\ref{fig:clamped1}\subref{fig:Clamped_against_g}. The onset of instability in the $12$ cm specimen, together with the stability of the shorter specimens, is consistent with the theoretically predicted bifurcation length  $L_c=11.15-14.05$ cm.\\

\begin{figure}[!hbt]
    \centering

    \begin{subfigure}{0.48\textwidth}
        \centering
        \includegraphics[width=\linewidth]{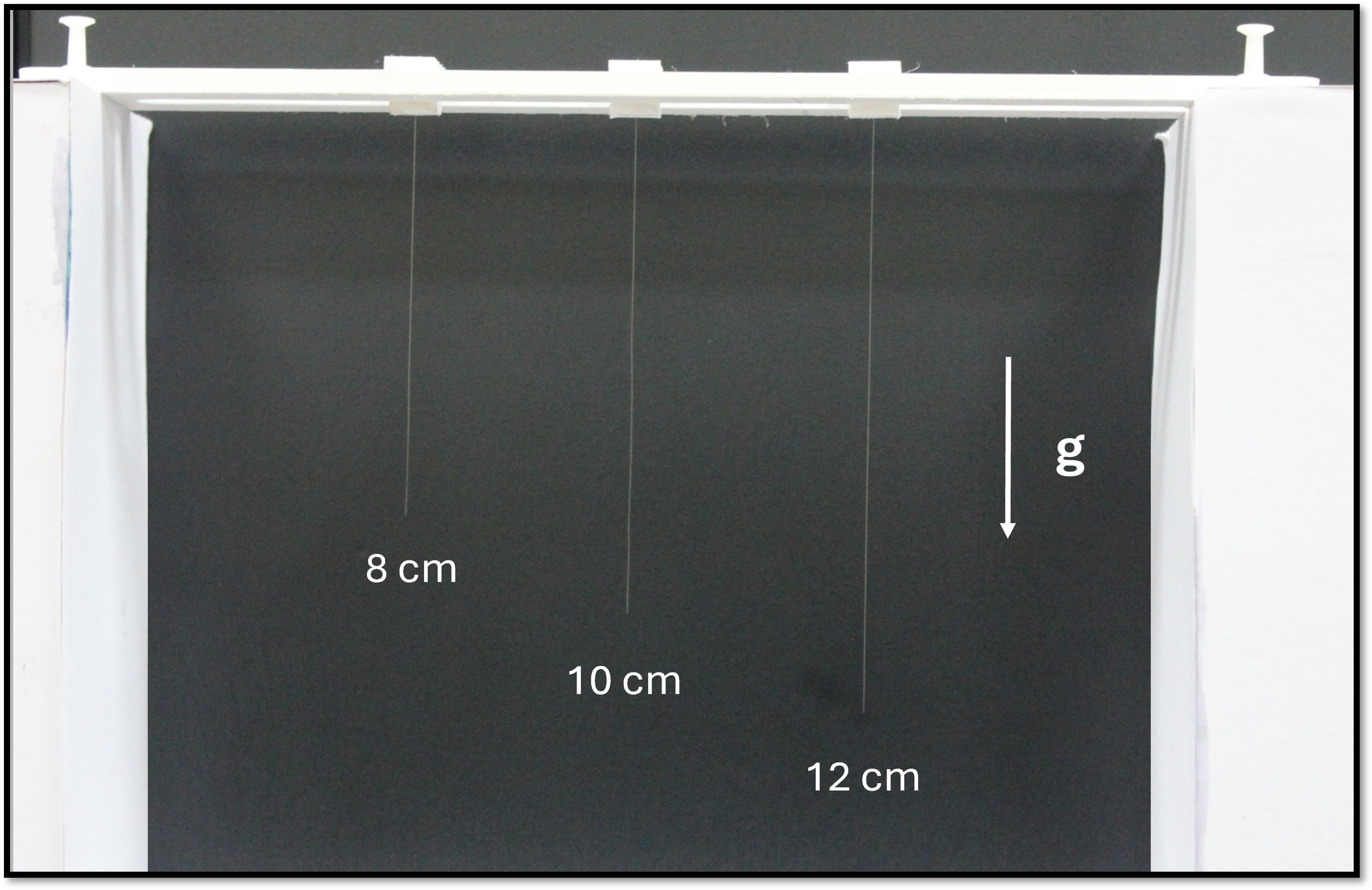}
        \caption{}
        \label{fig:Clamped_along_g}
    \end{subfigure}
    \hfill
    \begin{subfigure}{0.48\textwidth}
        \centering
        \includegraphics[width=\linewidth]{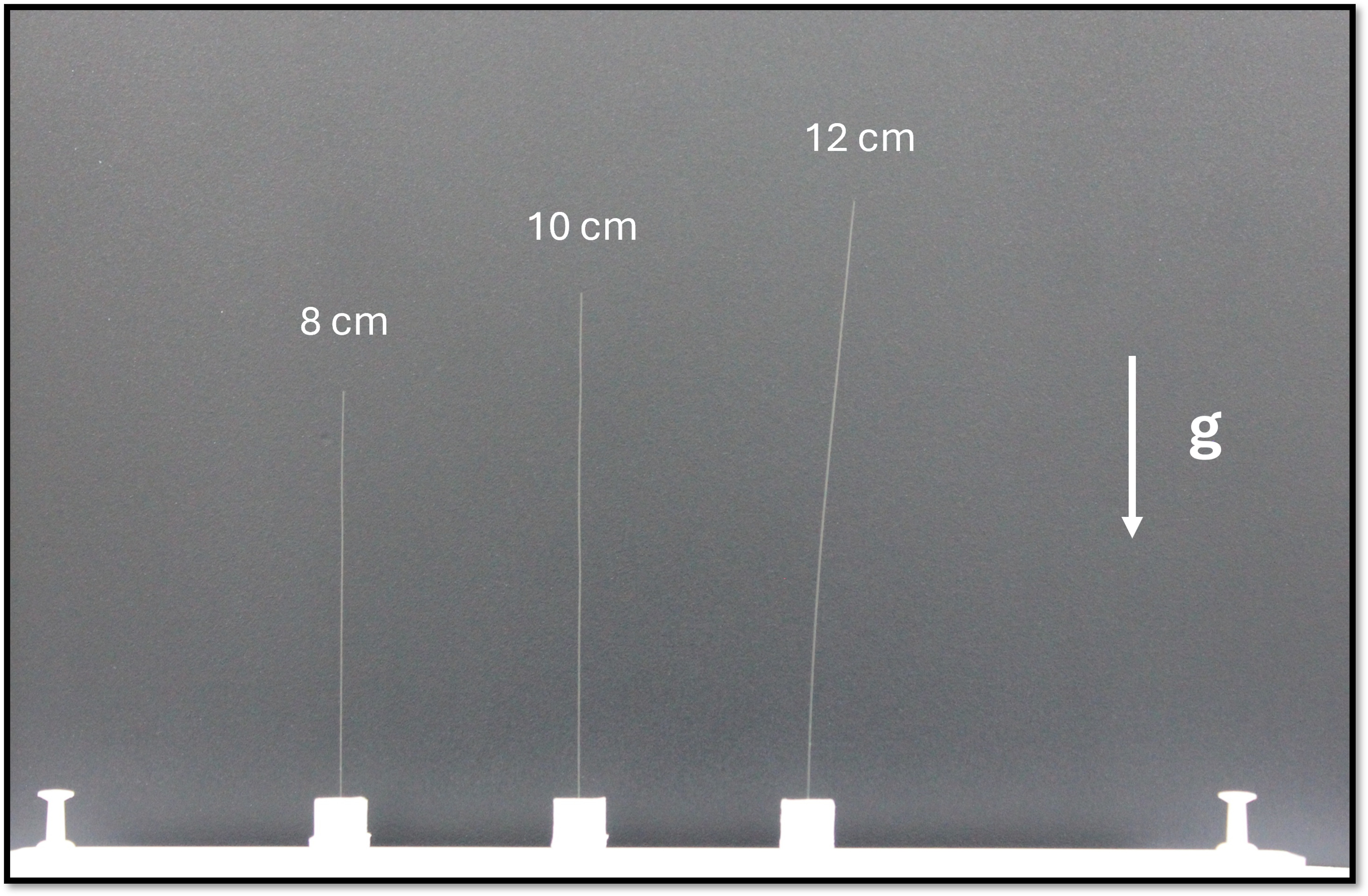}
        \caption{}
        \label{fig:Clamped_against_g}
    \end{subfigure}

    \caption{Clamped nickel wires of length $L=8$, $10$, and $12$ cm in the absence of a magnetic field a) suspended in gravity b) standing against gravity. The wire of length $L=12$ cm loses stability in the gravity-opposed configuration, whereas the shorter specimens remain stable.}
    \label{fig:clamped1}
\end{figure}

A uniform magnetic field of $30$ mT was subsequently applied parallel to the undeformed axis of the wires. Under the applied field, all three specimens recovered a stable vertical configuration. In particular, the $12$ cm wire, which was unstable in the absence of the field, was observed to remain straight and stable once the magnetic field was applied, see Figure~\ref{fig:clamped2}\subref{fig:Clamped_H_not_0}. The experiments therefore confirm the stabilizing influence of the magnetic field predicted by the theoretical model.\\

\begin{figure}[hbt]
    \centering

    \begin{subfigure}{0.48\textwidth}
        \centering
        \includegraphics[width=\linewidth]{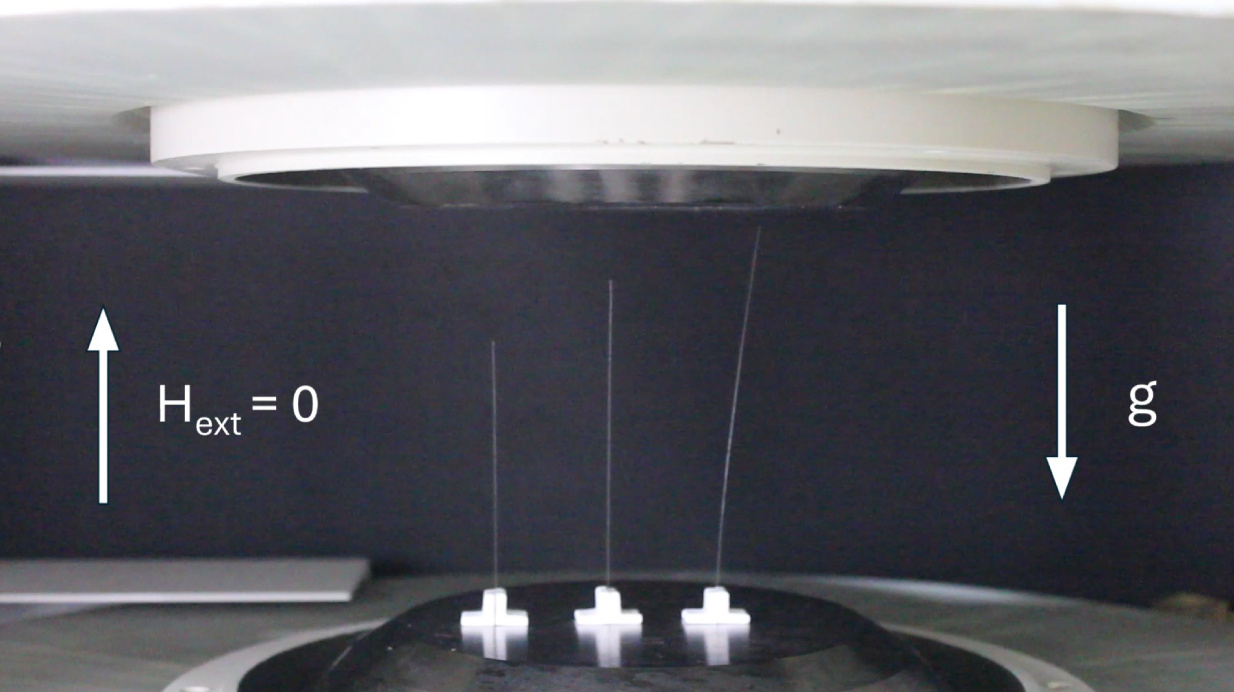}
        \caption{$H_{ext}=0$ T}
        \label{fig:Clamped_H_0}
    \end{subfigure}
    \hfill
    \begin{subfigure}{0.48\textwidth}
        \centering
        \includegraphics[width=\linewidth]{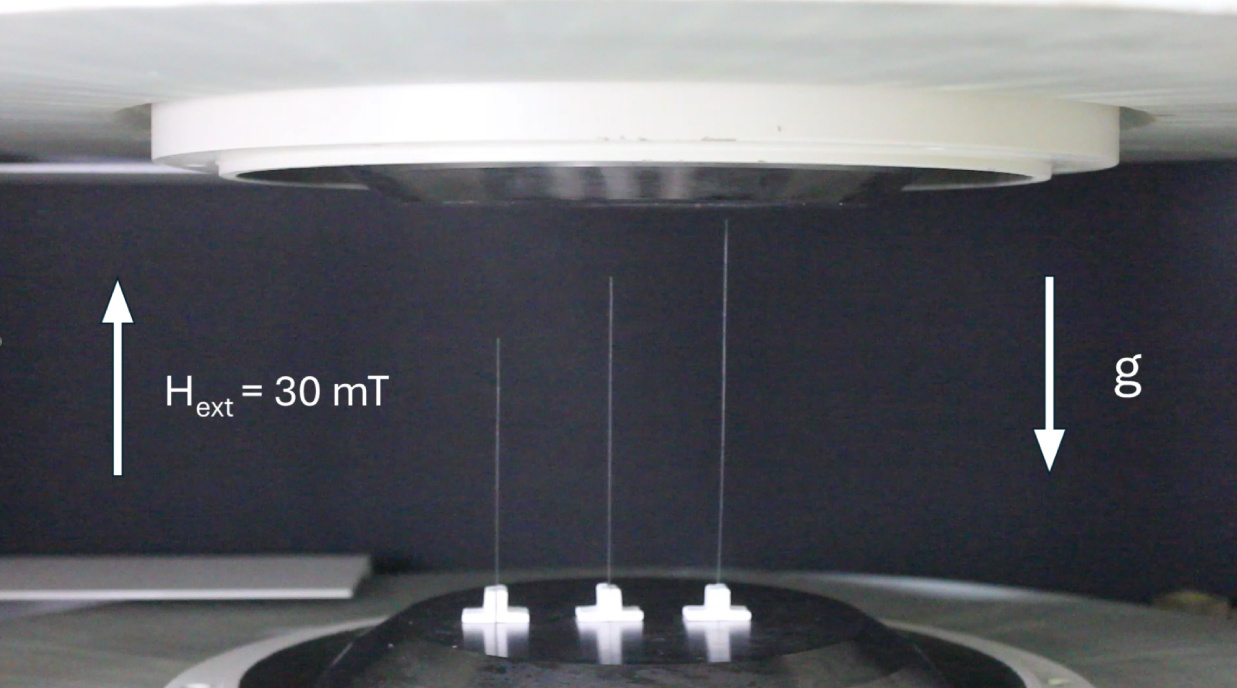}
        \caption{$H_{ext}=30$ mT}
        \label{fig:Clamped_H_not_0}
    \end{subfigure}

    \caption{Clamped nickel wires of lengths $L=8$, $10$, and $12$ cm: (a) in the absence of an applied magnetic field ($H_{\mathrm{ext}}=0$) and (b) in a uniform magnetic field of strength $H_{\mathrm{ext}}=30$ mT. The wire of length $L=12$ cm, which is unstable in the gravity-opposed configuration when $H_{\mathrm{ext}}=0$, is stabilized by the applied magnetic field.}
    \label{fig:clamped2}
\end{figure}

\subsubsection{Experimental Results for Pinned Boundary Condition (Uniform Field) and Comparison with Theory}

To approximate a pinned boundary condition, a low-friction floating support was constructed using a hollow spherical float (a ping pong ball) placed in a cylindrical water-filled cap. The inner diameter of the container was chosen to be only marginally larger than the diameter of the sphere, thereby restricting lateral translation while permitting nearly free rotation about the contact point. A wire was inserted into the sphere and aligned along a radial direction, such that rotations of the sphere resulted in corresponding rotations of the wire at its base. This arrangement effectively constrained translational motion while allowing rotational freedom, closely approximating the kinematic characteristics of an ideal pin joint.

\begin{figure}[!hbt]
    \centering

    \begin{subfigure}{0.48\textwidth}
        \centering
        \includegraphics[width=\linewidth]{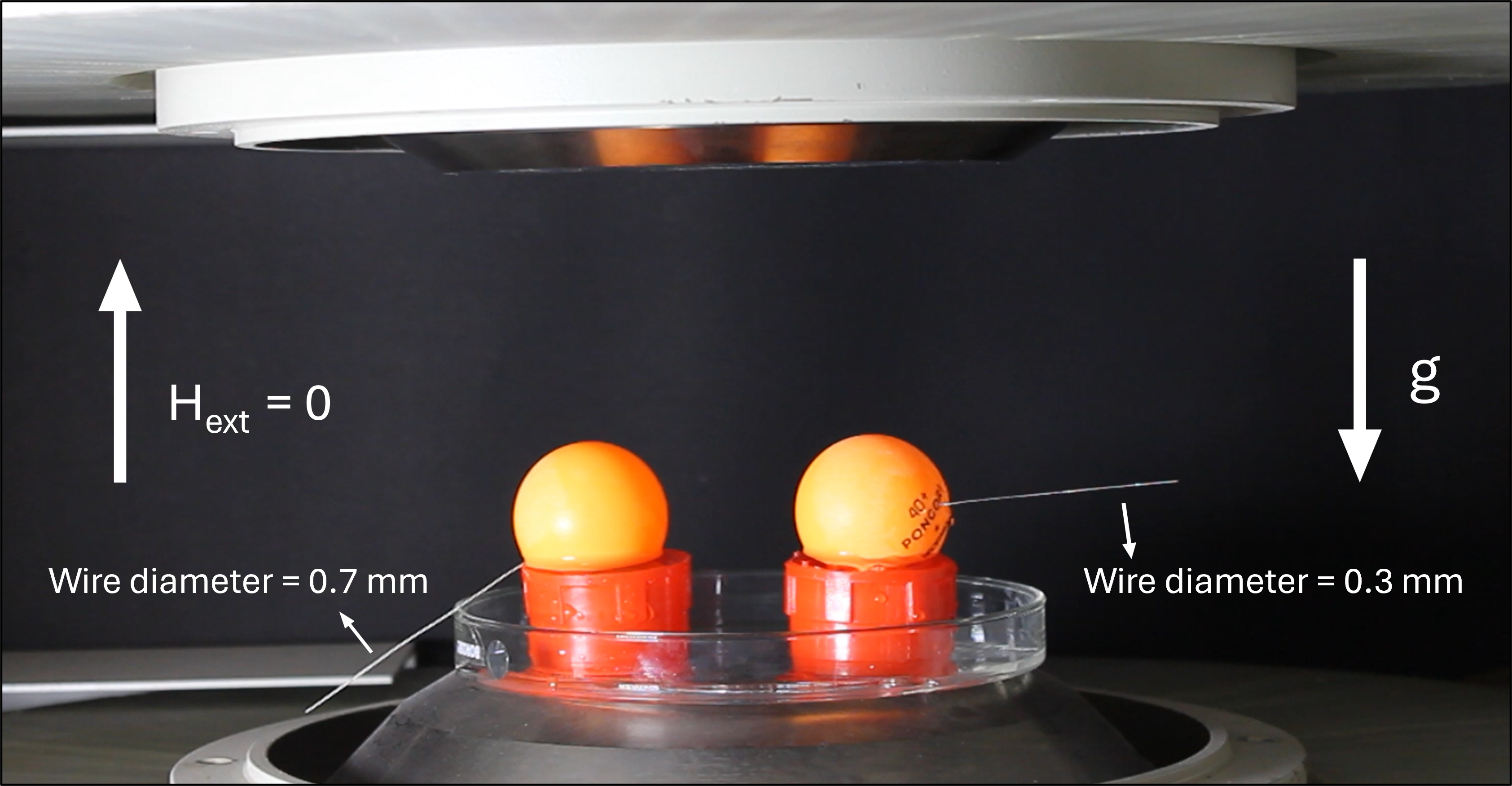}
        \caption{$H_{ext}=0$}
        \label{fig:pinned_0_H}
    \end{subfigure}
    \hfill
    \begin{subfigure}{0.48\textwidth}
        \centering
        \includegraphics[width=\linewidth]{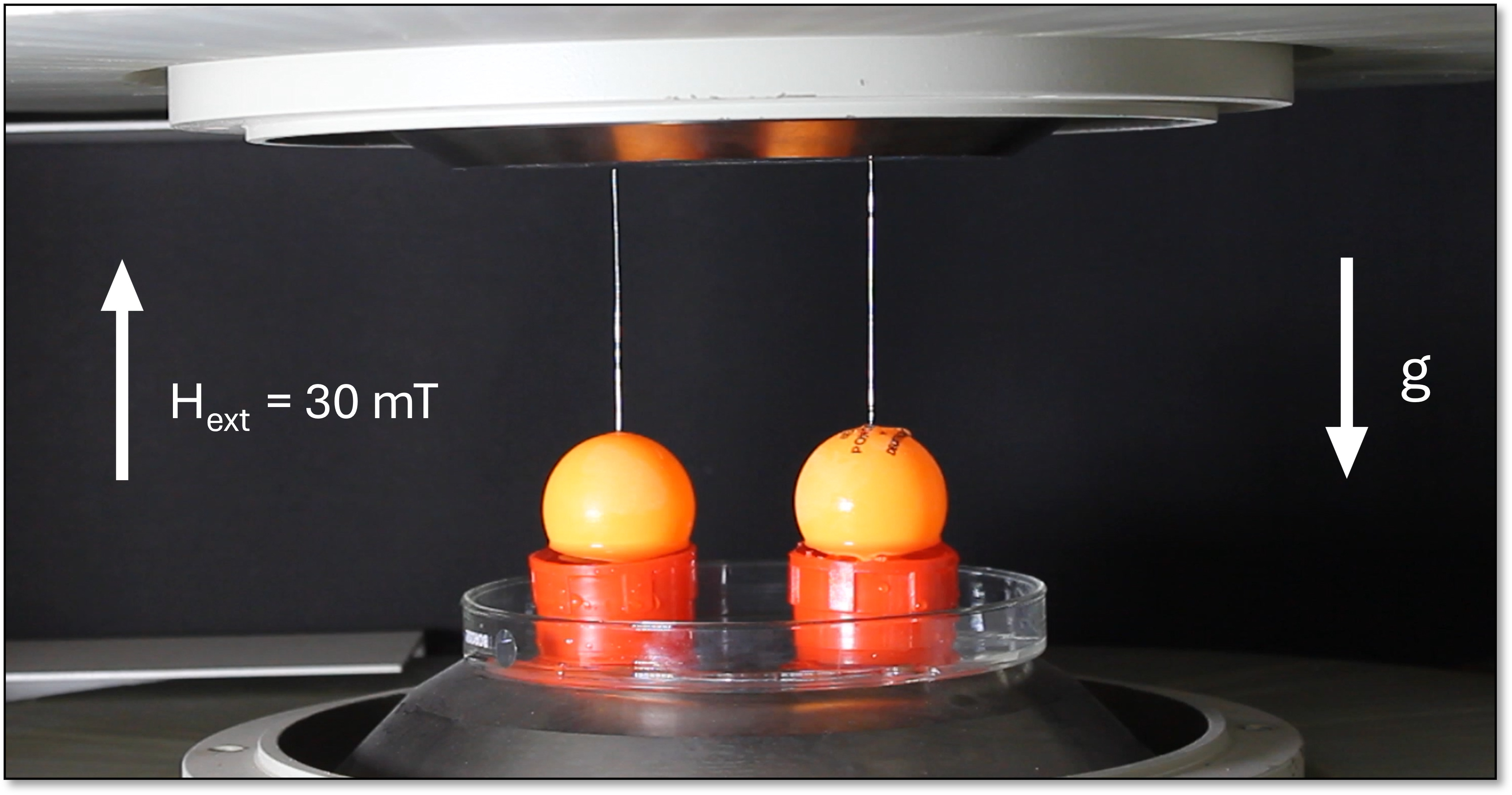}
        \caption{$H_{ext}=30$ mT}
        \label{fig:pinned_H_not_0}
    \end{subfigure}

    \caption{Pinned nickel wires of diameters $0.3\,\mathrm{mm}$ and $0.7\,\mathrm{mm}$ in the gravity-opposed configuration under an applied magnetic field. The vertical configuration for both the wires is  a) unstable when $H_{ext}=0$ T and b) stable when $H_{ext}=30$ mT.}
    \label{fig:pinned}
\end{figure}

The theoretical analysis predicts that the gravity-opposed straight configuration is always unstable in the absence of a magnetic field, see \autoref{thm:uniform_physical_pinned}. In the presence of a magnetic field, however, the condition $\mathsf{A}-\mathsf{G}>0$ provides a sufficient criterion for stability. For the material and loading parameters considered here, this criterion guarantees stability for all wire lengths below $L_{\mathrm{stab}}=1.70$~m, irrespective of the wire diameter, see \Cref{thm:stabilization}. This diameter independence allows the prediction to be tested using wires of different diameters. Accordingly, experiments were performed using nickel wires with diameters of $0.3$~mm and $0.7$~mm.

\begin{table}[ht]
    \centering
    \caption{Parameters for the pinned boundary condition}
    \label{tab:pinned}
    \begin{tabular}{|c|c|c|c|c|}
        \hline
        $E$ & $d$ & $\mathsf{G}$ & $\mathsf{A}$ &
        $\mathsf{A-G}$ \\
        \hline
        \multirow{2}{*}{$100~\mathrm{GPa}$}
        & $0.3~\mathrm{mm}$ & 0.209 & 3.226 & 3.017  \\
        & $0.7~\mathrm{mm}$ & 0.038 & 0.592 & 0.554  \\
        \hline
        \multirow{2}{*}{$200~\mathrm{GPa}$}
        & $0.3~\mathrm{mm}$ & 0.104 & 1.613 & 1.509  \\
        & $0.7~\mathrm{mm}$ & 0.019 & 0.296 & 0.277  \\
        \hline
    \end{tabular}
\end{table}

Two wires of the same length, $L=11$~cm, with diameters of $0.3$~mm and $0.7$~mm, respectively, were used to test these predictions. Consistent with the theoretical analysis, the gravity-opposed straight configurations were observed to be unstable in the absence of a magnetic field, see Figure~\ref{fig:pinned}\subref{fig:pinned_0_H}. Upon application of a uniform magnetic field of strength $30$~mT, both wires stabilized and remained in the vertically straight configuration, see Figure~\ref{fig:pinned}\subref{fig:pinned_H_not_0}. The corresponding values of $\mathsf{G}$, $\mathsf{A}$, and $\mathsf{A}-\mathsf{G}$ are listed in \autoref{tab:pinned}. The observed stabilization is consistent with the theoretical sufficient stability criterion, which predicts stability of the gravity-opposed configuration for wire lengths below $L_{\mathrm{stab}}=1.70$~m.

\subsection{Permanent Magnet on Sphere}
\label{subsec:permanent_magnet}

In this experiment, the magnet center is constrained to the circle \(\mathcal D p(\alpha)\) with $\mathcal D>1$ fixed, where
\[
    p(\alpha)=\sin\alpha\,e_1+\cos\alpha\,e_3.
\]

We assume the permanent magnet may be moved from the position \(\alpha = - \pi/2\) to \(\alpha = \pi/2\) along the upper semicircle. The experimental results show that the equilibrium configuration of the rod is planar until a sudden snap to a nearly straight configuration. Thus, in this section we study the existence and stability of planar equilibria. We characterize the snap as saddle-node bifurcations. Since computing these points directly by solving the full problem is numerically ill-conditioned, we implement a version of Lyapunov-Schmidt reduction to reduce the problem of finding the angles at which the bifurcation occurs to a well-conditioned 2D root finding problem. This section implements this program.\\

Our first result characterizes the relationship between the planar critical points and the 3D equilibrium configurations.

\begin{lemma}[Existence of planar equilibria]
    \label{thm:planar_equilibria}
    There exists \(\theta\in H^1((0,1))\) such that
    \begin{equation}\label{eq:planar_EL}
        -\theta''-\mathsf G(1-s)\sin\theta+\mathsf M\nabla\Phi(X)\cdot\left(\cos\theta\,e_1-\sin\theta\,e_3\right)=0,
        \qquad
        \theta(0)=0,\quad \theta'(1)=0,
    \end{equation}
    where
    \[
        X=\int_0^1\left(\sin\theta\,e_1+\cos\theta\,e_3\right)\,ds.
    \]
    Every solution of \eqref{eq:planar_EL} defines an equilibrium of the full three-dimensional Euler-Lagrange system \eqref{eq:permanent_EL_system}.
    \begin{equation}
        \label{eq:permanent_EL_system}
        \begin{cases}
            d_3\times\left[-d_3''+\mathsf G(1-s)e_3+\mathsf M\nabla\Phi\left(\displaystyle\int_0^1d_3\,ds\right)\right]=0, & s\in(0,1), \\
            d_3(0)=e_3,                                                                                                                 \\
            d_3'(1)=0.
        \end{cases}
    \end{equation}
\end{lemma}
\begin{proof}
    By the direct method of calculus of variations and the computation of the planar Euler-Lagrange system, we can show that the existence of a planar solution which corresponds to solution of\eqref{eq:planar_EL}. In order to show that this planar equilibrium is a solution for the full Euler-Lagrange system, we show that the bracketed term in \eqref{eq:permanent_EL_system} is parallel to the tangent vector \(d_3\).
    \proofinappendix{planar_equilibria}
\end{proof}

To help analyze these planar solutions, we implement a decomposition in the spirit of Lyapunov-Schmidt reduction. We orthogonally decompose the planar equilibrium into a mean zero component and a one parameter family of solutions which encode the deformation. This decomposition is recorded in the following Lemma.

\begin{lemma}[Mode decomposition]
    \label{lem:mode_decomposition}
    Let
    \[
        V:=\{v\in H^1((0,1)):v(0)=0\},
        \qquad
        Y:=\left\{v\in V:\int_0^1v\,ds=0\right\},
        \qquad
        \delta:=\mathcal D-1>0,
    \]
    and define
    \begin{equation}
        \label{eq:B0_def}
        B_0(u,v)
        :=
        \int_0^1\left[
            u'v'
            +
            \left(
            \frac{2\mathsf M}{\delta^3}
            -\mathsf G(1-s)
            \right)uv
            \right]ds
        +
        \frac{3\mathsf M}{\delta^4}
        \left(\int_0^1u\,ds\right)
        \left(\int_0^1v\,ds\right).
    \end{equation}
    Assume
    \[
        \frac{2\mathsf M}{\delta^3}>\mathsf G.
    \]
    Then there exists a unique
    \[
        \psi
        =
        \argmin_{\substack{\phi\in V\\ \int_0^1\phi\,ds=1}}
        B_0(\phi,\phi),
    \]
    such that
    \begin{equation}
        \label{eq:psi_B0_orthogonality}
        B_0(\psi,v)=0
        \qquad
        \text{for every }v\in Y.
    \end{equation}

    and every \(\theta\in V\) admits the unique decomposition
    \begin{equation}
        \label{eq:theta_mode_decomposition}
        \theta=q\psi+w,
        \qquad
        q:=\int_0^1\theta\,ds,
        \qquad
        w\in Y.
    \end{equation}
    Furthermore, \(B_0\) is coercive on \(Y\) with the estimate
    \begin{equation}\label{eq:B0_Y_coercivity}
        B_0(v,v) \geq \|v'\|_{L^2}^2 \qquad \forall v \in Y
    \end{equation}
\end{lemma}
\begin{proof}
    The bilinear form \(B_0\) is the planar second variation at the straight-line configuration \((\theta,\alpha)=(0,0)\). It is easily seen to be coercive on \(V\). We can prove the existence of a unique minimizer \(\psi\) by the direct method of calculus of variations. The orthogonality condition \eqref{eq:psi_B0_orthogonality} follows from the Lagrange multiplier condition for the constrained minimization problem. The decomposition \eqref{eq:theta_mode_decomposition} and the coercivity \eqref{eq:B0_Y_coercivity} is a simple consequence of the definition of \(Y\).
    \proofinappendix{mode_decomposition}
\end{proof}

Next we record some properties of the normalized mode \(\psi\) and the mean-zero space \(Y\).

\begin{lemma}[Properties of the normalized mode]
    \label{lem:mode_properties}
    Let \(\psi\) be given by \Cref{lem:mode_decomposition}, and define
    \[
        a_0(s):=\frac{2\mathsf M}{\delta^3}-\mathsf G(1-s)>0.
    \]
    Then \(\psi\) is the unique normalized solution of
    \begin{equation}
        \label{eq:psi_BVP}
        -\psi''+a_0(s)\psi=\lambda_\psi,
        \qquad
        \psi(0)=0,\quad \psi'(1)=0,\quad \int_0^1\psi\,ds=1,
    \end{equation}
    where
    \begin{equation}
        \label{eq:lambda_psi}
        \lambda_\psi
        =
        \int_0^1\left((\psi')^2+a_0\psi^2\right)ds>0.
    \end{equation}
    In particular,
    \begin{equation}
        \label{eq:psi_positive}
        \psi(s)>0
        \qquad
        \text{for every }s\in(0,1],
    \end{equation}
    and
    \begin{equation}
        \label{eq:psi_mean_identity}
        \|\psi-1\|_{L^2}^2=\|\psi\|_{L^2}^2-1.
    \end{equation}
    Finally, define
    \begin{equation}
        \label{eq:mu_perp_def}
        \mu_\perp
        :=
        \inf_{v\in Y\setminus\{0\}}
        \frac{\int_0^1(v')^2\,ds}{\int_0^1v^2\,ds} \approx 20.1907
    \end{equation}
    Then, for every \(v\in Y\),
    \begin{equation}
        \label{eq:Y_poincare}
        \|v\|_{L^2}
        \leq
        \frac{1}{\sqrt{\mu_\perp}}\|v'\|_{L^2},
    \end{equation}
\end{lemma}
\begin{proof}
    The equation for \(\psi\) follows from the Lagrange multiplier condition for the constrained minimization problem. The positivity of \(\psi\) follows from the maximum principle. The identity \eqref{eq:psi_mean_identity} follows from the normalization condition \(\int_0^1\psi\,ds=1\). The Poincaré inequality \eqref{eq:Y_poincare} follows from the definition of \(\mu_\perp\).
    \proofinappendix{mode_properties}
\end{proof}

In the following proposition, we use these properties to show that when \((q,\alpha)\) satisfy certain inequalities, we can always find a mean-zero correction such that the first variation is 0 for every mean-zero variation.
\begin{proposition}[Construction and error estimate for the mean-zero correction]
    \label{thm:mode_accuracy}
    Equip \(Y\) with the norm
    \[
        \|v\|_Y:=\|v'\|_{L^2}.
    \]
    Fix \(q\in\mathbb R\) and \(\alpha\in[-\pi/2,\pi/2]\), and define the one-mode configuration and its normal vector by
    \[d_{q} = \sin(q\psi)e_1+\cos(q\psi)e_3, \quad n_{q}:=\cos(q\psi)\,e_1-\sin(q\psi)\,e_3.\]

    If certain computable quantities \(\Delta_{q,\alpha},K_{q,\alpha},\eta_{q,\alpha}\) satisfy
    \begin{equation}
        \label{eq:mode_accuracy_hypotheses}
        \Delta_{q,\alpha}>0,
        \qquad
        K_{q,\alpha}<\frac12.
    \end{equation}
    Then there exists a unique \(w=w(q,\alpha)\in Y\) which satisfies the estimates
    \begin{equation}
        \label{eq:w_error_bound}
        \|w\|_Y
        \leq
        \varepsilon_{q,\alpha}
        :=
        \frac{\eta_{q,\alpha}}{1-K_{q,\alpha}},
        \qquad
        \|w\|_{L^\infty}\leq\varepsilon_{q,\alpha},
    \end{equation}
    \begin{equation}
        \label{eq:endpoint_separation_bound}
        \left|
        \int_0^1\left[\sin(q\psi+w)e_1+\cos(q\psi+w)e_3\right]\,ds-\mathcal Dp(\alpha)
        \right|
        \geq
        \Delta_{q,\alpha},
    \end{equation}
    and solves the equation
    \begin{equation}
        \label{eq:projected_equilibrium}
        \delta\mathcal E_{\mathrm{pl},\alpha}[q\psi+w](v)=0
        \qquad
        \text{for every }v\in Y.
    \end{equation}
    Furthermore, the correction \(w(q,\alpha)\) depends smoothly on \((q,\alpha)\) wherever
    \eqref{eq:mode_accuracy_hypotheses} holds.
\end{proposition}
\begin{proof}
    The idea of the proof is simple. We define a solution map, for which fixed points will solve the first variation equation for every mean-zero variation. We then show that under our assumptions, this map is a contraction on a ball in \(Y\). All constants are numerically computable which allows verification of the assumptions of the theorem.
    \proofinappendix{mode_accuracy}
\end{proof}

The next proposition gives a condition for when the constructed configurations are a planar equilibria. Thus, we reduce the problem of finding a planar equilibrium to finding the roots of a function on a finite dimensional parameter space \((q,\alpha)\).
\begin{proposition}[Scalar reduction of the planar equilibrium equation]
    \label{prop:scalar_reduction}
    Suppose that the hypotheses of \Cref{thm:mode_accuracy} hold at
    \((q,\alpha)\), and let \(w(q,\alpha)\in Y\) be the corresponding
    mean-zero correction. Define
    \begin{equation}
        \label{eq:reduced_equation}
        F(q,\alpha)
        :=
        \delta\mathcal E_{\mathrm{pl},\alpha}
            [q\psi+w(q,\alpha)](\psi).
    \end{equation}
    Then
    \[
        \theta=q\psi+w(q,\alpha)
    \]
    is a planar equilibrium if and only if
    \[
        F(q,\alpha)=0.
    \]
\end{proposition}
\begin{proof}
    The proof is by direct computation of the first variation condition on all of \(V\).
    \proofinappendix{scalar_reduction}
\end{proof}

The next corollary shows that we can characterize the unique branch of equilibria emanating from the straight line equilibrium configuration at \(\alpha=0\) as those constructed by  \Cref{thm:mode_accuracy}.

\begin{corollary}[Planar equilibrium branch from the straight configuration]
    \label{cor:planar_branch}
    There exists a neighborhood \(U\) of \((0,0)\) such that the zero set
    \[
        \{(q,\alpha)\in U:F(q,\alpha)=0\}
    \]
    is a unique smooth curve through \((0,0)\). The corresponding
    configurations
    \[
        \theta=q\psi+w(q,\alpha)
    \]
    form the unique local branch of planar equilibria emanating from the
    straight configuration.

    In particular, this branch may be written locally as
    \[
        q=q(\alpha),
        \qquad
        q(0)=0,
    \]
    with
    \begin{equation}
        \label{eq:qprime_origin}
        q'(0)
        =
        \frac{\mathsf M(2\mathcal D+1)}
        {\delta^4 B_0(\psi,\psi)}
        >0.
    \end{equation}
\end{corollary}
\begin{proof}
    The existence of a unique smooth curve follows from the implicit function theorem applied to the scalar function \(F\). The uniqueness of the branch follows from the uniqueness of the mean-zero correction \(w(q,\alpha)\) constructed in \Cref{thm:mode_accuracy}. The formula for \(q'(0)\) follows from differentiating \(F(q(\alpha),\alpha)=0\) at \(\alpha=0\).
    \proofinappendix{planar_branch}
\end{proof}

Next we characterize the planar stability in terms of this function \(F\). Before we do this, we denote by \(\mathcal U\) the set of all \((q,\alpha)\) in a neighborhood of \((0,0)\) such that the hypotheses of \Cref{thm:mode_accuracy} hold. In particular, the constructed planar branches have a smooth dependence on \((q,\alpha)\) in \(\mathcal U\).

\begin{lemma}[Planar stability criterion]
    \label{thm:planar_stability}
    Let \((q,\alpha)\in\mathcal U\) satisfy
    \[
        F(q,\alpha)=0,
    \]
    and define the quantities
    \[
        \theta:=q\psi+w(q,\alpha),\quad \xi_{q,\alpha}
        :=
        \psi+\partial_qw(q,\alpha).
    \]
    We also define the planar second
    variation
    \[
        Q_{\mathrm{pl}}(h,v)
        :=
        \delta^2\mathcal E_{\mathrm{pl},\alpha}[\theta](h,v),
        \qquad
        h,v\in V.
    \]
    Then the following hold:
    \begin{equation}
        \label{eq:planar_Y_coercivity}
        Q_{\mathrm{pl}}(v,v)
        \geq
        (1-K_{q,\alpha})\|v\|_Y^2
        \qquad
        \text{for every }v\in Y.
    \end{equation}
    \begin{equation}
        \label{eq:xi_Y_orthogonality}
        Q_{\mathrm{pl}}(\xi_{q,\alpha},v)=0
        \qquad
        \text{for every }v\in Y,
    \end{equation}
    and
    \begin{equation}
        \label{eq:partialqF_second_variation}
        \partial_qF(q,\alpha)
        =
        Q_{\mathrm{pl}}(\xi_{q,\alpha},\xi_{q,\alpha}).
    \end{equation}
    Furthermore, every \(h\in V\) admits the unique decomposition
    \[
        h=a\xi_{q,\alpha}+v,
        \qquad
        a:=\int_0^1h\,ds,
        \qquad
        v\in Y,
    \]
    and for this decomposition, we may write the planar second variation as
    \begin{equation}
        \label{eq:planar_Hessian_decomposition}
        Q_{\mathrm{pl}}(h,h)
        =
        a^2\partial_qF(q,\alpha)
        +
        Q_{\mathrm{pl}}(v,v).
    \end{equation}
    Thus, we have the following characterization of planar stability:
    \begin{enumerate}
        \item if \(\partial_qF(q,\alpha)>0\), then \(\theta\) is strictly stable with respect to planar variations;
        \item if \(\partial_qF(q,\alpha)=0\), then \(Q_{\mathrm{pl}}\) is nonnegative and
              \[
                  \ker Q_{\mathrm{pl}}
                  =
                  \operatorname{span}\{\xi_{q,\alpha}\};
              \]
        \item if \(\partial_qF(q,\alpha)<0\), then the state is unstable and  \(Q_{\mathrm{pl}}\) has Morse index one.
    \end{enumerate}
\end{lemma}
\begin{proof}
    This is a direct computation from differentiating \(F\).
    \proofinappendix{planar_stability}
\end{proof}

The previous lemma only characterizes planar stability. In general, a planar equilibrium may still lose stability to out-of-plane perturbations. However, we will now show that in certain parameter regimes, the planar equilibria are strictly stable to out-of-plane perturbations. In particular, this is the case for the experimental parameters. We start with a lemma which shows that the full second variation splits into independent in-plane and out-of-plane components.

\begin{lemma}[Planar splitting of the full second variation]
    \label{prop:second_variation_splitting}
    Let
    \[
        d_\theta=\sin\theta\,e_1+\cos\theta\,e_3
    \]
    be a planar equilibrium, and set
    \[
        X_\theta:=\int_0^1d_\theta\,ds,
        \qquad
        g_\theta:=\nabla\Phi_\alpha(X_\theta).
    \]
    Every admissible tangent variation may be written uniquely as
    \[
        \eta=a\,n_\theta+b\,e_2,
        \qquad
        a,b\in V.
    \]
    Define
    \begin{equation}
        \label{eq:lambda_theta_def-lemma}
        \lambda_\theta
        :=
        \theta'^2
        +\mathsf G(1-s)\cos\theta
        +\mathsf M g_\theta\cdot d_\theta.
    \end{equation}
    Then the full second variation splits as
    \begin{equation}
        \label{eq:full_second_variation_split}
        \delta^2\mathcal E_{\mathrm{perm}}[d_\theta](\eta,\eta)
        =
        Q_{\mathrm{pl}}(a,a)+Q_\perp[b],
    \end{equation}
    where
    \begin{equation}
        \label{eq:Qperp_def}
        Q_\perp[b]
        =
        \int_0^1\left((b')^2-\lambda_\theta b^2\right)\,ds
        +
        \mathsf M\,\partial_{22}\Phi_\alpha(X_\theta)
        \left(\int_0^1b\,ds\right)^2.
    \end{equation}
    In particular, the magnetic potential satisfies
    \begin{equation}
        \label{eq:partial22_Phi_positive}
        \partial_{22}\Phi_\alpha(X_\theta)
        =
        -3\frac{p(\alpha)\cdot(X_\theta-\mathcal Dp(\alpha))}
        {|X_\theta-\mathcal Dp(\alpha)|^5}
        >0.
    \end{equation}
\end{lemma}
\begin{proof}
    This is proved by directly computing the second variation and using the fact that \(n_\theta\) and \(e_2\) are orthogonal.
    \proofinappendix{second_variation_splitting}
\end{proof}

In particular, \eqref{eq:partial22_Phi_positive} shows that the magnetic portion of the out-of-plane second variation is strictly positive and the stability can be determined by the sign of the local terms. Next, we give a sufficient condition for out-of-plane stability.

\begin{proposition}[Conditional out-of-plane stability]
    \label{thm:conditional_out_of_plane_stability}
    Let \(d_\theta\) be a nontrivial planar equilibrium and assume
    \begin{equation}
        \label{eq:out_of_plane_sign_conditions}
        0<\theta(s)<\pi
        \qquad
        \text{for every }s\in(0,1],\;\text{and}
        \qquad
        g_\theta \cdot e_1 =: (g_\theta)_1<0.
    \end{equation}
    or
    \begin{equation}
        -\pi<\theta(s)<0
        \qquad
        \text{for every }s\in(0,1],
        \qquad
        (g_\theta)_1>0.
    \end{equation}
    Then
    \[
        Q_\perp[b]>0
        \qquad
        \text{for every }b\in V\setminus\{0\}.
    \]
    Thus, \(d_\theta\) is strictly stable with respect to out-of-plane
    variations.
\end{proposition}
\begin{proof}
    The proof is based on Picone's identity to rewrite the local part of the out-of-plane second variation as the sum of two nonnegative terms and a boundary term. We show that the boundary term goes to 0, concluding the proof.
    \proofinappendix{conditional_out_of_plane_stability}
\end{proof}
Now we show that under certain choice of parameters, these conditions are fulfilled and we have out-of-plane stability.

\begin{proposition}[Sign preservation along the planar branch]
    \label{thm:sign_preservation}
    Assume
    \begin{align}
        \beta_0^{\mathrm{cl}}(\mathsf G) >\frac{\sqrt{6}\mathsf M}{9\delta^3}.
    \end{align}
    Let there exist \(a>0\) and a continuous planar equilibrium branch
    \[
        t\longmapsto(q(t),\alpha(t)),
        \qquad
        t\in[0,a),
    \]
    which is the positive planar equilibrium branch obtained by continuation of the
    local branch given by \Cref{cor:planar_branch}, with
    \[
        (q(0),\alpha(0))=(0,0),
    \]
    and
    \[
        F(q(t),\alpha(t))=0
        \qquad
        \text{for every }t\in[0,a).
            \]
            Set
            \[
                \theta_t
                =
                q(t)\psi+w(q(t),\alpha(t)),
                \qquad
                g_t
                =
                \nabla\Phi_{\alpha(t)}(X_{\theta_t}).
            \]
            For every \(t\in[0,a)\) such that
                \[
                    0<\alpha(t)\leq\frac{\pi}{2},
                \]
                we have
                \[
                0<\theta_t(s)<\pi
                \qquad
                \text{for every }s\in(0,1],
            \]
            and
            \[
                (g_t)_1:= g_t \cdot e_1<0.
            \]
            By reflection symmetry, the corresponding negative \(\alpha\) branch satisfies
            \[
            -\pi<\theta_t(s)<0
            \qquad
            \text{for every }s\in(0,1],
    \]
    and
    \[
        (g_t)_1>0
    \]
    whenever
    \[
        -\frac{\pi}{2}\leq\alpha(t)<0.
    \]
\end{proposition}
\begin{proof}
    We show that the sign conditions are satisified near the straight configuration at \((q,\alpha)=(0,0)\) and then use a continuation argument to show that they are preserved along the entire branch. The key is to show that the sign conditions are open and closed along the branch and thus must hold for the entire connected branch.
    \proofinappendix{sign_preservation}
\end{proof}

Now we characterize the type of planar bifurcation with conditions that are numerically verifiable.
\begin{proposition}[Planar saddle-node bifurcation]
    \label{thm:planar_fold}
    Let \((q_*,\alpha_*)\in\mathcal U\) satisfy
    \[
        F(q_*,\alpha_*)=0,
        \qquad
        \partial_qF(q_*,\alpha_*)=0.
    \]
    Assume that
    \[
        \partial_\alpha F(q_*,\alpha_*)\neq0,
        \qquad
        \partial_{qq}F(q_*,\alpha_*)\neq0.
    \]
    Then \((q_*,\alpha_*)\) is a nondegenerate saddle-node of the planar equilibrium branch and there exists a neighborhood of \(q_*\) and a unique smooth function \(\alpha(q)\)
    whose graph gives the planar equilibria near \((q_*,\alpha_*)\), with

    \[
        \alpha(q)
        =
        \alpha_*
        -\frac{\partial_{qq}F(q_*,\alpha_*)}
        {2\partial_\alpha F(q_*,\alpha_*)}
        (q-q_*)^2
        +o\left(|q-q_*|^2\right).
    \]
    At the fold, the planar second variation is nonnegative with
    \[
        \ker Q_{\mathrm{pl}}
        =
        \operatorname{span}\{\xi_{q_*,\alpha_*}\},
    \]
    with \(\xi_{q_*,\alpha_*}\) defined as in \Cref{thm:planar_stability}.
    Furthermore, the fold opens toward increasing \(\alpha\) if
    \[
        -\frac{\partial_{qq}F(q_*,\alpha_*)}
        {\partial_\alpha F(q_*,\alpha_*)}>0,
    \]
    and toward decreasing \(\alpha\) if this quantity is negative.
\end{proposition}
\begin{proof}
    This is a direct application of the implicit function theorem to the system
    \[
        F(q,\alpha)=0,\qquad \partial_qF(q,\alpha)=0.
    \]
    \proofinappendix{planar_fold}
\end{proof}

The measured experimental quantity is the angle of the tip displacement. The following Corollary writes the bifurcation characterization in terms of the tip displacement angle.

\begin{corollary}[Fold in the experimental tip displacement angle]
    \label{cor:tip_angle_fold}
    Let \((q_*,\alpha_*)\) satisfy the assumptions of
    \Cref{thm:planar_fold}, and define
    \[
        \Theta_{\mathrm{tip}}(q,\alpha)
        :=
        \arctan{\left(\frac{X_{\theta_{q,\alpha}}\cdot e_1}{X_{\theta_{q,\alpha}}\cdot e_3}\right)}.
    \]
    Set
    \[
        X_*:=X_{\theta_{q_*,\alpha_*}},
        \qquad
        V_*:=\int_0^1n_{\theta_*}\xi_{q_*,\alpha_*}\,ds.
    \]
    If
    \[
        (X_*\cdot e_3)(V_*\cdot e_1)
        -
        (X_*\cdot e_1)(V_*\cdot e_3)
        \neq0,
    \]
    then \(\Theta_{\mathrm{tip}}\) is a valid local coordinate for the
    equilibrium branch near the fold. In particular, writing
    \[
        \Theta_{\mathrm{tip},*}
        =
        \Theta_{\mathrm{tip}}(q_*,\alpha_*),
    \]
    the branch satisfies
    \[
        \alpha
        =
        \alpha_*
        -
        \frac{\partial_{qq}F(q_*,\alpha_*)}
        {2\partial_\alpha F(q_*,\alpha_*)\,T_*^2}
        (\Theta_{\mathrm{tip}}-\Theta_{\mathrm{tip},*})^2
        +
        o\!\left(
        |\Theta_{\mathrm{tip}}-\Theta_{\mathrm{tip},*}|^2
        \right),
    \]
    where
    \[
        T_*
        =
        \frac{
            (X_*\cdot e_3)(V_*\cdot e_1)
            -
            (X_*\cdot e_1)(V_*\cdot e_3)
        }{
            |X_*|^2
        }.
    \]
\end{corollary}
\begin{proof}
    This is also a consequence of the inverse function theorem which allows us to invert the map \((q,\alpha)\mapsto(\Theta_{\mathrm{tip}},\alpha)\) near the fold. The formula for \(T_*\) follows from differentiating \(\Theta_{\mathrm{tip}}\) with respect to \(q\) at the fold.
    \proofinappendix{tip_angle_fold}
\end{proof}

We summarize the physical implications of the theoretical results in the following Theorem.

\begin{theorem}[Physical consequences of Permanent Magnet Regime]
    \label{thm:permanent_physical}
    Assume \(\mathcal \delta := \mathcal{D}-1>0\), \(\mathsf M>0\), \(\mathsf G\geq0\),
    \(2\mathsf M/\delta^3>\mathsf G\), and \(\beta_0^{\mathrm{cl}}(\mathsf G) >\frac{\sqrt6\,\mathsf M}{9\delta^3}\)
    Let \((q,\alpha)\in\mathcal U\) satisfy \(F(q,\alpha)=0\), where
    \(\mathcal U\) is a region on which the hypotheses
    \(\Delta_{q,\alpha}>0\) and \(K_{q,\alpha}<1/2\) of
    \Cref{thm:mode_accuracy} hold.
    Then the following conclusions hold about the branch of planar equilibria emanating from the straight line equilibrium at \((q,\alpha)=(0,0)\).
    \begin{enumerate}
        \item
              The planar equilibria are strictly stable with respect to out-of-plane
              perturbations. Thus, if the planar equilbria are 3D unstable, they are unstable in a planar mode.

        \item
              Suppose that the planar equilibria \((q_*,\alpha_*)\) also satisfies
              \begin{equation}
                  \label{eq:permanent_physical_fold_conditions}
                  \partial_qF(q_*,\alpha_*)=0,
                  \qquad
                  \partial_\alpha F(q_*,\alpha_*)
                  \partial_{qq}F(q_*,\alpha_*)\neq0.
              \end{equation}
              Then the bifurcation is a nondegenerate saddle-node of the planar second variation.

        \item
              If additionally \(T_*\neq0\), with \(T_*\) defined in
              \Cref{cor:tip_angle_fold}, the tip displacement angle of the rod satisfies
              \begin{equation}
                  \label{eq:permanent_physical_scaling}
                  \alpha-\alpha_*
                  =-
                  \frac{\partial_{qq}F(q_*,\alpha_*)}
                  {2\partial_\alpha F(q_*,\alpha_*)T_*^2}
                  (\Theta_{\mathrm{tip}}-\Theta_{\mathrm{tip},*})^2
                  +o\left((\Theta_{\mathrm{tip}}
                  -\Theta_{\mathrm{tip},*})^2\right).
              \end{equation}
    \end{enumerate}
\end{theorem}

The reduced formulation is especially convenient for locating the folds numerically. A direct Newton solve of the full planar equilibrium equation becomes ill-conditioned near a saddle-node, since the linearized equilibrium operator is precisely the planar second variation and therefore develops a zero eigenvalue at the fold. In contrast, the decomposition
\[
    \theta=q\psi+w,\qquad w\in Y,
\]
encodes this critical direction in the scalar variable \(q\), and by \Cref{thm:planar_stability}, the complementary problem for \(w\) remains coercive on \(Y\) when \(\partial_qF=0\). Thus the infinite-dimensional nonlinear solve remains well conditioned, and we are able to locate the folds by solving the finite-dimensional system
\[
    F(q,\alpha)=0,\qquad \partial_qF(q,\alpha)=0.
\]
At a nondegenerate fold, the Jacobian of this system is
\[
    \begin{pmatrix}
        0              & \partial_\alpha F   \\
        \partial_{qq}F & \partial_{q\alpha}F
    \end{pmatrix},
\]
whose determinant is
\[
    -\partial_\alpha F\,\partial_{qq}F\neq0.
\]
Thus, as long as the nondegeneracy conditions are satisfied, the fold system is well-conditioned. To further improve the convergence of the fold computation, we also initialize the root finding with an approximate solution. For an inexpensive initial approximation to the folds, we neglect the correction \(w\) and restrict the energy to the one-dimensional family \(\theta=q\psi\), and define
\begin{equation}
    F_0(q,\alpha) := D\mathcal E_{\mathrm{pl},\alpha}[q\psi](\psi).
\end{equation}
Its folds can be computed via
\[ F_0(q_0,\alpha_0)=0,\qquad \partial_qF_0(q_0,\alpha_0)=0. \]

We use \((q_0,\alpha_0)\) as initial guesses for the exact reduced system and use the root-finding algorithm from the Scipy package. So for each trial pair, we solve for the mean-zero correction \(w(q,\alpha)\) and then evaluate \(F\) and \(\partial_q F\). For our experimental parameters, we are able to solve the system to a tolerance of \(10^{-8}\) within \(10-12\) iterations of the root-finding. We summarize our numerical results for the experimental parameters of \Cref{tab:parameters_radial} in the following Corollary.

\begin{corollary}[Numerical predictions at the experimental parameters]
    \label{cor:permanent_physical_parameters}
    Numerical solution of the reduced equations
    \(F(q,\alpha)=\partial_qF(q,\alpha)=0\) gives the values below.

    \begin{center}
        \begin{tabular}{llrrrrr}
            \toprule
            \(E\) & Fold
            & \(q_0\) & \(\alpha_0\)
            & \(q_*\) & \(\alpha_*\)
            & \(\Theta_{\mathrm{tip},*}\) \\
            \midrule
            \(200\,\mathrm{GPa}\) & Lower
            & \(0.142206\) & \(29.85402^\circ\)
            & \(0.143157\) & \(30.12789^\circ\)
            & \(8.20398^\circ\) \\
            & Upper
            & \(0.539770\) & \(39.73885^\circ\)
            & \(0.528278\) & \(39.30739^\circ\)
            & \(30.35258^\circ\) \\
            \addlinespace
            \(100\,\mathrm{GPa}\) & Lower
            & \(0.173427\) & \(39.45516^\circ\)
            & \(0.178723\) & \(40.95699^\circ\)
            & \(10.24345^\circ\) \\
            & Upper
            & \(0.919017\) & \(63.23606^\circ\)
            & \(0.872981\) & \(61.50409^\circ\)
            & \(50.41667^\circ\) \\
            \bottomrule
        \end{tabular}
    \end{center}
    At every computed fold, \(\partial_\alpha F>0\),
    \(\partial_{qq}F<0\) at each lower fold, and
    \(\partial_{qq}F>0\) at each upper fold. Thus the saddle-nodes are non-degenerate, and the lower folds open
    toward increasing \(\alpha\), and the upper folds open toward
    decreasing \(\alpha\).

    All of the hypotheses of \Cref{thm:permanent_physical} are satisfied except for the contraction condition \(K<1/2\) at the upper fold for \(E=100\,\mathrm{GPa}\) where we have \(K\approx1.28>1/2\).

    The computed hysteresis width is
    \begin{equation}
        \label{eq:permanent_physical_threshold_separation}
        \alpha_{\mathrm{upper}}-\alpha_{\mathrm{lower}}\approx
        \begin{cases}
            9.1795^\circ,&E=200\,\mathrm{GPa},\\
            20.5471^\circ,&E=100\,\mathrm{GPa}.
        \end{cases}
    \end{equation}
    Reflection over the \(e_3\) axis in the \(e_1-e_3\) plane gives the corresponding negative-angle folds.
\end{corollary}

We especially highlight that at \(E = 100\) GPa falls out of range of our contraction estimate, but the reduced equation is still numerically solvable. This suggests that this limitation is purely technical and that a more refined estimate should be able to strengthen the hypothesis of \Cref{thm:mode_accuracy} to also handle this case.

\begin{table}[ht]
    \centering
    \caption{Parameters for permanent magnet on sphere}
    \label{tab:parameters_radial}
    \begin{tabular}{|c|c|c|c|}
        \hline
        $E$ & $\mathsf{G}$ & $\mathsf{M}$ & $\mathsf{A}$\\
        \hline
        $100~\mathrm{GPa}$ & $4.4$ & $0.0472$ & $81.25$\\
        \hline
        $200~\mathrm{GPa}$ & $2.2$ & $0.0236$ & $40.62$\\
        \hline
    \end{tabular}
\end{table}

\subsubsection{Experiments and numerics for permanent magnet on sphere and Comparison with Theory}

To investigate the deformation of a ferromagnetic wire under a localized magnetic field, a custom apparatus was fabricated using 3D-printed components, see Figure~\ref{fig:radial}\subref{fig:radial_setup}. The setup consisted of two semicircular guide arcs and a movable holder carrying a spherical NdFeB permanent magnet of diameter $15$ mm. The magnet was housed in a tight-fitting cavity within the holder, ensuring a fixed orientation throughout the experiment, with its magnetization directed radially outward. Two curved slots in the holder engaged with the guide arcs, constraining the holder to slide along a circular path whose centerline passed through the magnet's center.

A nickel wire of diameter $50~\mu \mathrm{m}$ and length $9.2\,\mathrm{cm}$  was clamped at one end, with the clamp center coinciding with the center of the guide arcs. The radius of the guide arcs was chosen to be $11.25\,\mathrm{cm}$, providing sufficient clearance to prevent contact between the wire and the magnet during the experiment.

\begin{figure}[h]
    \centering

    \begin{subfigure}{0.48\textwidth}
        \centering
        \includegraphics[width=\linewidth]{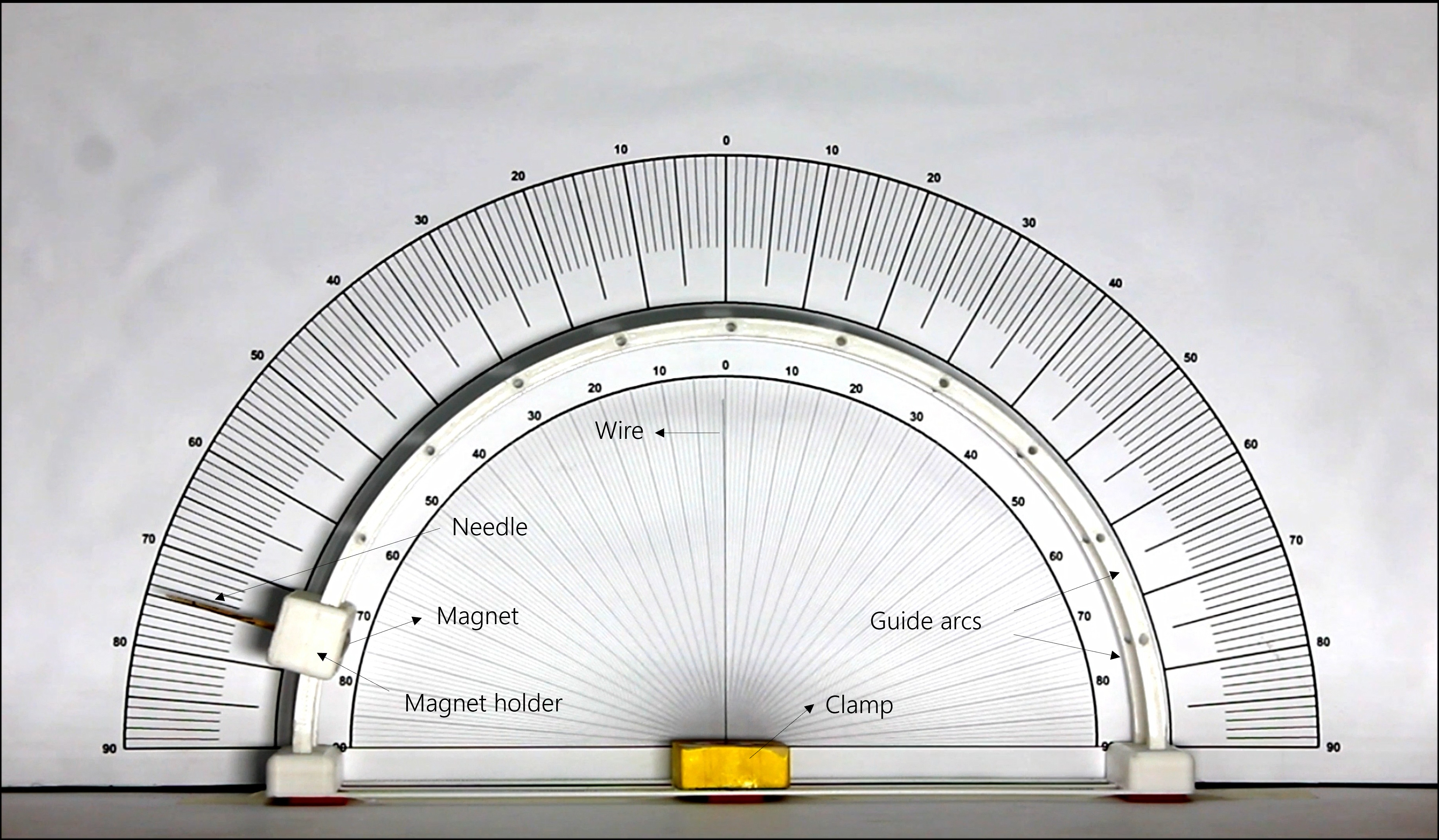}
        \caption{}
        \label{fig:radial_setup}
    \end{subfigure}
    \hfill
    \begin{subfigure}{0.48\textwidth}
        \centering
        \includegraphics[width=\linewidth]{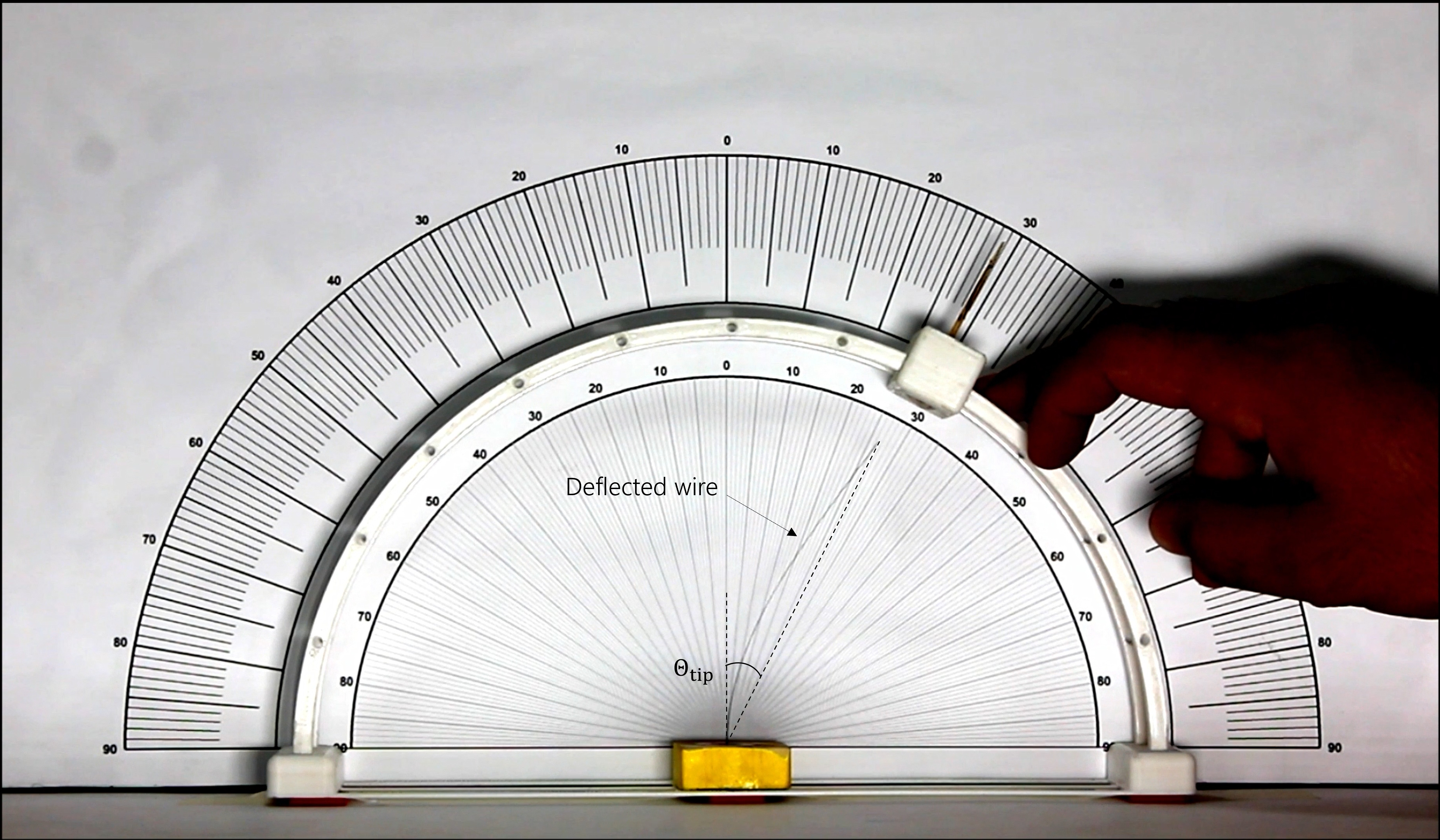}
        \caption{}
        \label{fig:radial_deformed}
    \end{subfigure}

    \caption{Experimental setup for studying the deformation of a clamped nickel wire under localized magnetic actuation. (a) Photograph of the apparatus showing the guide arcs, magnet holder, and angular measurement scales. (b) Deformed configuration of the wire for a prescribed magnet angle $\alpha$. The wire-tip angle $\Theta_{\mathrm{tip}}$ is measured with respect to the vertical direction.}
    \label{fig:radial}
\end{figure}

A white background bearing two angular scales was positioned behind the apparatus to enable visual tracking of the system's configuration. A needle mounted on the holder traced the first scale, giving the magnet angle $\alpha$ relative to the vertical. The second scale, a set of radial lines emanating from the clamp center, was used to read the wire-tip angle $\Theta_{\mathrm{tip}}$, defined as the polar angle of the wire tip from the vertical, see Figue~\ref{fig:radial}\subref{fig:radial_deformed}.

Prior to the experiment, the relative magnitudes of $\mathsf A$, $\mathsf M$, and $\mathsf G$, see \autoref{tab:parameters_radial}, show that the demagnetization energy dominates the Zeeman and gravitational contributions, favoring a uniaxial magnetization state along the local tangent, $\mathbf{m} \parallel \mathbf{d}_3$. The wire was magnetized by placing the clamped rod in a strong uniform field oriented vertically upward, which selected the $\mathbf m=+\mathbf{d}_3$ polarity along its length.

In the experiment, the magnet angle $\alpha$ was swept from
$-50^\circ$ to $+50^\circ$ in the clockwise (CW) direction, then back from $+50^\circ$ to $-50^\circ$ in the counterclockwise (CCW) direction. At each magnet position, the wire was allowed to relax before recording $\Theta_{\mathrm{tip}}$. The resulting $\Theta_{\mathrm{tip}}$ versus $\alpha$ plot is shown in Figure~\ref{fig:theta_plot}.

\begin{figure}[hbt]
    \centering
    \centering
    \includegraphics[scale=0.8]{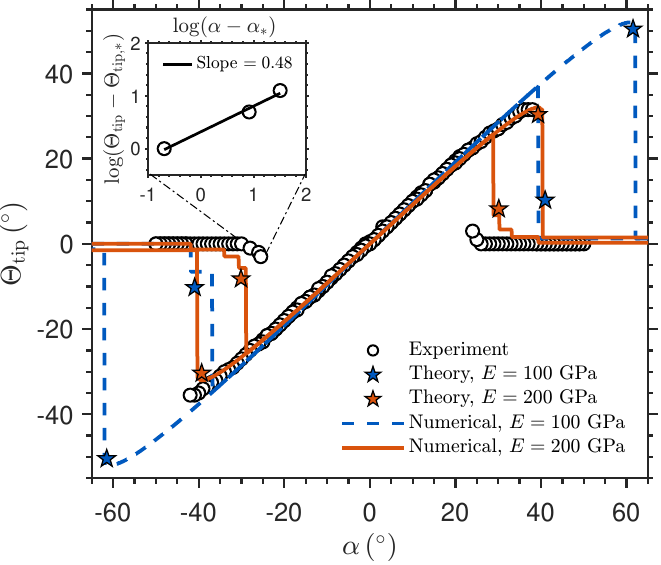}
    \caption{Variation of the wire-tip angle $\Theta_{\mathrm{tip}}$ with the magnet angle $\alpha$. Black circles denote experimental measurements, while dashed-blue and solid-orange curves represent numerical solutions for $E=100$ and $200~\mathrm{GPa}$, respectively. Blue and orange stars denote the theoretically computed fold angles for $E=100$ and $200~\mathrm{GPa}$, respectively. Inset shows a log--log plot of $(\Theta_\mathrm{tip} - \Theta_{\mathrm{tip},*})$ versus $(\alpha - \alpha_*)$ near the fold point. The linear fit gives a slope of $0.48$, consistent with the theoretically predicted square-root scaling.}
    \label{fig:theta_plot}
\end{figure}

During the CW sweep, the wire initially exhibits negligible deformation as the magnet
approaches the vertical direction. At a first fold, $\alpha \approx -25^\circ$, the wire
undergoes a saddle-node bifurcation, snapping rapidly to a strongly deflected, magnetically engaged configuration and
producing a sharp change in $\Theta_{\mathrm{tip}}$. In the approach to this fold, $\Theta_{\mathrm{tip}}$
varies nonlinearly with $\alpha$, consistent with the local scaling predicted by
Corollary~\ref{cor:tip_angle_fold}: near a fold, $\alpha - \alpha_* $ is proportional to $ (\Theta_{\mathrm{tip}} - \Theta_{\mathrm{tip,}*})^2$,
so that $\Theta_{\mathrm{tip}}$ departs from its fold value as $\sqrt{\alpha - \alpha_*}$ rather than
linearly. This square-root scaling law is supported experimentally by the log--log fit shown in the inset of Figure~\ref{fig:theta_plot}, which yields a slope of $0.48$, in close agreement with the predicted exponent of $1/2$.

Following the transition, the wire remains magnetically engaged, and
$\Theta_{\mathrm{tip}}$ varies approximately linearly with $\alpha$ over a broad range. Upon further
rotation, a second fold is reached at $\alpha \approx 38^\circ$, at which the wire abruptly
disengages from the field and returns to the weakly deflected configuration; as before, the
approach to this fold shows the same characteristic nonlinear variation
in the experimental data.

Reversing the sweep direction (CCW) reproduces the same sequence of engagement, tracking,
and disengagement, but with the folds now located at $\alpha \approx 24^\circ$ and
$\alpha \approx -42^\circ$, each again preceded by the characteristic nonlinear approach in
$\Theta_{\mathrm{tip}}$. This dependence of the fold locations on sweep direction is a consequence of the bistability of the rod over a range of $\alpha$ and is precisely what gives rise to the hysteresis loop observed in Figure~\ref{fig:theta_plot}.

The numerical results obtained by direct energy minimization reproduce the smooth branches of the experimental response with good accuracy for the assumed Young's modulus $E=200\,\mathrm{GPa}$. In contrast, the lower value $E=100\,\mathrm{GPa}$ substantially overpredicts the magnitude of both fold angles, while $E=200\,\mathrm{GPa}$ closely captures their experimentally observed locations. This comparison suggests that the actual Young's modulus of the wire lies toward the upper end of the considered range.

As discussed in \autoref{subsec:permanent_magnet}, direct energy minimization becomes ill-conditioned in the vicinity of the folds. The fold angles predicted by the Lyapunov--Schmidt reduction in Corollary~\ref{cor:permanent_physical_parameters}, nevertheless, agree closely with the jump locations of the corresponding numerical response curves for both values of $E$. This agreement demonstrates consistency between the reduced analytical description and the numerical planar minimization.

Minor discrepancies between the experimental and theoretical results appear near the transition regions, where the response is most sensitive to geometric imperfections, uncertainties in material properties, and small misalignments in the experimental setup. Taken together, this agreement demonstrates that the model captures the essential mechanics governing the deformation of the wire under localized magnetic actuation, including both the smooth equilibrium response away from the folds and the bistable, saddle-node-mediated switching that gives rise to the observed hysteresis.

\section{Discussion and outlook}
\label{s.discussion}
In this paper we presented a variational and a bifurcation landscape analysis for slender magnetoelastic rods in two loading regimes. Our paper brings together analysis, experiments, and simulation, with quantitative agreement. Two insights underlie our modeling and analysis, both of which have application to  experiments beyond what we report here.

First, the dimension reduction of the micromagnetic energy turns the nonlocal stray-field energy to a local anisotropy. Soft magnets are saturated by moderate strengths of applied fields, and therefore,  the stability of the vertical rod depends on geometry only through the
two ratios $\mathsf{G}$ and $\mathsf{A}$, and the sufficient condition
$L < K_d/(\rho g)$ of Corollary~\ref{thm:stabilization}. The latter condition involves neither the diameter
nor the Young modulus of the rod.

Secondly, in regimes of small exchange, it is also energetically favorable for the magnetization to align with the tangent of the rod. In particular when such a rod is
placed in any field that derives from a scalar potential, one is led to an elastica under gravity
loaded by a single force at its free end. The nonlocal magnetoelastic problem becomes
an end-loaded elastica.

Our analysis also informs a computational method of the folds of the hysteresis loop, since the solutions of the full equilibrium equations become ill-conditioned near the folds, whereas the Lyapunov-Schmidt reduced system
$F = \partial_q F = 0$ is not, and we are led to solving a two-dimensional problem to determine the folds. The fold
angles are sensitive to the modulus: the lower fold moves from $30^\circ$ to
$41^\circ$ as $E$ drops from $200$ to $100$ GPa. Measured switching angles can therefore provide a means to estimate~$\mathsf E$ of a thin wire, which is otherwise difficult to measure. The hysteresis width \eqref{eq:intro-hysteresis} is
set by $\mathsf{M}$ and $\mathcal{D}$, that is, by magnet strength and clearance, and is a
tunable design parameter for a remotely switched bistable element.

Several questions remain open. The two experiments studied here were consistent with a relatively explicit magnetization pattern in which exchange energy was unimportant. We expect to see (and have preliminary experiments) when this is not always the case.

Our
experiments are quasi-static, whereas the snap itself is dynamic, and its time scale and the
role of damping are also natural topics worthy of investigation.  Finally, trajectories that leave the plane, engaging twist and the out-of-plane modes, or problems of prescribed relative twists of the two ends of the rod, is a very interesting direction for future work.

\appendix
\section{Proofs for Uniform Magnetic Field}
\begin{appendixproof}[Proof of \Cref{thm:uniform_variations}]\label{pf:uniform_variations}
    Let \(\eta\) be an admissible tangent variation, so that \(\eta\cdot d_3=0\), and in the clamped case \(\eta(0)=0\). Writing \(d_{3,\varepsilon}=d_3+\varepsilon\eta\), we note that the orthogonality condition preserves the length constraint to first order. Differentiation of the three terms in the energy gives
    \[
        \left.\frac{d}{d\varepsilon}\right|_{\varepsilon=0}\frac12\int_0^1|d_{3,\varepsilon}'|^2\,ds=\int_0^1d_3'\cdot\eta'\,ds,
    \]
    \[
        \left.\frac{d}{d\varepsilon}\right|_{\varepsilon=0}\frac{\mathsf A}{2}\int_0^1\left[1-(e_3\cdot d_{3,\varepsilon})^2\right]\,ds=-\mathsf A\int_0^1(e_3\cdot d_3)(e_3\cdot\eta)\,ds,
    \]
    and
    \[
        \left.\frac{d}{d\varepsilon}\right|_{\varepsilon=0}\mathsf G\int_0^1(1-s)e_3\cdot d_{3,\varepsilon}\,ds=\mathsf G\int_0^1(1-s)e_3\cdot\eta\,ds.
    \]
    Thus,
    \[
        \delta\mathcal E_{\mathrm{unif}}[d_3](\eta)=\int_0^1d_3'\cdot\eta'\,ds-\mathsf A\int_0^1(e_3\cdot d_3)(e_3\cdot\eta)\,ds+\mathsf G\int_0^1(1-s)e_3\cdot\eta\,ds.
    \]
    Integrating the first term by parts gives
    \[
        \delta\mathcal E_{\mathrm{unif}}[d_3](\eta)=\left[d_3'\cdot\eta\right]_0^1+\int_0^1\left[-d_3''-\left(\mathsf A(e_3\cdot d_3)-\mathsf G(1-s)\right)e_3\right]\cdot\eta\,ds.
    \]
    Since \(\eta\) is an arbitrary tangential variation, the interior Euler--Lagrange equation is
    \[
        d_3\times\left[-d_3''-\left(\mathsf A(e_3\cdot d_3)-\mathsf G(1-s)\right)e_3\right]=0.
    \]
    Equivalently, there exists a scalar function \(\lambda\) such that
    \[
        -d_3''-\left(\mathsf A(e_3\cdot d_3)-\mathsf G(1-s)\right)e_3=\lambda d_3.
    \]
    Taking the scalar product with \(d_3\) and using \(|d_3|=1\), and \(d_3''\cdot d_3=-|d_3'|^2\), yields
    \[
        \lambda=|d_3'|^2-\mathsf A(e_3\cdot d_3)^2+\mathsf G(1-s)e_3\cdot d_3.
    \]
    In the clamped and pinned case, the boundary term vanishes, and we obtain the corresponding boundary conditions.\\

    We now compute the second variation. Let \(d_{3,\varepsilon}\) be an admissible variation of the equilibrium \(d_3\), and set
    \[
        \eta:=\left.\frac{d}{d\varepsilon}\right|_{\varepsilon=0}d_{3,\varepsilon},
        \qquad
        \zeta:=\left.\frac{d^2}{d\varepsilon^2}\right|_{\varepsilon=0}d_{3,\varepsilon}.
    \]
    Differentiating \(|d_{3,\varepsilon}|^2=1\) gives the conditions
    \[
        d_3\cdot\eta=0,
        \qquad
        d_3\cdot\zeta=-|\eta|^2.
    \]
    For the bending energy,
    \[
        \left.\frac{d^2}{d\varepsilon^2}\right|_{\varepsilon=0}\frac12\int_0^1|d_{3,\varepsilon}'|^2\,ds
        =
        \int_0^1|\eta'|^2\,ds+\int_0^1d_3'\cdot\zeta'\,ds.
    \]
    The anisotropy term gives
    \[
        \left.\frac{d^2}{d\varepsilon^2}\right|_{\varepsilon=0}\frac{\mathsf A}{2}\int_0^1\left[1-(e_3\cdot d_{3,\varepsilon})^2\right]\,ds
        =
        -\mathsf A\int_0^1\left[(e_3\cdot\eta)^2+(e_3\cdot d_3)(e_3\cdot\zeta)\right]\,ds,
    \]
    while the gravitational term gives
    \[
        \left.\frac{d^2}{d\varepsilon^2}\right|_{\varepsilon=0}\mathsf G\int_0^1(1-s)e_3\cdot d_{3,\varepsilon}\,ds
        =
        \mathsf G\int_0^1(1-s)e_3\cdot\zeta\,ds.
    \]
    Combining these terms, we obtain
    \[
        \delta^2\mathcal E_{\mathrm{unif}}[d_3](\eta)
        =
        \int_0^1\left[|\eta'|^2-\mathsf A(e_3\cdot\eta)^2\right]\,ds
        +\int_0^1d_3'\cdot\zeta'\,ds
        +\int_0^1\left[-\mathsf A(e_3\cdot d_3)e_3+\mathsf G(1-s)e_3\right]\cdot\zeta\,ds.
    \]
    Integrating the term involving \(d_3'\) by parts, and using that the boundary contribution vanishes by the corresponding boundary conditions in both the clamped or pinned case, we achieve
    \[
        \delta^2\mathcal E_{\mathrm{unif}}[d_3](\eta)
        =
        \int_0^1\left[|\eta'|^2-\mathsf A(e_3\cdot\eta)^2\right]\,ds
        +\int_0^1\left[-d_3''-\left(\mathsf A(e_3\cdot d_3)-\mathsf G(1-s)\right)e_3\right]\cdot\zeta\,ds.
    \]
    Using the Euler--Lagrange equation and \(d_3\cdot\zeta=-|\eta|^2\), we obtain
    \[
        \int_0^1\left[-d_3''-\left(\mathsf A(e_3\cdot d_3)-\mathsf G(1-s)\right)e_3\right]\cdot\zeta\,ds
        =
        \int_0^1\lambda d_3\cdot\zeta\,ds
        =
        -\int_0^1\lambda|\eta|^2\,ds.
    \]
    Therefore,
    \[
        \delta^2\mathcal E_{\mathrm{unif}}[d_3](\eta)
        =
        \int_0^1\left[|\eta'|^2-\mathsf A(e_3\cdot\eta)^2-\lambda|\eta|^2\right]\,ds,
    \]
    which proves the result.
\end{appendixproof}

\begin{appendixproof}[Proof of \Cref{thm:uniform_straight_stability}]\label{pf:uniform_straight_stability}
    By \Cref{cor:uniform_straight_second_variation}, every admissible tangent variation at \(d_3=e_3\) can be written as
    \[
        \eta=u_1e_1+u_2e_2,
    \]
    and the second variation is
    \[
        \delta^2\mathcal E_{\mathrm{unif}}[e_3](\eta)
        =
        \sum_{i=1}^2\int_0^1\left[(u_i')^2+\left(\mathsf A-\mathsf G(1-s)\right)u_i^2\right]\,ds.
    \]
    For the clamped boundary condition, \(u_i\in V\), while for the pinned boundary condition \(u_i\in H^1((0,1))\). Thus, by the definition of \(\beta_0^{\mathrm{bc}}(\mathsf G)\),
    \[
        \int_0^1\left[(u_i')^2-\mathsf G(1-s)u_i^2\right]\,ds
        \geq
        \beta_0^{\mathrm{bc}}(\mathsf G)\int_0^1u_i^2\,ds.
    \]
    Therefore,
    \[
        \delta^2\mathcal E_{\mathrm{unif}}[e_3](\eta)
        \geq
        \left(\mathsf A+\beta_0^{\mathrm{bc}}(\mathsf G)\right)
        \int_0^1|\eta|^2\,ds.
    \]
    If \(\mathsf A+\beta_0^{\mathrm{bc}}(\mathsf G)>0\), the second variation is strictly positive for every nonzero admissible variation, and the straight configuration is strictly stable.\\

    Conversely, the infimum defining \(\beta_0^{\mathrm{bc}}(\mathsf G)\) is attained by a principal eigenfunction \(u_0\neq0\). Taking \(\eta=u_0e_1\), we obtain
    \[
        \delta^2\mathcal E_{\mathrm{unif}}[e_3](u_0e_1)
        =
        \left(\mathsf A+\beta_0^{\mathrm{bc}}(\mathsf G)\right)
        \int_0^1u_0^2\,ds.
    \]
    Thus, if \(\mathsf A+\beta_0^{\mathrm{bc}}(\mathsf G)\leq0\), the second variation is not strictly positive. Therefore \(d_3=e_3\) is strictly stable if and only if
    \[
        \mathsf A+\beta_0^{\mathrm{bc}}(\mathsf G)>0.
    \]
\end{appendixproof}

\begin{appendixproof}[Proof of \Cref{thm:uniform_eigenvalue_bounds}]\label{pf:uniform_eigenvalue_bounds}

    First, we rewrite the eigenvalue problem in a more convenient form.
    We claim that we can write the principal eigenvalues as
    \[
        \nu_{\mathrm{bc}}(t)=\frac{\beta_0^{\mathrm{bc}}(\mathsf G)+\mathsf G}{\mathsf G^{2/3}},
    \]
    where
    \[
        \nu_{\mathrm p}(t)=\inf_{v\in H^1((0,t))\setminus\{0\}}\frac{\int_0^t\left[(v')^2+xv^2\right]\,dx}{\int_0^t v^2\,dx},
    \]
    and
    \[
        \nu_{\mathrm{cl}}(t)=\inf_{\substack{v\in H^1((0,t))\setminus\{0\}\\v(0)=0}}\frac{\int_0^t\left[(v')^2+xv^2\right]\,dx}{\int_0^t v^2\,dx}.
    \]

    To see this, we perform the change of variables for any \(\mathsf G>0\) and \(u\in H^1((0,1))\),
    \[
        t:=\mathsf G^{1/3},\qquad x:=ts,\qquad v(x)= u\left(\frac{x}{t}\right).
    \]
    Then \(ds=dx/t\) and \(u'(s)=tv'(x)\), so
    \[
        \int_0^1(u')^2\,ds=t\int_0^t(v')^2\,dx,
    \]
    \[
        -\mathsf G\int_0^1(1-s)u^2\,ds=-t^2\int_0^t v^2\,dx+t\int_0^t xv^2\,dx.
    \]
    \[
        \int_0^1u^2\,ds=\frac1t\int_0^t v^2\,dx.
    \]

    Plugging this into the formula for \(\beta_0^{\mathrm{bc}}(\mathsf G)\) gives exactly the claimed formula for \(\nu_{\mathrm{bc}}(t)\).

    Now we give some simple bounds. First, we have the trivial lower bound \(\nu_{\mathrm{bc}}(t)\geq0\) since all terms are positive, which gives
    \[
        \beta_0^{\mathrm{bc}}(\mathsf G)\geq-\mathsf G.
    \]
    Next, note that if we take a compactly supported test function \(v\) in the Rayleigh quotient for \(\nu_{\mathrm{bc}}(t)\), then we have a uniform bound on the Rayleigh quotient for every \(t\) larger than the support of the test function. Thus, \(\nu_{\mathrm{bc}}(t)\) is bounded above by a constant independent of \(t\).

    To get finer bounds, we use the eigenvalue ODE. The principal eigenvalue \(\nu_{\mathrm{bc}}(t)\) is characterized by the ODE
    \[
        -v''+xv=\nu_{\mathrm{bc}}(t)v,\qquad v\in H^1((0,t)),
    \]
    with the boundary conditions \(v'(0)=0\) for the pinned case and \(v(0)=0\) for the clamped case, and \(v'(t)=0\) for both cases. The general solution of the ODE is
    \[
        v(x)=c_1\Ai(x-\nu_{\mathrm{bc}}(t))+c_2\Bi(x-\nu_{\mathrm{bc}}(t)),
    \]
    where \(\Ai\) and \(\Bi\) are the Airy functions. Now we impose the boundary conditions. At \(x=t\), we have \(v'(t)=0\), which gives

    \begin{align*}
        &c_1\Ai'(t-\nu_{\mathrm{bc}}(t))+c_2\Bi'(t-\nu_{\mathrm{bc}}(t))=0\\
        &c_2 = -c_1\frac{\Ai'(t-\nu_{\mathrm{bc}}(t))}{\Bi'(t-\nu_{\mathrm{bc}}(t))}\\
        \implies & v(x)=c_1\left[\Ai(x-\nu_{\mathrm{bc}}(t))-\frac{\Ai'(t-\nu_{\mathrm{bc}}(t))}{\Bi'(t-\nu_{\mathrm{bc}}(t))}\Bi(x-\nu_{\mathrm{bc}}(t))\right].
    \end{align*}
    Thus, the eigenvalue is characterized by the remaining boundary condition at \(x=0\). For the pinned case, we have \(v'(0)=0\), which gives
    \begin{equation}
        \label{eq:pinned_Airy_characteristic}
        \Ai'(-\nu_{\mathrm{bc}}(t))-\frac{\Ai'(t-\nu_{\mathrm{bc}}(t))}{\Bi'(t-\nu_{\mathrm{bc}}(t))}\Bi'(-\nu_{\mathrm{bc}}(t)) = 0.
    \end{equation}
    For the clamped case, we have \(v(0)=0\), which gives
    \begin{equation}
        \label{eq:clamped_Airy_characteristic}
        \Ai(-\nu_{\mathrm{bc}}(t))-\frac{\Ai'(t-\nu_{\mathrm{bc}}(t))}{\Bi'(t-\nu_{\mathrm{bc}}(t))}\Bi(-\nu_{\mathrm{bc}}(t)) = 0.
    \end{equation}

    In both cases, we consider what happens as \(t\to\infty\). In this limit, since the eigenvalues are bounded uniformly in \(t\),we have that
    \[
        \frac{\Ai'(t-\nu_{\mathrm{bc}}(t))}{\Bi'(t-\nu_{\mathrm{bc}}(t))}\Bi(-\nu_{\mathrm{bc}}(t)) \to 0
    \]
    where we used the exponential decay bounds on \(\frac{\Ai'(x)}{\Bi'(x)}\) \cite{AS}. In the limit \(t\to\infty\), the eigenvalue is characterized as the root of \(\Ai\) or \(\Ai'\) depending on the boundary condition. Thus, we can get an upper bound on the eigenvalues by testing the minimization problem with the \(\Ai\) portion of \(v\). \\

    To be precise, define
    \[
        \mu_{\mathrm{bc}} = \begin{cases}
            a_1 & \mathrm{clamped}\\
            a_1' & \mathrm{pinned}
        \end{cases}
    \]
    where \(a_1\) and \(a_1'\) are the largest zeros of \(\Ai\) and \(\Ai'\). These are well defined as the roots are all negative and simple. Then we define the test functions:
    \[
        \phi_{\mathrm {bc}}(x):=\Ai(x+\mu_{bc}).
    \]

    Note that \(\phi_{\mathrm {bc}}\) satisfies the eigenvalue equation
    \[
        -\phi_{\mathrm {bc}}''+x\phi_{\mathrm {bc}}=-\mu_{\mathrm {bc}}\phi_{\mathrm {bc}}.
    \]
    It satisfies the correct boundary condition at \(x=0\) for the corresponding boundary condition.
    Thus, we can use \(\phi_{\mathrm {bc}}\) as a test function in the Rayleigh quotient for \(\nu_{\mathrm {bc}}(t)\).

    By plugging in, integrating by parts, and applying the boundary conditions, we obtain
    \[
        \nu_{\mathrm {bc}}(t)\leq-\mu_{\mathrm {bc}}+\frac{\phi_{\mathrm {bc}}(t)\phi_{\mathrm {bc}}'(t)}{\int_0^t\phi_{\mathrm {bc}}^2\,dx}.
    \]

    Now we split into cases. For the pinned case, we have \(\mu_{\mathrm p}=a_1'\). Since \(a_1'\) is the largest zero of \(\Ai'\), \(a_1 < a_1'< 0\), and \(\Ai(0)>0\), we have \(\phi_{\mathrm p}(t)>0\) and \(\phi_{\mathrm p}'(t)<0\) for every \(t>0\). This gives the upper bound
    \[
        \nu_{\mathrm p}(t)<|a_1'|,
    \]
    which yields
    \[
        \beta_0^{\mathrm p}(\mathsf G)<-\mathsf G+|a_1'|\mathsf G^{2/3}.
    \]
    On the other hand, testing the minimization problem for \(\beta_0^{\mathrm p}\) with the constant function \(u =1\), we get the upper bound \(\beta_0^{\mathrm p}(\mathsf G)< -\frac{\mathsf G}{2}\).
    Thus, this gives the upper bound for the pinned eigenvalue as the minimum of the two upper bounds we constructed. For the clamped case, we specialize to the experimentally relevant case \(\mathsf G>\mathsf G_{\mathrm{crit}}\). Note that \(\phi_{\mathrm{cl}}(t)>0\) always. However, we only have \(\phi_{\mathrm{cl}}'(t)<0\) for every \(t>a_1'-a_1\) since \(a_1'\) is the largest zero of \(\Ai'\). Numerically \(a_1'- a_1 \approx1.3193<\mathsf G_{\mathrm{crit}}^{1/3}\approx1.9864\). Thus, in the regime \(\mathsf G>\mathsf G_{\mathrm{crit}}\), we have
    \[
        \nu_{\mathrm{cl}}(t)<|a_1|,
    \]
    which yields
    \[
        \beta_0^{\mathrm{cl}}(\mathsf G)<-\mathsf G+|a_1|\mathsf G^{2/3}.
    \]
    This gives the upper bound for the clamped eigenvalue.\\

    Now we give the lower bounds, by showing the monotonicity of the eigenvalues with respect to \(t\).
    We restrict to \(\mathsf {G} > \mathsf{G}_{\mathrm{crit}}\) for symmetric presentation, even though in the pinned case, monotonicity holds for all \(\mathsf G>0\).\\

    Let \(\phi_{\mathrm{crit}}\) be the principal eigenfunction for the critical value \( t_{\mathrm{crit}}=\mathsf G_{\mathrm{crit}}^{1/3}\), normalized so that \(\int_0^{t_{\mathrm{crit}}}\phi_{\mathrm{crit}}^2\,dx=1\). Furthermore, note that we have a direct upper bound on the eigenvalue at the critical value, \(\nu_{\mathrm{bc}}(t_{\mathrm{crit}})\leq t_{\mathrm{crit}}\). Indeed, in the pinned case, this is true by testing with the constant function. In the clamped case, it is because by definition of the critical value, \(\beta_0^{\mathrm{cl}}(\mathsf G_{\mathrm{crit}})=0\), which gives \(\nu_{\mathrm{cl}}(t_{\mathrm{crit}})=t_{\mathrm{crit}}\).

    Now given any \(t>t_{\mathrm{crit}}\), we can extend \(\phi_{\mathrm{crit}}\) (without replacing the notation) to a function on \((0,t)\) by defining
    \[
        \phi_{\mathrm{crit}}(x):=\phi_{\mathrm{crit}}(t_{\mathrm{crit}})\qquad\text{for }x>t_{\mathrm{crit}}.
    \]
    This function is in the admissible space for the Rayleigh quotient for \(\nu_{\mathrm{bc}}(t)\). Now given the eigenfunction \(v_t\) for \(\nu_{\mathrm{bc}}(t)\), we can compute the following product
    \begin{align*}
        \nu_{\mathrm{bc}}({t_{\mathrm{crit}}})\int_0^{t}\phi_{\mathrm{crit}}v_t\,dx& = \nu_{\mathrm{bc}}({t_{\mathrm{crit}}})\int_0^{t_{\mathrm{crit}}}\phi_{\mathrm{crit}}v_t\,dx + \nu_{\mathrm{bc}}({t_{\mathrm{crit}}})\int_{t_{\mathrm{crit}}}^{t}\phi_{\mathrm{crit}}v_t\,dx\\
        & = \int_0^{t_{\mathrm{crit}}}\left(-\phi_{\mathrm{crit}}''+ x \phi_{\mathrm{crit}}\right)v_t\,dx + \nu_{\mathrm{bc}}({t_{\mathrm{crit}}})\int_{t_{\mathrm{crit}}}^{t}\phi_{\mathrm{crit}}v_t\,dx\\
        & \leq \int_0^{t}\left(-\phi_{\mathrm{crit}}''+ x \phi_{\mathrm{crit}}\right)v_t\,dx,
    \end{align*}
    where in the last line we used that \(\phi_{\mathrm{crit}}\) is constant in the extension region, and that \(\nu_{\mathrm{bc}}({t_{\mathrm{crit}}})\leq {t_{\mathrm{crit}}}\leq x\) in the extension region. Using the self-adjointness of the operator, we can integrate by parts to move the derivatives onto \(v_t\), and then use the eigenvalue equation for \(v_t\) to obtain
    \[
        \nu_{\mathrm{bc}}({t_{\mathrm{crit}}})\int_0^{t}\phi_{\mathrm{crit}}v_t\,dx\leq \nu_{\mathrm{bc}}(t)\int_0^{t}\phi_{\mathrm{crit}}v_t\,dx.
    \]
    Since \(\phi_{\mathrm{crit}}\) and \(v_t\) are both positive, we can divide by the integral to obtain
    \[
        \nu_{\mathrm{bc}}({t_{\mathrm{crit}}})\leq \nu_{\mathrm{bc}}(t).
    \]
    Plugging into the definition of \(\nu_{\mathrm{bc}}(t)\) gives the lower bounds
    \[
        \beta_0^{\mathrm p}(\mathsf G)\geq-\mathsf G+c_{\mathrm p}\mathsf G^{2/3},\qquad
        \beta_0^{\mathrm{cl}}(\mathsf G)\geq-\mathsf G+\mathsf G_{\mathrm{crit}}^{1/3}\mathsf G^{2/3}.
    \]
\end{appendixproof}

\begin{appendixproof}[Proof of \Cref{thm:uniform_bifurcation_criticality}]
    \label{pf:uniform_bifurcation_criticality}

    Let \(\varphi\in H^1((0,1);\mathbb R^3)\) denote the tangential variation,
    with \(\varphi\cdot e_3=0\) and, in the clamped case, \(\varphi(0)=0\).
    We parameterize the director as

    \[
        d_3 =\exp(A_{e_3\times\varphi})e_3,
    \]
    where \(A_v\) is the antisymmetric matrix corresponding to the cross product with \(v\).
    At \(\mathsf A=\mathsf A_{\mathrm{crit}}\), the second variation acts identically
    in the two transverse directions \(e_1,e_2\), and the operator associated to the eigenvalue problem is

    \begin{equation}
        \label{eq:uniform_critical_eigenvalue_problem}
        L_{\mathrm{crit}}h :=-h''+\bigl(\mathsf A_{\mathrm{crit}}-\mathsf G(1-s)\bigr)h,
        \qquad h\in e_3^\perp.
    \end{equation}

    This operator is self-adjoint and its kernel is
    \(\operatorname{span}\{ue_1,ue_2\}\), where \(u\) is the associated Airy-type function which satisfies \(L_{\mathrm{crit}}u=0\) with the appropriate boundary conditions.\\

    We use this to construct a Lyapunov--Schmidt
    decomposition. In the equations below, orthogonality will always be respect to the \(L^2((0,1))\) inner product. We decompose \(\varphi \in \operatorname{span}(e_1,e_2)\) as

    \begin{equation}
        \label{eq:uniform_lyapunov_schmidt_decomposition}
        \varphi=u q+w,
        \qquad q \in \operatorname{span}(e_1,e_2),
        \qquad \int_0^1u w\,ds=0,
    \end{equation}
    where \(w \in H^1((0,1))\) and satisfies \(w(0)=0\) in the clamped case. This allows us to make the following reduction. For fixed \(q\), \(w\), and $\mathsf A$, define the linear functional for every \(\psi\in\operatorname{span}\{ue_1,ue_2\}^{\perp},\)
    \begin{align}
        f(q,w,\mathsf A)(\psi)
        & :=\left.\frac{d}{d\varepsilon}
        \mathcal E_{\mathrm{unif}}\left[
            \exp\left(A_{e_3\times(u q+w+\varepsilon\psi)}\right)e_3
            \right]\right|_{\varepsilon=0} \nonumber \\
        & = \left.\delta \mathcal{E}_{\mathrm{unif}}\left[\exp(A_{e_3\times(u q+w+\varepsilon\psi)}e_3)\right](\partial_\varepsilon \exp\left(A_{e_3\times(u q+w+\varepsilon\psi)}e_3\right))\right|_{\varepsilon=0}, \label{eq:uniform_complementary_equation}
    \end{align}
    where we applied chain rule and the Frechet derivative became the constrained first variation because we always have \(\partial_\varepsilon \exp\left(A_{e_3\times(u q+w+\varepsilon\psi)}\right)e_3\cdot \exp\left(A_{e_3\times(u q+w+\varepsilon\psi)}\right)e_3=0\).
    To compute the derivative of the exponential we use the general form of the exponential derivative:
    \begin{equation} \label{eq:uniform_exponential_derivative}
        \partial_\varepsilon \left.\exp\left(A_{e_3\times(u q+w+\varepsilon\psi)}\right)e_3\right|_{\varepsilon=0}
        =\int_0^1\exp\left(tA_{e_3\times(u q+w)}\right)A_{e_3\times\psi}\exp\left((1-t)A_{e_3\times(u q+w)}\right)e_3\,dt
    \end{equation}
    We claim that \(f(0,0,\mathsf A_{\mathrm{crit}})=0\). This comes from the relation that
    \begin{multline*}
        \int_0^1\exp\left(tA_{e_3\times(u q+w)}\right)A_{e_3\times\psi}\exp\left((1-t)A_{e_3\times(u q+w)}\right)e_3\,dt \big|_{(q,w) = (0,0)} \\= (e_3 \times \psi) \times e_3 = |e_3|^2 \psi - (e_3\cdot \psi) e_3 = \psi,
    \end{multline*}
    where we used the triple cross product identity and the fact that \(\psi\) is orthogonal to \(e_3\). Therefore, we have
    \[
        f(0,0,\mathsf A_{\mathrm{crit}})(\psi) = \delta \mathcal{E}_{\mathrm{unif}}\left[e_3\right](\psi)= 0,
    \]
    since \(d_3=e_3\) is an equilibrium of the energy at \(\mathsf A_{\mathrm{crit}}\).
    We wish to apply the implicit function theorem to solve for \(w\) as a function of \(q\) and \(\mathsf A\). So we compute \(D_wf(0,0,\mathsf A_{\mathrm{crit}})\):
    \begin{multline*}
        D_wf(0,0,\mathsf A_{\mathrm{crit}})[h](\psi)
        =\left.\frac{d}{d\varepsilon}f(0,\varepsilon h,\mathsf A_{\mathrm{crit}})(\psi)\right|_{\varepsilon=0} \\
        =\left.\frac{d}{d\varepsilon}\delta \mathcal{E}_{\mathrm{unif}}\left[\exp(A_{e_3\times(\varepsilon h)})e_3\right]\left(\int_0^1\exp\left(tA_{e_3\times(\varepsilon h)}\right)A_{e_3\times\psi}\exp\left((1-t)A_{e_3\times(\varepsilon h)}\right)e_3\,dt\right)\right|_{\varepsilon=0} \\
        = D^2\mathcal{E}_{\mathrm{unif}}[e_3](h,\psi) + D\mathcal{E}_{\mathrm{unif}}[e_3]\left(\partial_\varepsilon \int_0^1\exp\left(tA_{e_3\times(\varepsilon h)}\right)A_{e_3\times\psi}\exp\left((1-t)A_{e_3\times(\varepsilon h)}\right)e_3\,dt\right)\bigg|_{\varepsilon=0},
    \end{multline*}
    where \(D\mathcal E_{\mathrm{unif}}\) and
    \(D^2\mathcal E_{\mathrm{unif}}\) denote the unconstrained Fréchet derivatives of the energy in the \(H^1\) topology. The second term does not vanish, since the derivative of the integral is not necessarily orthogonal to \(e_3\). To compute the second term, we calculate
    \begin{align*}
        & \partial_\varepsilon \int_0^1\exp\left(tA_{e_3\times(\varepsilon h)}\right)A_{e_3\times\psi}\exp\left((1-t)A_{e_3\times(\varepsilon h)}\right)e_3\,dt \\
        & \quad =\partial_\varepsilon \left[\left(\int_0^1
            \exp\left(\varepsilon tA_{e_3\times h}\right)A_{e_3\times\psi}\exp\left(-\varepsilon tA_{e_3\times h}\right)\;dt\right)\exp\left(\varepsilon A_{e_3\times h}\right)e_3\right] \\
        & \quad = \partial_\varepsilon \left[\left(\int_0^1A_{e_3\times\psi}+\varepsilon t[A_{e_3\times h},A_{e_3\times\psi}]+O(\varepsilon^2)\;dt\right)\exp\left(\varepsilon A_{e_3\times h}\right)e_3\right],
    \end{align*}
    where in the last line we used the Campbell Identity to expand the integrand. Taking the derivative and evaluating at \(\varepsilon=0\) gives
    \begin{align}
        & \partial_\varepsilon \int_0^1\exp\left(tA_{e_3\times(\varepsilon h)}\right)A_{e_3\times\psi}\exp\left((1-t)A_{e_3\times(\varepsilon h)}\right)e_3\,dt\big|_{\varepsilon=0} \nonumber \\
        & \quad = \left(\int_0^1 t[A_{e_3\times h},A_{e_3\times\psi}]\,dt\right)e_3 + A_{e_3\times\psi}\left[\int_0^1 \exp(0)A_{e_3\times h}\exp(0)\,dt\right]e_3 \nonumber \\
        & \quad = \frac12[A_{e_3\times h},A_{e_3\times\psi}]e_3 + A_{e_3\times\psi}A_{e_3\times h}e_3. = \frac{1}{2}\left(A_{e_3\times\psi}A_{e_3\times h}+A_{e_3\times h}A_{e_3\times\psi}\right)e_3 \nonumber \\
        & \quad = \frac{1}{2}\left((e_3 \times \psi)\times ((e_3 \times h)\times e_3) + (e_3 \times h)\times ((e_3 \times \psi)\times e_3)\right) \nonumber \\
        & \quad = \frac{1}{2}\left((e_3 \times \psi)\times h + (e_3 \times h)\times \psi\right) = \frac{1}{2}\left(-(\psi \cdot h)e_3 + (h\cdot e_3)\psi - (\psi \cdot h)e_3 + (\psi \cdot e_3)h \right) \nonumber \\
        & \quad = -(\psi \cdot h)e_3, \label{eq:uniform_campbell_derivative_1}
    \end{align}
    where we used the definition of the commutator, the triple cross product identity, and that \(\psi\) and \(h\) are orthogonal to \(e_3\). Finally, note that the unconstrained Frechet derivatives at \(e_3\) and \(\mathsf A_{\mathrm{crit}}\) are
    \begin{align}
        D\mathcal E_{\mathrm{unif}}[e_3][z]
        & =\int_0^1
        \bigl(-\mathsf A_{\mathrm{crit}}+\mathsf G(1-s)\bigr)
        e_3\cdot z\,ds, \label{eq:unif_first_derivative} \\
        D^2\mathcal E_{\mathrm{unif}}[e_3][h,\psi]
        & =\int_0^1\left[
            h'\cdot\psi'
            -\mathsf A_{\mathrm{crit}}(e_3\cdot h)(e_3\cdot\psi)
            \right]ds
        =\int_0^1h'\cdot\psi'\,ds, \label{eq:unif_second_derivative}
    \end{align}
    where in the last line we used the fact that \(h\) and \(\psi\) are orthogonal to \(e_3\). Thus, the end result of the full computation is
    \[
        D_wf(0,0,\mathsf A_{\mathrm{crit}})[h](\psi)
        =\int_0^1h'\cdot\psi'\,ds-\int_0^1\bigl(-\mathsf A_{\mathrm{crit}}+\mathsf G(1-s)\bigr)(\psi\cdot h)\,ds = \delta^2\mathcal E_{\mathrm{unif}}[e_3](h,\psi).
    \]
    Here we have abused notation and wrote \(\delta^2\mathcal E_{\mathrm{unif}}[e_3](h,\psi)\) to denote the symmetric bilinear form associated with the
    constrained second variation. Since \(L_{\mathrm{crit}}\) is self-adjoint with compact
    resolvent and zero is its lowest eigenvalue, its restriction to \(\operatorname{span}\{ue_1,ue_2\}^{\perp}\) has a strictly positive spectral gap. The associated bilinear form is therefore coercive in the \(H^1\) norm and boundedly invertible when restricted to this orthogonal complement.\\

    The implicit function theorem then gives a unique smooth function \(w=w(q,\mathsf A)\), defined near \((0,\mathsf A_{\mathrm{crit}})\), such that
    \[
        f(q,w(q,\mathsf A),\mathsf A)=0,
        \qquad w(0,\mathsf A_{\mathrm{crit}})=0.
    \]

    So, we can reduce the problem of finding equilibria to finding roots of the equation: \(R(q,\mathsf A)=0\), where
    \begin{equation}
        \label{eq:uniform_reduced_equation}
        R(q,\mathsf A):= \sum_{i=1}^2
        f(q,w(q,\mathsf A),\mathsf A)(ue_i)\,e_i.
    \end{equation}

    To determine the criticality of the bifurcation, we compute the Taylor expansion of \(R\) near \((0,\mathsf A_{\mathrm{crit}})\). We already have \(R(0,\mathsf A)=0\) since the straight configuration is an equilibrium for every \(\mathsf A\).

    We first compute the linear term in \(q\). Since \(f\)
    depends on \(q\) and \(w\) through \(\varphi=uq+w\),
    the preceding Campbell calculation applies to differentiation in \(q\) by taking the increment \(h=up\). Thus, for \(p\in\operatorname{span}\{e_1,e_2\}\) and any admissible test function \(\zeta \in \operatorname{span}\{e_1,e_2\}\),
    \begin{align*}
        D_qf(0,0,\mathsf A)[p](\zeta)
        & =
        \int_0^1
        \left[
            u'p\cdot\zeta'
            +\bigl(\mathsf A-\mathsf G(1-s)\bigr)
            up\cdot\zeta
            \right]\,ds \\
        & =
        (\mathsf A-\mathsf A_{\mathrm{crit}})
        p\cdot\int_0^1u\zeta\,ds,
    \end{align*}
    where we used the weak form of the eigenvalue equation for \(u\). In particular, note that this derivative vanishes whenever \(\zeta\) is orthogonal to the kernel. We use this fact to compute the dependence of \(w\) on \(q\).
    Fix \(\mathsf A\) near \(\mathsf A_{\mathrm{crit}}\)
    and \(p\in\operatorname{span}\{e_1,e_2\}\).
    By construction, the complementary equation holds along
    the curve \(q=\varepsilon p\):
    \[
        f(\varepsilon p,w(\varepsilon p,\mathsf A),
        \mathsf A)(\psi)=0
    \]
    for every admissible \(\psi\) orthogonal to the kernel.
    Differentiating with respect to \(\varepsilon\) at \(\varepsilon=0\) gives
    \[
        D_qf(0,0,\mathsf A)[p](\psi)
        +D_wf(0,0,\mathsf A)[D_qw(0,\mathsf A)[p]](\psi)=0.
    \]
    The first term vanishes since the test function \(\psi\) is orthogonal to the kernel. We also note that \(D_qw(0,\mathsf A)[p] \in \operatorname{span}(ue_1,ue_2)^\perp\). Thus, the second term is exactly the previously computed bilinear form associated with the constrained second variation. Since it is boundedly invertible, we conclude that \(D_qw(0,\mathsf A)=0\).\\

    We now can proceed with computing \(D_qR(0,\mathsf A)\). For \(p,z\in\operatorname{span}\{e_1,e_2\}\), we have
    \begin{align*}
        D_qR(0,\mathsf A)[p]\cdot z
        & =D_qf(0,0,\mathsf A)[p](uz)
        +D_wf(0,0,\mathsf A)[D_qw(0,\mathsf A)[p]](uz) \\
        & =D_qf(0,0,\mathsf A)[p](uz) = (\mathsf A-\mathsf A_{\mathrm{crit}})p\cdot z \int_0^1 u^2\;ds,
    \end{align*}
    where in the last line we used the fact that \(D_qw(0,\mathsf A)=0\) and the definition of \(D_qf\). Now we may normalize \(u\) so that \(\int_0^1 u^2\;ds=1\). Then we have
    \begin{equation} \label{eq:uniform_dqR}
        D_qR(0,\mathsf A)[p]\cdot z = (\mathsf A-\mathsf A_{\mathrm{crit}})p\cdot z \implies D_qR(0,\mathsf A)[p] = (\mathsf A-\mathsf A_{\mathrm{crit}})p
    \end{equation}
    In particular, we see that \(D_qR(0,\mathsf A_{\mathrm{crit}})=0\). Since \(R(0,\mathsf A)=0\) for every \(\mathsf A\) near \(\mathsf A_{\mathrm{crit}}\), we also have \(\partial_{\mathsf A}R(0,\mathsf A_{\mathrm{crit}})=0\). Thus both first-order Taylor coefficients vanish.\\

    Now we move to the second order coefficients. Again the same argument gives that the pure second derivative in \(\mathsf A\) vanishes. Furthermore, using the previous calculation, we can read off the mixed derivative:
    \[
        \partial_{\mathsf A}D_qR(0,\mathsf A_{\mathrm{crit}})[p] = p \implies \partial_{\mathsf A}D_qR(0,\mathsf A_{\mathrm{crit}}) = \mathrm{Id}.
    \]

    To compute \(D_q^2R(0,\mathsf A_{\mathrm{crit}})\), we use some symmetry properties. Let \(Q\) denote the rotation about \(e_3\), such that \(Qe_3=e_3\) and \(Q\varphi=-\varphi\) for \(\varphi\in\operatorname{span}\{e_1,e_2\}\). This can be realized by \(Q = \mathrm{diag}(-1,-1,1)\). Then note that the energy is invariant under this rotation as we have
    \[
        |(Qd_3)'|^2 = |d_3'|^2 \quad \text{and}\quad e_3 \cdot Q d_3 = (Q^T e_3)\cdot d_3 = e_3 \cdot d_3
    \]
    It also preserves the boundary conditions.
    Furthermore, the exponential parametrization satisfies
    \[
        Q\exp(A_{e_3\times\varphi})e_3 = Q \exp(A_{e_3\times\varphi}) Q^T e_3 = \exp(A_{Qe_3\times Q\varphi}) e_3 = \exp(A_{e_3\times(-\varphi)}) e_3.
    \]

    Thus, we may conclude that
    \[
        \mathcal E_{\mathrm{unif}}[\exp(A_{e_3\times(-uq-w+\varepsilon\zeta)})e_3]
        =\mathcal E_{\mathrm{unif}}[\exp(A_{e_3\times(uq+w-\varepsilon\zeta)})e_3].
    \]
    for any admissible test function \(\zeta\in\operatorname{span}\{e_1,e_2\}\). Differentiating with respect to \(\varepsilon\) at \(\varepsilon=0\) gives
    \[
        f(-q,-w,\mathsf A)(\zeta)=-f(q,w,\mathsf A)(\zeta).
    \]
    In particular, this means that if \(w(q,\mathsf A)\) solves the complementary equation at \(q\), then \(-w(q,\mathsf A)\) solves it at \(-q\). Uniqueness therefore gives \(w(-q,\mathsf A)=-w(q,\mathsf A)\). Substituting into the definition of \(R\), we obtain
    \[
        R(-q,\mathsf A)=\sum_{i=1}^2 f(-q,-w(q,\mathsf A),\mathsf A)(ue_i)e_i=-R(q,\mathsf A).
    \]
    Thus \(R\) is odd in \(q\), which ensures that \(D_q^2R(0,\mathsf A)=0\).\\

    Now we compute the third-order terms. A similar argument to that in the previous order shows that \(D^3_q R(0,\mathsf A)\) is the only one that may be nonzero.
    To compute it, we also need to note that since \(w\) is odd in \(q\), we have \(D^2_q w(0,\mathsf A)=0\). Thus, the only possibly nonzero term in the third derivatives are:
    \[
        D_q^3R(0,\mathsf A)[p,p,p]\cdot z
        =D_q^3f(0,0,\mathsf A)[p,p,p](uz)
        +D_wf(0,0,\mathsf A)[D_q^3w(0,\mathsf A)[p,p,p]](uz),
    \]

    Again we note that by differentiating the orthogonaloity condition, we have that the third order tensor~\(D_q^3w(0,\mathsf A)[p,p,p]\) is orthogonal to the kernel. Combining this with a similar computation to the one done for \(D_qf\) shows that \(D_wf(0,0,\mathsf A)[D_q^3w(0,\mathsf A)[p,p,p]](uz)= 0\). Thus, we have just have to compute the first term. We do it by chain rule. Note that the energy is quadratic and so the unconstrained Frechet derivatives vanish after the second derivative. To keep the algebra succinct, let us denote \(d_3(\tau,\varepsilon):=\exp(A_{e_3\times u(\tau p+\varepsilon z)})e_3\), and we denote derivatives by subscripts. Then we compute:
    \begin{align}
        \left.\partial_\tau^3\partial_\varepsilon\mathcal E_{\mathrm{unif}}[d_3]\right|_{\tau=\varepsilon=0}
        & =D\mathcal E_{\mathrm{unif}}[e_3][d_{3,\tau \tau \tau\varepsilon}]
        +D^2\mathcal E_{\mathrm{unif}}[e_3][d_{3,\tau\tau\tau},d_{3,\varepsilon}] \nonumber \\
        & \quad+3D^2\mathcal E_{\mathrm{unif}}[e_3][d_{3,\tau\tau},d_{3,\tau\varepsilon}]
        +3D^2\mathcal E_{\mathrm{unif}}[e_3][d_{3,\tau},d_{3,\tau\tau\varepsilon}], \label{eq:uniform_third_derivative_f}
    \end{align}
    where all director derivatives on the right are evaluated at \(\tau=\varepsilon=0\).

    Let us compute the director derivatives. Using \eqref{eq:uniform_exponential_derivative}, we have
    \begin{align*}
        d_{3,\varepsilon} & = \int_0^1 \exp(t A_{e_3\times(\tau up + \varepsilon uz)})A_{e_3\times u z}\exp((1-t)A_{e_3\times (\tau up + \varepsilon uz)})e_3\,dt \\
        & = \int_0^1 \exp(tA_{e_3\times(\tau up + \varepsilon uz)})A_{e_3\times u z}\exp(-tA_{e_3\times (\tau up + \varepsilon uz)})\,dt \exp(A_{e_3\times (\tau up + \varepsilon uz)})e_3 \\
        d_{3,\varepsilon}\big |_{\varepsilon=0} &= \int_0^1 \exp(t \tau A_{e_3\times up})A_{e_3\times u z}\exp(-t\tau A_{e_3\times up})\,dt \exp(\tau A_{e_3\times up})e_3,
    \end{align*}
    where in the last line we evaluated at \(\varepsilon=0\) since the remaining derivatives are in terms of \(\tau\).
    To deduce the rest of the derivatives, we use the Campbell identity again to expand the conjugation term in the integrand:
    \begin{align*}
        \left.d_{3,\varepsilon}\right|_{\varepsilon=0}
        & =\Bigl(A_{e_3\times uz}
        +\frac{\tau}{2}[A_{e_3\times up},A_{e_3\times uz}]
        +\frac{\tau^2}{6}[A_{e_3\times up},[A_{e_3\times up},A_{e_3\times uz}]] \\
        & \qquad+\frac{\tau^3}{24}[A_{e_3\times up},[A_{e_3\times up},[A_{e_3\times up},A_{e_3\times uz}]]]
        +O(\tau^4)\Bigr)\exp(\tau A_{e_3\times up})e_3.
    \end{align*}
    Differentiating with respect to \(\tau\) and evaluating at \(\tau=0\), we get:
    \begin{multline*}
        d_{3,\varepsilon}\big|_{\tau=\varepsilon=0} = A_{e_3\times uz}e_3, \\
        d_{3,\tau\varepsilon}\big|_{\tau=\varepsilon=0} = \frac{1}{2}[A_{e_3\times up},A_{e_3\times uz}]e_3 + A_{e_3\times uz}A_{e_3\times up}e_3 = \frac{1}{2}\left(A_{e_3\times up}A_{e_3\times uz} + A_{e_3\times uz}A_{e_3\times up}\right)e_3, \\
        d_{3,\tau\tau\varepsilon}\big|_{\tau=\varepsilon=0} = \frac{1}{3}[A_{e_3 \times up}, [A_{e_3\times up},A_{e_3\times uz}]]e_3 + [A_{e_3\times up},A_{e_3\times uz}]A_{e_3\times up}e_3 + A_{e_3\times uz}A_{e_3\times up}^2e_3, \\
        d_{3,\tau\tau\tau\varepsilon}\big|_{\tau=\varepsilon=0} = \frac{1}{4}[A_{e_3\times up},[A_{e_3\times up},[A_{e_3\times up},A_{e_3\times uz}]]]e_3 + [A_{e_3\times up},[A_{e_3\times up},A_{e_3\times uz}]]A_{e_3\times up}e_3 \\
        \qquad + \frac{3}{2}[A_{e_3\times up},A_{e_3\times uz}]A_{e_3\times up}^2e_3 + A_{e_3\times uz}A_{e_3\times up}^3e_3.
    \end{multline*}
    Let us simplify these expressions using the triple cross product identity and that \(p,z \in \operatorname{span}\{e_1,e_2\}\). This gives that
    \begin{equation} \label{eq:director_derivative_e}
        d_{3,\varepsilon}\big|_{\tau=\varepsilon=0} = A_{e_3\times uz}e_3 = u z.
    \end{equation}

    For the second derivative, we note that the expression is the same as the one we computed in \eqref{eq:uniform_campbell_derivative_1}, so we have
    \begin{equation} \label{eq:director_derivative_et}
        d_{3,\tau\varepsilon}\big|_{\tau=\varepsilon=0} = -u^2(p \cdot z)e_3,
    \end{equation}

    For the third derivative, we first expand the double commutator:
    \[
        [A_{e_3\times up},[A_{e_3\times up},A_{e_3\times uz}]]
        =A_{e_3\times up}^2A_{e_3\times uz}
        -2A_{e_3\times up}A_{e_3\times uz}A_{e_3\times up}
        +A_{e_3\times uz}A_{e_3\times up}^2.
    \]
    Substituting this identity and expanding the remaining commutator gives
    \begin{align*}
        d_{3,\tau\tau\varepsilon}\big|_{\tau=\varepsilon=0}
        & =\frac13\left(A_{e_3\times up}^2A_{e_3\times uz}-2A_{e_3\times up}A_{e_3\times uz}A_{e_3\times up}+A_{e_3\times uz}A_{e_3\times up}^2\right)e_3 \\
        & \quad+\left(A_{e_3\times up}A_{e_3\times uz}A_{e_3\times up}-A_{e_3\times uz}A_{e_3\times up}^2\right)e_3
        +A_{e_3\times uz}A_{e_3\times up}^2e_3 \\
        & =\frac13\left(A_{e_3\times up}^2A_{e_3\times uz}+A_{e_3\times up}A_{e_3\times uz}A_{e_3\times up}+A_{e_3\times uz}A_{e_3\times up}^2\right)e_3.
    \end{align*}
    We now compute the identity for \(h,\psi \in \operatorname{span}\{e_1,e_2\}\):
    \begin{equation} \label{eq:uniform_useful_crossproduct_identity}
        A_{e_3\times h}A_{e_3\times\psi}e_3= A_{e_3\times h}\psi = (e_3\cdot\psi)h-(h\cdot\psi)e_3 = -(h\cdot\psi)e_3.
    \end{equation}

    Applying this identity to the two rightmost factors of each product, yields
    \begin{equation} \label{eq:director_derivative_tte}
        d_{3,\tau\tau\varepsilon}\big|_{\tau=\varepsilon=0}
        =-\frac{u^2}{3}\left(2(p\cdot z)A_{e_3\times up}e_3+|p|^2A_{e_3\times uz}e_3\right)= -\frac{u^3}{3}\bigl(2(p\cdot z)p+|p|^2z\bigr).
    \end{equation}

    For the fourth derivative, we first expand the triple commutator:
    \begin{align*}
        & [A_{e_3\times up},[A_{e_3\times up},[A_{e_3\times up},A_{e_3\times uz}]]] \\
        & \qquad=A_{e_3\times up}^3A_{e_3\times uz}
        -3A_{e_3\times up}^2A_{e_3\times uz}A_{e_3\times up}
        +3A_{e_3\times up}A_{e_3\times uz}A_{e_3\times up}^2
        -A_{e_3\times uz}A_{e_3\times up}^3.
    \end{align*}
    Substituting this identity and expanding the remaining commutators gives
    \begin{multline*}
        d_{3,\tau\tau\tau\varepsilon}\big|_{\tau=\varepsilon=0}\\
        =\frac14\left(A_{e_3\times up}^3A_{e_3\times uz}-3A_{e_3\times up}^2A_{e_3\times uz}A_{e_3\times up}
        +3A_{e_3\times up}A_{e_3\times uz}A_{e_3\times up}^2-A_{e_3\times uz}A_{e_3\times up}^3\right)e_3 \\
        \quad+\left(A_{e_3\times up}^2A_{e_3\times uz}A_{e_3\times up}-2A_{e_3\times up}A_{e_3\times uz}A_{e_3\times up}^2+A_{e_3\times uz}A_{e_3\times up}^3\right)e_3 \\
        \quad+\frac32\left(A_{e_3\times up}A_{e_3\times uz}A_{e_3\times up}^2-A_{e_3\times uz}A_{e_3\times up}^3\right)e_3
        +A_{e_3\times uz}A_{e_3\times up}^3e_3 \\
        =\frac14\left(A_{e_3\times up}^3A_{e_3\times uz}+A_{e_3\times up}^2A_{e_3\times uz}A_{e_3\times up}
        +A_{e_3\times up}A_{e_3\times uz}A_{e_3\times up}^2+A_{e_3\times uz}A_{e_3\times up}^3\right)e_3.
    \end{multline*}
    Applying \eqref{eq:uniform_useful_crossproduct_identity} to the two rightmost factors of each product yields
    \[
        d_{3,\tau\tau\tau\varepsilon}\big|_{\tau=\varepsilon=0}
        =-\frac{u^2}{4}\left(2(p\cdot z)A_{e_3\times up}^2
        +|p|^2\left(A_{e_3\times up}A_{e_3\times uz}+A_{e_3\times uz}A_{e_3\times up}\right)\right)e_3.
    \]
    Applying the same identity once more, we obtain
    \begin{equation}\label{eq:director_derivative_ettt}
        d_{3,\tau\tau\tau\varepsilon}\big|_{\tau=\varepsilon=0}
        =u^4|p|^2(p\cdot z)e_3.
    \end{equation}

    To evaluate the Frechet derivatives, we also need the pure \(\tau\)-derivatives. These follow from the preceding computations by setting \(z=p\), since then the role of \(\varepsilon\) is identical to that of \(\tau\).
    Thus, substituting \(z=p\) into the previously computed expressions gives
    \begin{equation}\label{eq:director_derivatives_tau}
        d_{3,\tau}\big|_{\tau=\varepsilon=0}=up,\qquad
        d_{3,\tau\tau}\big|_{\tau=\varepsilon=0}=-u^2|p|^2e_3,\qquad
        d_{3,\tau\tau\tau}\big|_{\tau=\varepsilon=0}=-u^3|p|^2p.
    \end{equation}

    Thus, we have computed all the necessary director derivatives to evaluate \eqref{eq:uniform_third_derivative_f}. Substituting \eqref{eq:director_derivative_e}, \eqref{eq:director_derivative_et}, \eqref{eq:director_derivative_tte}, \eqref{eq:director_derivative_ettt}, and \eqref{eq:director_derivatives_tau} into \eqref{eq:uniform_third_derivative_f} gives
    \begin{multline*}
        D_q^3R(0,\mathsf A_{\mathrm{crit}})[p,p,p]\cdot z
        =\left.\partial_\tau^3\partial_\varepsilon\mathcal E_{\mathrm{unif}}[d_3]\right|_{\tau=\varepsilon=0} \\
        =D\mathcal E_{\mathrm{unif}}[e_3][u^4|p|^2(p\cdot z)e_3]
        +D^2\mathcal E_{\mathrm{unif}}[e_3][-u^3|p|^2p,uz] \\
        +3D^2\mathcal E_{\mathrm{unif}}[e_3][-u^2|p|^2e_3,-u^2(p\cdot z)e_3]
        -D^2\mathcal E_{\mathrm{unif}}[e_3][up,u^3(|p|^2z+2(p\cdot z)p)], \\
        =|p|^2(p\cdot z)\int_0^1\bigl(-\mathsf A_{\mathrm{crit}}+\mathsf G(1-s)\bigr)u^4\,ds
        -3|p|^2(p\cdot z)\int_0^1u^2(u')^2\,ds \\
        \quad+3|p|^2(p\cdot z)\int_0^1\bigl(4u^2(u')^2-\mathsf A_{\mathrm{crit}}u^4\bigr)\,ds
        -9|p|^2(p\cdot z)\int_0^1u^2(u')^2\,ds \\
        =|p|^2(p\cdot z)\int_0^1\bigl(\mathsf G(1-s)-4\mathsf A_{\mathrm{crit}}\bigr)u^4\,ds
    \end{multline*}
    Set
    \begin{equation}\label{eq:uniform_bifurcation_coefficient}
        \mathsf A_2:=\frac16\int_0^1\bigl(4\mathsf A_{\mathrm{crit}}-\mathsf G(1-s)\bigr)u^4\,ds.
    \end{equation}
    The Taylor expansion then becomes
    \begin{equation}\label{eq:uniform_reduced_taylor}
        R(q,\mathsf A)=(\mathsf A-\mathsf A_{\mathrm{crit}})q-\mathsf A_2|q|^2q
        +O\bigl(|\mathsf A-\mathsf A_{\mathrm{crit}}||q|^3+|q|^5\bigr).
    \end{equation}

    To characterize the roots, we use the symmetry of the problem. A similar argument as used before shows that the energy is invariant under every orthogonal transformation \(Q\) fixing \(e_3\), and the same uniqueness argument used above therefore gives \(w(Qq,\mathsf A)=Qw(q,\mathsf A)\), which means \(R(Qq,\mathsf A)=QR(q,\mathsf A)\). For \(q\ne0\), choose \(Q\) such that \(Qe_3=e_3\), \(Qq=q\), and \(Q(e_3\times q)=-(e_3\times q)\). Since \(q,e_3 \times q\) form a orthogonal basis for \(\operatorname{span}\{e_1,e_2\}\), and thus form a basis for \(R(q,\mathsf A)\), we conclude that \(R(q,\mathsf A)\) must be parallel to \(q\).\\

    Now we reduce the problem to a scalar equation. Fix a unit vector \(p\in\operatorname{span}\{e_1,e_2\}\) and write \(q=rp\), where \(r\in\mathbb R\). To apply the implicit function theorem, define the smooth scalar function
    \[
        g(r,\mathsf A):=\int_0^1D_qR(trp,\mathsf A)[p]\cdot p\,dt.
    \]
    Since \(R(0,\mathsf A)=0\), the fundamental theorem of calculus gives \(R(rp,\mathsf A)\cdot p=r\,g(r,\mathsf A)\). Since we have shown that \(R\) is parallel to \(q\), for \(r\ne0\), we have that\(R(rp,\mathsf A)=0\) is equivalent to \(g(r,\mathsf A)=0\). This g is now the reduced scalar equation that we will analyze to determine the character of the bifurcation branch.\\

    On the other hand, plugging in \eqref{eq:uniform_dqR} gives
    \[
        g(0,\mathsf A)=D_qR(0,\mathsf A)[p]\cdot p=\mathsf A-\mathsf A_{\mathrm{crit}}
    \]

    In particular, this implies that \(g(0,\mathsf A_{\mathrm{crit}})=0\) and \(\partial_{\mathsf A}g(0,\mathsf A_{\mathrm{crit}})=1\). Thus, the implicit function theorem applies and gives a unique smooth function \(\mathsf A(r)\), with \(\mathsf A(0)=\mathsf A_{\mathrm{crit}}\), such that \(g(r,\mathsf A(r))=0\). Furthermore by differentiating this identity, we have
    \[
        0=\frac{d}{dr}g(r,\mathsf A(r))\big|_{r=0} = \partial_rg(0,\mathsf{A}_\mathrm{crit}) + \partial_\mathsf{A}g(0,\mathsf{A}_\mathrm{crit})\mathsf{A}'(0)= \mathsf A'(0),
    \]
    where in the last equality, we used that \(D^2_qR(0,\mathsf A_{\mathrm{crit}})[p,p]=0\). For the next order term, we differentiate again and evaluate at \(r=0\):
    \[
        0=\frac{d^2}{dr^2}g(r,\mathsf A(r))\big|_{r=0} = \partial_r^2g(0,\mathsf{A}_\mathrm{crit}) + \partial_\mathsf{A}g(0,\mathsf{A}_\mathrm{crit})\mathsf{A}''(0) = \partial_r^2g(0,\mathsf{A}_\mathrm{crit}) + \mathsf{A}''(0).
    \]

    To compute the second order derivative, we use the definition of \(g\) and the Taylor expansion \eqref{eq:uniform_reduced_taylor} to obtain
    \[
        g(r,\mathsf{A}_\mathrm{crit}) = R(rp,\mathsf{A}_\mathrm{crit})\cdot p/r = -\mathsf A_2 r^2 + O(r^4) \implies \partial_r^2g(0,\mathsf{A}_\mathrm{crit}) = -2\mathsf A_2.
    \]
    This gives us the Taylor expansion of \(\mathsf A(r)\) near \(r=0\):
    \[
        \mathsf A(r)=\mathsf A_{\mathrm{crit}}+\mathsf A_2r^2+o(r^2).
    \]

    Thus, for \(\mathsf A_2 \neq 0\), this is a pitchfork bifurcation. We check the criticality of the bifurcation by considering the sign of \(\mathsf A_2\). We first test the critical eigenvalue equation \(L_{\mathrm{crit}}u=0\) from \eqref{eq:uniform_critical_eigenvalue_problem} with \(u^3\). Integration by parts, using the boundary conditions, gives
    \[
        3\int_0^1u^2(u')^2\,ds+\mathsf A_{\mathrm{crit}}\int_0^1u^4\,ds
        =\mathsf G\int_0^1(1-s)u^4\,ds.
    \]
    Plugging this into \eqref{eq:uniform_bifurcation_coefficient} gives a useful simplified formula for \(\mathsf A_2\):
    \begin{equation}\label{eq:uniform_bifurcation_coefficient_simplified}
        \mathsf A_2=\frac12\left(\mathsf A_{\mathrm{crit}}\int_0^1u^4\,ds-\int_0^1u^2(u')^2\,ds\right) = \frac12\left(\int_0^1u^2 [\mathsf A_{\mathrm{crit}}u^2-(u')^2]\,ds\right).
    \end{equation}

    Now we specialize to the two cases. First, we consider the clamped case. Note that for \(G = G_{\mathrm{crit}}\), we have \(\mathsf A_{\mathrm{crit}}=0\). Thus, \eqref{eq:uniform_bifurcation_coefficient_simplified} gives that \(\mathsf A_2<0\). By continuity, \(\mathsf A_2<0\) also for \(\mathsf G>\mathsf G_{\mathrm{crit}}\) sufficiently close to the threshold. On the other hand, \eqref{eq:uniform_bifurcation_coefficient} shows that \(\mathsf A_2>0\) whenever \(4\mathsf A_{\mathrm{crit}}>\mathsf G\). The clamped eigenvalue bound gives using the bound from \Cref{thm:uniform_eigenvalue_bounds} that
    \[
        \mathsf A_{\mathrm{crit}}=-\beta_0^{\mathrm{cl}}(\mathsf G)
        \geq \mathsf G-|a_1|\mathsf G^{2/3}.
    \]
    Thus, we have that for sufficiently large \(\mathsf G\), \(4\mathsf A_{\mathrm{crit}}>\mathsf G\) and therefore \(\mathsf A_2>0\). By continuity, \(\mathsf A_2\) must vanish at least once between these regimes. Thus, the clamped bifurcation is sign-indefinite and can change character from supercritical to subcritical.\\

    Now we consider the pinned case. Without loss of generality, we may assume that the principal eigenfunction \(u\) is positive. Again it is easy to see from \eqref{eq:uniform_bifurcation_coefficient} that if \(\mathsf A_{\mathrm{crit}}\geq \frac{\mathsf G}{4}\), then \(\mathsf A_2>0\). So we can just check the other case, \(\mathsf A_{\mathrm{crit}}<\frac{\mathsf G}{4}\). In this case, we can use the equation to note that since \(u'' = (\mathsf A_{\mathrm{crit}}-\mathsf G(1-s))u\), we have that \(u''<0\) for \(s<1-\frac{\mathsf A_{\mathrm{crit}}}{\mathsf G}\) and \(u''>0\) for \(s>1-\frac{\mathsf A_{\mathrm{crit}}}{\mathsf G}\). Since the pinned boundary conditions give \(u'(0)=u'(1)=0\), we have that \(u'<0\) on \((0,1)\). Using this, we compute the sign of \((u')^2-\mathsf A_{\mathrm{crit}}u^2\).
    \[
        \frac{1}{2}\frac{d}{ds}\left((u')^2-\mathsf A_{\mathrm{crit}}u^2\right) = u'u''-\mathsf A_{\mathrm{crit}}uu' = -\mathsf G(1-s)uu'\geq0,
    \]
    where in the last equality, we used the eigenvalue equation.
    So this quantity is increasing and at \(s=1\) it is negative.
    This means that \((u')^2-\mathsf A_{\mathrm{crit}}u^2<0\) for all \(s\in[0,1]\). Substituting this into \eqref{eq:uniform_bifurcation_coefficient_simplified} gives that \(\mathsf A_2>0\). Thus, the pinned bifurcation is always subcritical.

\end{appendixproof}

\begin{appendixproof}[Proof of \Cref{thm:stabilization}] \label{pf:stabilization}
    For either boundary condition, \( \beta_0^{\mathrm{bc}}(\mathsf G)\geq-\mathsf G. \) Thus, by \Cref{thm:uniform_straight_stability}, \( \mathsf A>\mathsf G \) is sufficient for strict stability.
    For a circular wire,
    \[
        \mathsf G=\frac{4\rho gL^3}{Eb^2},\qquad \mathsf A=\frac{4K_dL^2}{Eb^2},
    \]
    so rearranging gives
    \[
        L<\frac{K_d}{\rho g}=:L_{\mathrm{stab}}.
    \]
    The resulting sufficient stability length is therefore independent of both \(b\) and \(E\). \\
    We have shown that \(\beta_0^{\mathrm{cl}}(\mathsf G)<-\mathsf G/4\) is sufficient for subcriticality of the bifurcation. At the critical length, \(L_{\mathrm{crit}}\),
    \(\beta_0^{\mathrm{cl}}(\mathsf G)=-\mathsf A\),
    so this condition is equivalent to
    \[
        4\mathsf A >\mathsf G
        \quad\Longleftrightarrow\quad
        4L_{\mathrm{stab}}>L_{\mathrm{crit}}.
    \]

    On the other hand, we may also use the upper bound on \(\beta_0^{\mathrm{cl}}\) and rearrange the terms to get
    \[
        \frac{4\rho gL_{\mathrm{crit}}^3}{Eb^2}
        <\frac{4K_dL_{\mathrm{crit}}^2}{Eb^2}
        +|a_1|\left(\frac{4\rho g}{Eb^2}\right)^{2/3}L_{\mathrm{crit}}^2.
    \]
    Dividing by \(4\rho gL_{\mathrm{crit}}^2/(Eb^2)>0\) yields
    \[
        L_{\mathrm{crit}}
        <L_{\mathrm{stab}}+|a_1|\left(\frac{Eb^2}{4\rho g}\right)^{1/3}.
    \]

    Thus, a sufficient condition in terms of the lengths for the bifurcation to be subcritical is
    \[
        L_{\mathrm{stab}}+|a_1|\left(\frac{Eb^2}{4\rho g}\right)^{1/3} < 4L_{\mathrm{stab}} \implies L_{\mathrm{stab}} > \frac{|a_1|}{3}\left(\frac{Eb^2}{4\rho g}\right)^{1/3}
    \]

\end{appendixproof}

\section{Proofs for Permanent Magnet}
\begin{appendixproof}[Proof of \Cref{thm:planar_equilibria}]\label{pf:planar_equilibria}
    Given \(\theta\in H^1((0,1))\) with the clamp \(\theta(0)=0\) define
    \[
        d_\theta:=\sin\theta\,e_1+\cos\theta\,e_3,
        \qquad
        X_\theta:=\int_0^1 d_\theta\,ds.
    \]
    Then \(d_\theta(0)=e_3\), \(|d_\theta|=1\), and
    \[
        |d_\theta'|^2=|\theta'|^2.
    \]
    The restriction of \(\mathcal E_{\mathrm{perm}}\) to planar configurations can be written in terms of \(\theta\) as
    \[
        \mathcal E_{\mathrm{pl}}[\theta]
        =
        \frac12\int_0^1|\theta'|^2\,ds
        +
        \mathsf G\int_0^1(1-s)\cos\theta\,ds
        +
        \mathsf M\Phi(X_\theta).
    \]

    By direct method of calculus of variations, \(\mathcal E_{\mathrm{pl}}\) has a minimizer.
    Furthermore, setting
    \[
        n_\theta:=\cos\theta\,e_1-\sin\theta\,e_3.
    \]
    we can compute the first variation of \(\mathcal E_{\mathrm{pl}}\) with respect to \(\theta\). For \(h\in H^1((0,1))\) with \(h(0)=0\),
    \[
        \left.\frac{d}{d\varepsilon}\right|_{\varepsilon=0}
        d_{\theta+\varepsilon h}
        =
        h\,n_\theta,
        \qquad
        \left.\frac{d}{d\varepsilon}\right|_{\varepsilon=0}
        X_{\theta+\varepsilon h}
        =
        \int_0^1h\,n_\theta\,ds,
    \]
    and gives the Euler-Lagrange equation for planar equilibria
    \[
        -\theta''
        -\mathsf G(1-s)\sin\theta
        +
        \mathsf M\nabla\Phi(X_\theta)\cdot n_\theta
        =0,
    \]
    with
    \[
        \theta(0)=0,
        \qquad
        \theta'(1)=0.
    \]

    It remains to show that every solution of the planar equation is an equilibrium of the full three-dimensional problem. Using the fact that \(p\in\operatorname{span}\{e_1,e_3\}\) and \(X_\theta \in\operatorname{span}\{e_1,e_3\}\), we can directly compute that \(                                                 \nabla\Phi(X)\in\operatorname{span}\{e_1,e_3\}.\)
    Furthermore, we also have for \(d_3 := d_\theta\) that \(d_3''\in\operatorname{span}\{e_1,e_3\}\) and thus we can conclude that the force \(F\) lies in the plane spanned by \(d_3\) and \(n_\theta\) with
    \[
        F:=
        -d_3''
        +\mathsf G(1-s)e_3
        +\mathsf M\nabla\Phi(X_\theta)
    \]

    Using the relations,
    \[
        d_3'=\theta'n_\theta,
        \qquad
        d_3''=\theta''n_\theta-(\theta')^2d_3,
    \]
    we obtain
    \[
        F\cdot n_\theta
        =
        -\theta''
        -\mathsf G(1-s)\sin\theta
        +
        \mathsf M\nabla\Phi(X_\theta)\cdot n_\theta
        =0,
    \]
    where in the last equality we used the planar Euler-Lagrange equation.
    Thus, \(F\) is parallel to \(d_3\), which means that \(d_3 \times F =0\) which is exactly the interior equation in \eqref{eq:permanent_EL_system}. Moreover, we have
    \[
        d_3(0)=e_3,
        \qquad
        d_3'(1)=\theta'(1)n_{\theta(1)}=0.
    \]
    Thus \(d_3 = d_\theta\) satisfies the full three-dimensional Euler-Lagrange system.\\
    We leave the verification that \eqref{eq:permanent_EL_system} is the full three-dimensional Euler-Lagrange system to the reader. The proof is similar to those done before.
\end{appendixproof}

\begin{appendixproof}[Proof of \Cref{lem:mode_decomposition}]\label{pf:mode_decomposition}
    Under the assumption
    \[
        \frac{2\mathsf M}{\delta^3}>\mathsf G,
    \]
    it is easy to see that \(B_0\) is a coercive bilinear form on \(V\). Using the direct method and strict convexity, we can find an unique minimizer \(\psi\) in the space
    \[
        \left\{\phi\in V:\int_0^1\phi\,ds=1\right\}.
    \]
    The Euler-Lagrange equation for this minimization problem is precisely the orthogonality condition \eqref{eq:psi_B0_orthogonality}.

    Now let \(\theta\in V\), set
    \[
        q:=\int_0^1\theta\,ds,
        \qquad
        w:=\theta-q\psi.
    \]
    Since \(\int_0^1\psi\,ds=1\), we have \(w\in Y\), proving \eqref{eq:theta_mode_decomposition}.
    Finally, if \(v\in Y\), the nonlocal term in \eqref{eq:B0_def} vanishes, giving
    \[
        B_0(v,v)
        =
        \int_0^1\left((v')^2+a_0v^2\right)ds
        \geq
        \int_0^1(v')^2\,ds,
    \]
    which proves \eqref{eq:B0_Y_coercivity}.

    For this decomposition, we have the following useful identities:
    \begin{equation}
        \label{eq:B0_projected_identity}
        B_0(\theta,v)=B_0(w,v)
        \qquad
        \text{for every }v\in Y,
    \end{equation}
    and
    \begin{equation}
        \label{eq:B0_pythagorean}
        B_0(\theta,\theta)
        =
        q^2B_0(\psi,\psi)+B_0(w,w).
    \end{equation}
    For \(v\in Y\), \eqref{eq:psi_B0_orthogonality} gives
    \[
        B_0(\theta,v)
        =
        qB_0(\psi,v)+B_0(w,v)
        =
        B_0(w,v),
    \]
    which proves \eqref{eq:B0_projected_identity}. Finally, using that \(B_0(\psi,w)=0\),
    \[
        B_0(\theta,\theta)
        =
        B_0(q\psi+w,q\psi+w)
        =
        q^2B_0(\psi,\psi)+B_0(w,w),
    \]
    which is \eqref{eq:B0_pythagorean}.
\end{appendixproof}
\begin{appendixproof}[Proof of \Cref{lem:mode_properties}]\label{pf:mode_properties}
    Using the Lagrange multiplier method, the minimizer \(\psi\) of \(B_0\) under the constraint \(\int_0^1\psi\,ds=1\) satisfies \eqref{eq:psi_BVP} for some \(\lambda_\psi\in\mathbb R\).
    Taking \(v=\psi\) and using \(\int_0^1\psi\,ds=1\) gives \eqref{eq:lambda_psi}. Since \(a_0>0\), we have \(\lambda_\psi>0\).

    To prove positivity, let \(\psi_-:=\max\{-\psi,0\}\). Testing \eqref{eq:psi_BVP} with \(\psi_-\) gives
    \[
        -\int_0^1\left((\psi_-')^2+a_0\psi_-^2\right)ds
        =
        \lambda_\psi\int_0^1\psi_-\,ds.
    \]

    The left-hand side is nonpositive and the right-hand side is nonnegative, so \(\psi_-\equiv0\) almost everywhere. On the other hand, the continuity of \(\psi\) implies that \(\psi_- \equiv 0\) everywhere. Thus \(\psi \geq0\). Since \(\lambda_\psi>0\), the strong maximum principle gives \eqref{eq:psi_positive}.

    Using \(\int_0^1\psi\,ds=1\), we can directly calculate that
    \[
        \|\psi-1\|_{L^2}^2
        =
        \int_0^1\psi^2\,ds
        -2\int_0^1\psi\,ds+1
        =
        \|\psi\|_{L^2}^2-1,
    \]
    which proves \eqref{eq:psi_mean_identity}.

    The positivity of \(\mu_\perp\) and \eqref{eq:Y_poincare} follow from the Poincar\'e inequality on \(Y\).
\end{appendixproof}

\begin{appendixproof}[Proof of \Cref{thm:mode_accuracy}]\label{pf:mode_accuracy}

    We define the endpoint of the one-mode configuration and the corresponding gradient of the dipole potential by
    \[
        X_q
        :=
        \int_0^1
        d_{q} ds,
        \quad
        g_{q,\alpha}:=\nabla\Phi_\alpha(X_q).
    \]
    Set
    \begin{equation}
        \label{eq:rho_qalpha_def}
        \rho_{q,\alpha}
        :=
        \mathsf G(1-s)\bigl(q\psi-\sin(q\psi)\bigr)
        +\mathsf M g_{q,\alpha}\cdot n_{q}
        -\frac{2\mathsf M}{\delta^3}q\psi,
    \end{equation}
    and
    \begin{equation}
        \label{eq:eta_qalpha_def}
        \eta_{q,\alpha}:=\|\rho_{q,\alpha}\|_{Y^*},
    \end{equation}
    where \(\|\cdot\|_{Y^*}\) denotes the dual norm induced by \(\|\cdot\|_Y\).
    Define the following quantities to simplify the rest of the bounds in the proof.
    \begin{equation}
        \label{eq:S_qalpha_def}
        S_{q,\alpha}
        :=
        |q|\sqrt{\|\psi\|_{L^2}^2-1}
        +\frac{2\eta_{q,\alpha}}{\sqrt{\mu_\perp}},
    \end{equation}
    \begin{equation}
        \label{eq:E_Delta_qalpha_def}
        E_{q,\alpha}
        :=
        \frac{2\eta_{q,\alpha}}{\sqrt{\mu_\perp}}
        \left(
        |q|\sqrt{\|\psi\|_{L^2}^2-1}
        +\frac{\eta_{q,\alpha}}{\sqrt{\mu_\perp}}
        \right),
        \qquad
        \Delta_{q,\alpha}
        :=
        |X_q-\mathcal Dp(\alpha)|-E_{q,\alpha},
    \end{equation}
    \begin{equation}
        \label{eq:H_qalpha_def}
        H_{q,\alpha}
        :=
        \left\|
        \mathsf G(1-s)\bigl(1-\cos(q\psi)\bigr)
        -\mathsf M g_{q,\alpha}\cdot d_{q}
        -\frac{2\mathsf M}{\delta^3}
        \right\|_{L^\infty},
    \end{equation}
    \begin{equation}
        \label{eq:K_qalpha_def}
        K_{q,\alpha}
        :=
        \frac{1}{\mu_\perp}
        \left[
            H_{q,\alpha}
            +
            2\bigl(\mathsf G+\mathsf M|g_{q,\alpha}|\bigr)\eta_{q,\alpha}
            +
            \frac{6\mathsf M}{\Delta_{q,\alpha}^4}
            \bigl(E_{q,\alpha}+S_{q,\alpha}^2\bigr)
            \right].
    \end{equation}
    Fix \(q\) and \(\alpha\), and for any \(w\in Y\). Set
    \[
        \theta:=q\psi+w,
        \qquad
        X:=\int_0^1d_\theta\,ds.
    \]
    Rewriting the first variation of the planar energy using \eqref{eq:B0_projected_identity}, we have that for every \(v\in Y\),
    \begin{align}
        \delta\mathcal E_{\mathrm{pl},\alpha}[\theta](v)
        & =
        \int_0^1\theta'v'\,ds
        -\mathsf G\int_0^1(1-s)\sin\theta\,v\,ds
        +\mathsf M\nabla\Phi_\alpha(X)\cdot\int_0^1n_\theta v\,ds \notag \\
        & =
        B_0(w,v)+N_{q,\alpha}(w)[v],
        \label{eq:projected_residual_identity}
    \end{align}
    where
    \begin{equation}
        \label{eq:N_qalpha_def}
        N_{q,\alpha}(w)[v]
        :=
        \int_0^1
        \left[
            \mathsf G(1-s)(\theta-\sin\theta)
            +\mathsf M\nabla\Phi_\alpha(X)\cdot n_\theta
            -\frac{2\mathsf M}{\delta^3}\theta
            \right]v\,ds.
    \end{equation}

    Since \(B_0\) is coercive on \(Y\),
    the Lax--Milgram theorem gives an inverse \(B_0^{-1}:Y^*\to Y\). We use it to define a map \(T_{q,\alpha}:Y\to Y\) that will be used to construct the mean-zero correction \(w\).
    Define
    \[
        T_{q,\alpha}(w):=-B_0^{-1}N_{q,\alpha}(w).
    \]
    Then note that by rearranging \eqref{eq:projected_residual_identity}, we can see that
    \eqref{eq:projected_equilibrium} is equivalent to the fixed-point
    equation
    \[
        w=T_{q,\alpha}(w).
    \]

    Thus, the goal of the rest of the proof is to set up a Banach fixed-point argument to show the existence of a unique fixed point with the bounds that we desire.

    We define the closed ball
    \[
        \mathcal B_{q,\alpha}
        :=
        \left\{
        w\in Y:\|w\|_Y\leq2\eta_{q,\alpha}
        \right\}.
    \]
    We first show that \(T_{q,\alpha}\) maps \(\mathcal B_{q,\alpha}\) into itself.

    We do this by estimating \(T_{q,\alpha}(0)\). Let
    \[
        z:=T_{q,\alpha}(0).
    \]
    At \(w=0\), \(\theta=q\psi\) and \(X=X_q\), so by
    \eqref{eq:rho_qalpha_def}, we have
    \begin{equation}
        \label{eq:N0_rho}
        N_{q,\alpha}(0)[v]
        =
        \int_0^1\rho_{q,\alpha}v\,ds.
    \end{equation}

    So using this relation \eqref{eq:N0_rho}, we have that \(z\) satisfies
    \[
        B_0(z,v)=-\int_0^1\rho_{q,\alpha}v\,ds
        \qquad
        \text{for every }v\in Y.
    \]
    Taking \(v=z\), using coercivity of \(B_0\), and then the definition
    of the dual norm gives
    \begin{align}
        \|z\|_Y^2
        & \leq B_0(z,z) \notag                 \\
        & =
        -\int_0^1\rho_{q,\alpha}z\,ds \notag    \\
        & \leq
        \|\rho_{q,\alpha}\|_{Y^*}\|z\|_Y \notag \\
        & =
        \eta_{q,\alpha}\|z\|_Y.
    \end{align}
    Therefore, we have the bound
    \begin{equation}
        \label{eq:T0_bound}
        \|T_{q,\alpha}(0)\|_Y\leq\eta_{q,\alpha}.
    \end{equation}

    Now we claim that we can establish that
    \begin{equation}
        \label{eq:T_contraction}
        \|DT_{q,\alpha}(w)\|_{\mathcal L(Y)}\leq K_{q,\alpha}
        \qquad
        \text{for every }w\in\mathcal B_{q,\alpha}.
    \end{equation}

    Note that if \eqref{eq:T_contraction} holds, then by the mean value theorem, we have that for every \(w\in\mathcal B_{q,\alpha}\),
    \begin{align}
        \|T_{q,\alpha}(w)\|_Y
        & \leq
        \|T_{q,\alpha}(0)\|_Y+\|T_{q,\alpha}(w)-T_{q,\alpha}(0)\|_Y \notag \\
        & \leq
        \eta_{q,\alpha}+K_{q,\alpha}\|w\|_Y \notag                         \\
        & \leq
        \eta_{q,\alpha}+2K_{q,\alpha}\eta_{q,\alpha} \notag                \\
        & =
        (1+2K_{q,\alpha})\eta_{q,\alpha}
        <
        2\eta_{q,\alpha},
    \end{align}
    where the last inequality follows from the assumption \(K_{q,\alpha}<1/2\). This shows that \(T_{q,\alpha}\) maps \(\mathcal B_{q,\alpha}\) into itself and is a contraction, so the Banach fixed-point theorem gives a unique fixed point \(w\in\mathcal B_{q,\alpha}\) satisfying \eqref{eq:projected_equilibrium}. \\

    Thus, we just need to establish \eqref{eq:T_contraction}. To do this, we first compute the derivative of \(T_{q,\alpha}\). Let \(w\in\mathcal B_{q,\alpha}\) and \(h\in Y\). Differentiating
    \[
        T_{q,\alpha}(w)=-B_0^{-1}N_{q,\alpha}(w),
    \]
    we obtain the relation
    \begin{equation}
        \label{eq:DT_equation}
        B_0(DT_{q,\alpha}(w)[h],v)
        =
        -DN_{q,\alpha}(w)[h](v)
        \qquad
        \text{for every }v\in Y.
    \end{equation}
    Therefore, to prove \eqref{eq:T_contraction}, it is enough to obtain an appropriate bound on \(DN_{q,\alpha}(w)\).

    We first compute this derivative explicitly. Since
    \[
        \theta=q\psi+w,
        \qquad
        X_\theta=\int_0^1d_\theta\,ds,
    \]
    a perturbation \(w+th\) gives
    \[
        \left.\frac{d}{dt}\right|_{t=0}\theta=h,
        \quad
        \left.\frac{d}{dt}\right|_{t=0}X_\theta
        =
        \int_0^1n_\theta h\,ds, \quad \left.\frac{d}{dt}\right|_{t=0}n_{\theta+th}
        =
        -d_\theta h.
    \]

    Differentiating \eqref{eq:N_qalpha_def}, we therefore obtain
    \begin{align}
        DN_{q,\alpha}(w)[h](v)
        ={} &
        \int_0^1
        \left[
            \mathsf G(1-s)(1-\cos\theta)
            -\mathsf M\nabla\Phi_\alpha(X_\theta)\cdot d_\theta
            -\frac{2\mathsf M}{\delta^3}
            \right]hv\,ds
        \notag  \\
        & +
        \mathsf M\nabla^2\Phi_\alpha(X_\theta)
        \left[
            \int_0^1n_\theta h\,ds,
            \int_0^1n_\theta v\,ds
            \right].
        \label{eq:DN_qalpha_formula}
    \end{align}
    We estimate the two terms in \eqref{eq:DN_qalpha_formula} separately. Before doing so, we need to establish some geometric bounds that hold for every
    \(w\in\mathcal B_{q,\alpha}\). First, by \eqref{eq:Y_poincare}, we note that the \(L^2\) norm of \(w\) is controlled:
    \begin{align}
        \|w\|_{L^2}
        & \leq
        \frac{1}{\sqrt{\mu_\perp}}\|w\|_Y
        \notag  \\
        & \leq
        \frac{2\eta_{q,\alpha}}{\sqrt{\mu_\perp}}.
        \label{eq:w_L2_ball_bound}
    \end{align}
    Since \(w(0)=0\), we also have the \(L^\infty\) bound
    \begin{align}
        |w(s)|
        & =
        \left|\int_0^s w'(t)\,dt\right|
        \notag  \\
        & \leq
        \|w'\|_{L^2} \notag\\
        \implies \label{eq:w_Linfty_ball_bound}
        \|w\|_{L^\infty}&\leq2\eta_{q,\alpha}.
    \end{align}

    Next, we will compare \(\theta\) to the constant angle \(q\). Since
    \[
        \theta-q=q(\psi-1)+w,
    \]
    using \eqref{eq:psi_mean_identity} and \eqref{eq:w_L2_ball_bound} gives
    \begin{align}
        \|\theta-q\|_{L^2}
        & \leq
        |q|\|\psi-1\|_{L^2}+\|w\|_{L^2}
        \notag  \\
        & \leq
        |q|\sqrt{\|\psi\|_{L^2}^2-1}
        +
        \frac{2\eta_{q,\alpha}}{\sqrt{\mu_\perp}}
        \notag  \\
        & =
        S_{q,\alpha}.
        \label{eq:theta_q_constant_bound}
    \end{align}

    Finally, we show that \(\Delta_{q,\alpha}\) gives a lower bound on the distance from the endpoint to the magnet. This is useful for establishing bounds on the magnetic dipole terms. Taylor expansion of the planar director about \(q\psi\) gives
    \[
        d_{q\psi+w}
        =
        d_q+n_qw+R,
        \qquad
        |R|\leq\frac12w^2.
    \]
    Since \(w\in Y\), we have
    \[
        \int_0^1w\,ds=0,
    \]
    and the map
    \[
        t\mapsto \cos t\,e_1-\sin t\,e_3
    \]
    is \(1\)-Lipschitz. Therefore, using
    \(\|\psi-1\|_{L^2}^2=\|\psi\|_{L^2}^2-1\),
    \eqref{eq:w_L2_ball_bound}, and the pointwise bound
    \(|R|\leq w^2/2\), we obtain
    \begin{align}
        |X_\theta-X_q|
        &\leq
        \left|\int_0^1n_qw\,ds\right|
        +\int_0^1|R|\,ds \notag\\
        &=
        \left|
        \int_0^1
        \left[
            n_q-(\cos q\,e_1-\sin q\,e_3)
            \right]w\,ds
        \right|
        +\int_0^1|R|\,ds \notag\\
        &\leq
        |q|\|\psi-1\|_{L^2}\|w\|_{L^2}
        +\frac12\|w\|_{L^2}^2 \notag\\
        &=
        |q|\sqrt{\|\psi\|_{L^2}^2-1}\|w\|_{L^2}
        +\frac12\|w\|_{L^2}^2 \notag\\
        &\leq
        \frac{2\eta_{q,\alpha}}{\sqrt{\mu_\perp}}
        |q|\sqrt{\|\psi\|_{L^2}^2-1}
        +
        \frac{2\eta_{q,\alpha}^2}{\mu_\perp} \notag\\
        &=
        \frac{2\eta_{q,\alpha}}{\sqrt{\mu_\perp}}
        \left(
        |q|\sqrt{\|\psi\|_{L^2}^2-1}
        +
        \frac{\eta_{q,\alpha}}{\sqrt{\mu_\perp}}
        \right) \notag\\
        &=
        E_{q,\alpha}.
        \label{eq:X_Xq_bound}
    \end{align}
    Thus, we are able to estimate
    \begin{align}
        |X_\theta-\mathcal Dp(\alpha)|
        &\geq
        |X_q-\mathcal Dp(\alpha)|-|X_\theta-X_q| \notag\\
        &\geq
        |X_q-\mathcal Dp(\alpha)|-E_{q,\alpha} \notag\\
        &=
        \Delta_{q,\alpha}>0.
        \label{eq:X_magnet_separation}
    \end{align}

    This already proves \eqref{eq:endpoint_separation_bound} for every
    \(w\in\mathcal B_{q,\alpha}\).

    We now estimate the first term in \eqref{eq:DN_qalpha_formula}. Set
    \[
        C(s)
        :=
        \mathsf G(1-s)(1-\cos\theta)
        -\mathsf M\nabla\Phi_\alpha(X_\theta)\cdot d_\theta
        -\frac{2\mathsf M}{\delta^3}.
    \]
    Using the definition of \(H_{q,\alpha}\), the \(1\)-Lipschitz continuity of
    \(t\mapsto \cos t\) and \(t\mapsto d_t\), and
    \eqref{eq:w_Linfty_ball_bound}, we obtain
    \begin{align}
        |C(s)|
        &\leq
        H_{q,\alpha}
        +\mathsf G|\cos\theta-\cos(q\psi)|
        +\mathsf M|g_{q,\alpha}|\,|d_\theta-d_q|
        +\mathsf M|\nabla\Phi_\alpha(X_\theta)-g_{q,\alpha}| \notag\\
        &\leq
        H_{q,\alpha}
        +2\bigl(\mathsf G+\mathsf M|g_{q,\alpha}|\bigr)\eta_{q,\alpha}
        +\mathsf M|\nabla\Phi_\alpha(X_\theta)-g_{q,\alpha}|.
    \end{align}

    It remains only to estimate the last term. For \(t\in[0,1]\), set
    \[
        X_t:=X_q+t(X_\theta-X_q).
    \]
    By \eqref{eq:X_Xq_bound},
    \begin{align}
        |X_t-\mathcal Dp(\alpha)|
        &\geq
        |X_q-\mathcal Dp(\alpha)|-t|X_\theta-X_q| \notag\\
        &\geq
        |X_q-\mathcal Dp(\alpha)|-E_{q,\alpha}
        =
        \Delta_{q,\alpha}.
    \end{align}
    A direct computation of the dipole Hessian gives the bound
    \[
        \|\nabla^2\Phi_\alpha(x)\|_{\mathrm{op}}
        \leq
        \frac{6}{|x-\mathcal Dp(\alpha)|^4}.
    \]
    Applying the mean value theorem and \eqref{eq:X_Xq_bound} give
    \begin{align}
        |\nabla\Phi_\alpha(X_\theta)-g_{q,\alpha}|
        &=
        |\nabla\Phi_\alpha(X_\theta)-\nabla\Phi_\alpha(X_q)| \notag\\
        &\leq
        \sup_{t\in[0,1]}
        \|\nabla^2\Phi_\alpha(X_t)\|_{\mathrm{op}}
        \,|X_\theta-X_q| \notag\\
        &\leq
        \frac{6E_{q,\alpha}}{\Delta_{q,\alpha}^4}.
    \end{align}
    Therefore,
    \begin{equation}
        \label{eq:C_final_bound}
        \|C\|_{L^\infty}
        \leq
        H_{q,\alpha}
        +
        2\bigl(\mathsf G+\mathsf M|g_{q,\alpha}|\bigr)\eta_{q,\alpha}
        +
        \frac{6\mathsf M E_{q,\alpha}}{\Delta_{q,\alpha}^4}.
    \end{equation}

    We next estimate the nonlocal term in \eqref{eq:DN_qalpha_formula}. Since
    \(h,v\in Y\), we may subtract the constant vector
    \(\cos q\,e_1-\sin q\,e_3\). Using the \(1\)-Lipschitz continuity of the
    normal vector and \eqref{eq:theta_q_constant_bound}, we obtain
    \begin{align}
        \left|\int_0^1n_\theta h\,ds\right|
        &=
        \left|
        \int_0^1
        \left[
            n_\theta-(\cos q\,e_1-\sin q\,e_3)
            \right]h\,ds
        \right| \notag\\
        &\leq
        \|\theta-q\|_{L^2}\|h\|_{L^2}
        \leq
        S_{q,\alpha}\|h\|_{L^2},
    \end{align}
    and similarly
    \[
        \left|\int_0^1n_\theta v\,ds\right|
        \leq
        S_{q,\alpha}\|v\|_{L^2}.
    \]
    Therefore, using \eqref{eq:X_magnet_separation} and
    bound on the hessian of the dipole potential,
    we have
    \begin{align}
        \mathsf M
        \left|
        \nabla^2\Phi_\alpha(X_\theta)
        \left[
            \int_0^1n_\theta h\,ds,
            \int_0^1n_\theta v\,ds
            \right]
        \right|
        &\leq
        \mathsf M
        \|\nabla^2\Phi_\alpha(X_\theta)\|_{\mathrm{op}}
        \left|\int_0^1n_\theta h\,ds\right|
        \left|\int_0^1n_\theta v\,ds\right| \notag\\
        &\leq
        \frac{6\mathsf M S_{q,\alpha}^2}{\Delta_{q,\alpha}^4}
        \|h\|_{L^2}\|v\|_{L^2}.
        \label{eq:nonlocal_Hessian_bound}
    \end{align}

    Combining \eqref{eq:DN_qalpha_formula}, \eqref{eq:C_final_bound}, and
    \eqref{eq:nonlocal_Hessian_bound}, we obtain
    \begin{align}
        |DN_{q,\alpha}(w)[h](v)|
        \leq{} &
        \left[
            H_{q,\alpha}
            +
            2\bigl(\mathsf G+\mathsf M|g_{q,\alpha}|\bigr)\eta_{q,\alpha}
            +
            \frac{6\mathsf M}{\Delta_{q,\alpha}^4}
            \bigl(E_{q,\alpha}+S_{q,\alpha}^2\bigr)
            \right]
        \|h\|_{L^2}\|v\|_{L^2}
        \notag   \\
        ={}    &
        \mu_\perp K_{q,\alpha}
        \|h\|_{L^2}\|v\|_{L^2}.
        \label{eq:DN_final_bound}
    \end{align}

    We can now establish \eqref{eq:T_contraction}. Let
    \[
        z:=DT_{q,\alpha}(w)[h].
    \]
    By \eqref{eq:DT_equation},
    \[
        B_0(z,v)
        =
        -DN_{q,\alpha}(w)[h](v)
        \qquad
        \text{for every }v\in Y.
    \]
    Taking \(v=z\), using \eqref{eq:B0_Y_coercivity},
    \eqref{eq:DN_final_bound}, and then \eqref{eq:Y_poincare}, we find
    \begin{align}
        \|z\|_Y^2
        & \leq
        B_0(z,z)
        \notag  \\
        & =
        -DN_{q,\alpha}(w)[h](z)
        \notag  \\
        & \leq
        \mu_\perp K_{q,\alpha}
        \|h\|_{L^2}\|z\|_{L^2}
        \notag  \\
        & \leq
        \mu_\perp K_{q,\alpha}
        \frac{\|h\|_Y}{\sqrt{\mu_\perp}}
        \frac{\|z\|_Y}{\sqrt{\mu_\perp}}
        \notag  \\
        & =
        K_{q,\alpha}\|h\|_Y\|z\|_Y.
    \end{align}
    Thus
    \[
        \|DT_{q,\alpha}(w)[h]\|_Y
        \leq
        K_{q,\alpha}\|h\|_Y,
    \]
    and therefore
    \[
        \|DT_{q,\alpha}(w)\|_{\mathcal L(Y)}
        \leq
        K_{q,\alpha}
        \qquad
        \text{for every }w\in\mathcal B_{q,\alpha}.
    \]
    This proves \eqref{eq:T_contraction} which was the claim.\\

    It remains to establish the sharper error estimate in
    \eqref{eq:w_error_bound}. Since \(w=T_{q,\alpha}(w)\),
    \begin{align}
        \|w\|_Y
        & =
        \|T_{q,\alpha}(w)\|_Y
        \notag  \\
        & \leq
        \|T_{q,\alpha}(0)\|_Y
        +
        \|T_{q,\alpha}(w)-T_{q,\alpha}(0)\|_Y
        \notag  \\
        & \leq
        \eta_{q,\alpha}
        +
        K_{q,\alpha}\|w\|_Y,\notag\\
        \implies \|w\|_Y
        & \leq
        \frac{\eta_{q,\alpha}}{1-K_{q,\alpha}}
        =
        \varepsilon_{q,\alpha}.
    \end{align}

    Finally, since \(w(0)=0\), we again have the \(L^\infty\) bound
    \[
        \|w\|_{L^\infty}
        \leq
        \|w\|_Y
        \leq
        \varepsilon_{q,\alpha}.
    \]
    This proves \eqref{eq:w_error_bound}. Finally, the fixed-point equation
    \[
        w-T_{q,\alpha}(w)=0
    \]
    is smooth in \((q,\alpha,w)\) as long as the endpoint remains separated
    from the magnet. By \eqref{eq:endpoint_separation_bound}, this is true
    at the fixed point. Moreover,
    \[
        D_w\bigl(w-T_{q,\alpha}(w)\bigr)
        =
        I-DT_{q,\alpha}(w),
    \]
    and \eqref{eq:T_contraction} gives
    \[
        \|DT_{q,\alpha}(w)\|_{\mathcal L(Y)}
        \leq
        K_{q,\alpha}<\frac{1}{2}.
    \]
    Thus \(I-DT_{q,\alpha}(w)\) is invertible, and the
    implicit function theorem gives the smooth dependence of \(w(q,\alpha)\)
    on \((q,\alpha)\).
\end{appendixproof}

\begin{appendixproof}[Proof of \Cref{prop:scalar_reduction}]\label{pf:scalar_reduction}
    For any \(v\in V\), using \(\int_0^1\psi\,ds=1\), we may write
    \[
        v
        =
        \left(\int_0^1v\,ds\right)\psi
        +
        \left[
            v-\left(\int_0^1v\,ds\right)\psi
            \right],
    \]
    where the second term belongs to \(Y\).
    By \Cref{thm:mode_accuracy}, \eqref{eq:projected_equilibrium} we have
    \[
        \delta\mathcal E_{\mathrm{pl},\alpha}
            [q\psi+w(q,\alpha)]\left(v-\left(\int_0^1v\,ds\right)\psi\right)=0
    \]
    Therefore,
    \begin{align}
        \delta\mathcal E_{\mathrm{pl},\alpha}
            [q\psi+w](v)
        &=
        \left(\int_0^1v\,ds\right)
        \delta\mathcal E_{\mathrm{pl},\alpha}[q\psi+w](\psi) \notag\\
        &=
        \left(\int_0^1v\,ds\right)F(q,\alpha).
    \end{align}
    Thus, the first variation vanishes on all of \(V\) if and only if
    \(F(q,\alpha)=0\).
\end{appendixproof}
\begin{appendixproof}[Proof of \Cref{cor:planar_branch}]\label{pf:planar_branch}
    Let \(\Phi_\alpha\) denote the explicit dependence of the dipole potential on the parameter \(\alpha\) through \(p(\alpha)\). At \((q,\alpha)=(0,0)\),
    \[
        X_0=e_3,
        \qquad
        p(0)=e_3,
        \qquad
        \nabla\Phi_0(e_3)
        =
        -\frac{2}{\delta^3}e_3.
    \]
    It follows from \eqref{eq:rho_qalpha_def} that
    \[
        \rho_{0,0}=0,
        \qquad
        \eta_{0,0}=0.
    \]
    Thus, \Cref{thm:mode_accuracy} gives
    \[
        w(0,0)=0.
    \]
    Furthermore,
    \[
        \Delta_{0,0}=\delta>0,
        \qquad
        K_{0,0}=0,
    \]
    so the reduced equation admits a unique smooth solution \(w(q,\alpha)\)
    for \((q,\alpha)\) sufficiently close to \((0,0)\).\\

    We first compute the derivatives of \(w\) at the origin. Define the linear functional
    \[
        \mathcal P(q,\alpha,w)\in Y^*
    \]
    by
    \[
        \mathcal P(q,\alpha,w)[v]
        :=
        \delta\mathcal E_{\mathrm{pl},\alpha}[q\psi+w](v),
        \qquad v\in Y.
    \]
    By construction of the correction \(w(q,\alpha)\),
    \begin{equation}
        \label{eq:projected_residual_zero}
        \mathcal P(q,\alpha,w(q,\alpha))=0
        \qquad\text{in }Y^*.
    \end{equation}

    Since \(w(q,\alpha)\) is smooth, differentiating
    \eqref{eq:projected_residual_zero} with respect to \(q\) gives
    \[
        \partial_q\mathcal P(q,\alpha,w(q,\alpha))
        +
        D_w\mathcal P(q,\alpha,w(q,\alpha))
        [\partial_qw(q,\alpha)]
        =
        0
        \qquad\text{in }Y^*.
    \]
    Evaluating this identity on \(v\in Y\), we obtain
    \begin{align}
        0
        &=
        \partial_q\mathcal P(q,\alpha,w(q,\alpha))[v]
        +
        D_w\mathcal P(q,\alpha,w(q,\alpha))
        [\partial_qw(q,\alpha)][v]
        \notag\\
        &=
        \delta^2\mathcal E_{\mathrm{pl},\alpha}
            [q\psi+w(q,\alpha)]
            [\psi,v]
        +
        \delta^2\mathcal E_{\mathrm{pl},\alpha}
            [q\psi+w(q,\alpha)]
            [\partial_qw(q,\alpha),v]
        \notag\\
        &=
        \delta^2\mathcal E_{\mathrm{pl},\alpha}
            [q\psi+w(q,\alpha)]
            [\psi+\partial_qw(q,\alpha),v].
        \label{eq:differentiated_projected_equation_q}
    \end{align}
    At \((q,\alpha)=(0,0)\), we have \(w(0,0)=0\), and the planar second variation at
    the straight configuration is \(B_0\). Thus, we have that
    \[
        B_0(\psi+\partial_qw(0,0),v)=0
        \qquad
        \text{for every }v\in Y.
    \]
    Using \eqref{eq:psi_B0_orthogonality},
    \begin{align}
        B_0(\partial_qw(0,0),v)
        &=
        B_0(\psi+\partial_qw(0,0),v)-B_0(\psi,v)
        \notag\\
        &=
        0
        \qquad
        \text{for every }v\in Y.
    \end{align}
    Since \(\partial_qw(0,0)\in Y\), taking
    \(v=\partial_qw(0,0)\) and using coercivity of \(B_0\) on \(Y\) we conclude that
    \[
        \partial_qw(0,0)=0.
    \]

    We similarly differentiate \eqref{eq:projected_residual_zero} with respect
    to \(\alpha\). The chain rule gives
    \[
        \partial_\alpha\mathcal P(q,\alpha,w(q,\alpha))
        +
        D_w\mathcal P(q,\alpha,w(q,\alpha))
        [\partial_\alpha w(q,\alpha)]
        =
        0
        \qquad\text{in }Y^*.
    \]
    Evaluating this identity on \(v\in Y\), we obtain
    \begin{align}
        0
        &=
        \partial_\alpha\mathcal P(q,\alpha,w(q,\alpha))[v]
        +
        D_w\mathcal P(q,\alpha,w(q,\alpha))
        [\partial_\alpha w(q,\alpha)][v] \notag\\
        &=
        \mathsf M
        \partial_\alpha\nabla\Phi_\alpha(X_\theta)
        \cdot
        \int_0^1 n_\theta v\,ds
        +
        \delta^2\mathcal E_{\mathrm{pl},\alpha}[\theta]
            [\partial_\alpha w(q,\alpha),v],
    \end{align}
    where
    \[
        \theta=q\psi+w(q,\alpha),
        \qquad
        X_\theta=\int_0^1d_\theta\,ds.
    \]
    Here \(\partial_\alpha\nabla\Phi_\alpha(X)\) denotes differentiation with
    respect to the explicit parameter \(\alpha\), with \(X\) held fixed.

    At \((q,\alpha)=(0,0)\), we have
    \[
        \theta=0,
        \qquad
        X_\theta=e_3,
        \qquad
        n_\theta=e_1,
    \]
    and the planar Hessian is \(B_0\). Thus, we have
    \begin{align}
        0
        &=
        \mathsf M
        \partial_\alpha\nabla\Phi_\alpha(e_3)\big|_{\alpha=0}
        \cdot e_1
        \int_0^1v\,ds
        +
        B_0(\partial_\alpha w(0,0),v) \notag\\
        &=
        B_0(\partial_\alpha w(0,0),v)
        \qquad
        \text{for every }v\in Y,
    \end{align}
    since \(\int_0^1v\,ds=0\).

    Again using coercivity of \(B_0\) on \(Y\), we conclude that
    \[
        \partial_\alpha w(0,0)=0.
    \]

    We now compute the partial derivatives of \(F\) at the origin. Recall that
    \[
        F(q,\alpha)
        =
        \delta\mathcal E_{\mathrm{pl},\alpha}
            [q\psi+w(q,\alpha)](\psi).
    \]
    Differentiating with respect to \(q\), we obtain
    \begin{align}
        \partial_qF(q,\alpha)
        &=
        \delta^2\mathcal E_{\mathrm{pl},\alpha}
            [q\psi+w(q,\alpha)]
            [\psi+\partial_qw(q,\alpha),\psi].
    \end{align}
    Therefore, using \(w(0,0)=0\), \(\partial_qw(0,0)=0\), and
    \(\delta^2\mathcal E_{\mathrm{pl},0}[0]=B_0\),
    \begin{align}
        \partial_qF(0,0)
        &=
        \delta^2\mathcal E_{\mathrm{pl},0}[0][\psi,\psi] \notag\\
        &=
        B_0(\psi,\psi)>0.
        \label{eq:Fq_origin}
    \end{align}

    We next compute \(\partial_\alpha F(0,0)\). Differentiating \(F\) with
    respect to \(\alpha\) gives
    \begin{align}
        \partial_\alpha F(q,\alpha)
        ={}&
        \delta^2\mathcal E_{\mathrm{pl},\alpha}
            [q\psi+w(q,\alpha)]
            [\partial_\alpha w(q,\alpha),\psi] \notag\\
        &+
        \mathsf M
        \partial_\alpha\nabla\Phi_\alpha(X_\theta)
        \cdot
        \int_0^1n_\theta\psi\,ds,
    \end{align}
    where
    \[
        \theta=q\psi+w(q,\alpha),
        \qquad
        X_\theta=\int_0^1d_\theta\,ds.
    \]
    At \((q,\alpha)=(0,0)\), we have
    \[
        \theta=0,
        \qquad
        X=e_3,
        \qquad
        n_\theta=e_1,
        \qquad
        \partial_\alpha w(0,0)=0.
    \]
    Thus, using \(\int_0^1\psi\,ds=1\),
    \begin{align}
        \partial_\alpha F(0,0)
        &=
        \mathsf M
        \partial_\alpha\nabla\Phi_\alpha(e_3)\big|_{\alpha=0}
        \cdot e_1.
    \end{align}
    For
    \[
        p(\alpha)=\sin\alpha\,e_1+\cos\alpha\,e_3,
    \]
    a direct differentiation of the dipole field gives
    \[
        \partial_\alpha\nabla\Phi_\alpha(e_3)\big|_{\alpha=0}
        =
        -\frac{2\mathcal D+1}{\delta^4}e_1.
    \]
    Thus,
    \begin{equation}
        \label{eq:Falpha_origin}
        \partial_\alpha F(0,0)
        =
        -\frac{\mathsf M(2\mathcal D+1)}{\delta^4}.
    \end{equation}

    Since \(\partial_qF(0,0)>0\), the implicit function theorem applied to
    \[
        F(q,\alpha)=0
    \]
    gives a unique smooth function \(q=q(\alpha)\) near \(\alpha=0\), with
    \(q(0)=0\). By \Cref{prop:scalar_reduction}, the corresponding
    configurations are precisely the local planar equilibria.\\

    Finally, differentiating
    \[
        F(q(\alpha),\alpha)=0
    \]
    at \(\alpha=0\) gives
    \begin{align}
        0
        &=
        \partial_qF(0,0)\,q'(0)
        +
        \partial_\alpha F(0,0),
    \end{align}
    and therefore
    \begin{align}
        q'(0)
        &=
        -\frac{\partial_\alpha F(0,0)}
        {\partial_qF(0,0)} \notag\\
        &=
        \frac{\mathsf M(2\mathcal D+1)}
        {\delta^4 B_0(\psi,\psi)}
        >0.
    \end{align}
\end{appendixproof}

\begin{appendixproof}[Proof of \Cref{thm:planar_stability}]\label{pf:planar_stability}
    For \(h,v\in Y\), differentiating
    \eqref{eq:projected_residual_identity} with respect to \(w\) gives
    \[
        Q_{\mathrm{pl}}(h,v)
        =
        B_0(h,v)+DN_{q,\alpha}(w)[h](v).
    \]
    The estimate established in the proof of \Cref{thm:mode_accuracy} gives
    \[
        |DN_{q,\alpha}(w)[h](v)|
        \leq
        \mu_\perp K_{q,\alpha}
        \|h\|_{L^2}\|v\|_{L^2}.
    \]
    Using \eqref{eq:Y_poincare} and the coercivity of \(B_0\) on \(Y\),
    \begin{align}
        Q_{\mathrm{pl}}(v,v)
        &\geq
        B_0(v,v)
        -\mu_\perp K_{q,\alpha}\|v\|_{L^2}^2 \notag\\
        &\geq
        \|v\|_Y^2
        -K_{q,\alpha}\|v\|_Y^2 \notag\\
        &=
        (1-K_{q,\alpha})\|v\|_Y^2,
    \end{align}
    which proves \eqref{eq:planar_Y_coercivity}.\\

    By construction of \(w(q,\alpha)\),
    \[
        \delta\mathcal E_{\mathrm{pl},\alpha}
            [q\psi+w(q,\alpha)](v)=0
        \qquad
        \text{for every }v\in Y.
    \]
    Differentiating with respect to \(q\) gives
    \begin{align}
        0
        &=
        \delta^2\mathcal E_{\mathrm{pl},\alpha}[\theta]
            [\psi+\partial_qw(q,\alpha),v] \notag\\
        &=
        Q_{\mathrm{pl}}(\xi_{q,\alpha},v)
        \qquad
        \text{for every }v\in Y,
    \end{align}
    which proves \eqref{eq:xi_Y_orthogonality}.\\

    To get the last equality, we differentiate
    \[
        F(q,\alpha)
        =
        \delta\mathcal E_{\mathrm{pl},\alpha}[\theta](\psi)
    \]
    with respect to \(q\), to obtain
    \begin{align}
        \partial_qF(q,\alpha)
        &=
        Q_{\mathrm{pl}}(\xi_{q,\alpha},\psi) \notag\\
        &=
        Q_{\mathrm{pl}}
        \bigl(\xi_{q,\alpha},
        \xi_{q,\alpha}-\partial_qw(q,\alpha)\bigr) \notag\\
        &=
        Q_{\mathrm{pl}}(\xi_{q,\alpha},\xi_{q,\alpha}),
    \end{align}
    where the last equality follows from
    \eqref{eq:xi_Y_orthogonality} since \(\partial_qw(q,\alpha)\in Y\). This proves
    \eqref{eq:partialqF_second_variation}.\\

    Now let \(h\in V\), and set
    \[
        a:=\int_0^1h\,ds,
        \qquad
        v:=h-a\xi_{q,\alpha}.
    \]
    By taking the mean, we have
    \(v\in Y\). Thus
    \[
        h=a\xi_{q,\alpha}+v,
    \]
    and the decomposition is unique. Using
    \eqref{eq:xi_Y_orthogonality} and
    \eqref{eq:partialqF_second_variation},
    \begin{align}
        Q_{\mathrm{pl}}(h,h)
        &=
        a^2Q_{\mathrm{pl}}(\xi_{q,\alpha},\xi_{q,\alpha})
        +
        2aQ_{\mathrm{pl}}(\xi_{q,\alpha},v)
        +
        Q_{\mathrm{pl}}(v,v) \notag\\
        &=
        a^2\partial_qF(q,\alpha)
        +
        Q_{\mathrm{pl}}(v,v),
    \end{align}
    which proves \eqref{eq:planar_Hessian_decomposition}.\\

    Now we are ready to prove the three stability criteria.
    First, by coercivity of \(Q_{\mathrm{pl}}\) on \(Y\) and \eqref{eq:planar_Hessian_decomposition}, we see that the stabiility is controlled by the sign of \(\partial_qF(q,\alpha)\). If \(\partial_qF(q,\alpha)>0\), then \(Q_{\mathrm{pl}}(h,h)>0\) for every nonzero \(h\in V\), so the planar equilibrium is strictly stable. If \(\partial_qF(q,\alpha)=0\), then \(Q_{\mathrm{pl}}(h,h)\geq0\) for every \(h\in V\), and the kernel is spanned by \(\xi_{q,\alpha}\). Finally, if \(\partial_qF(q,\alpha)<0\), then \(Q_{\mathrm{pl}}(h,h)<0\) for \(h=\xi_{q,\alpha}\), and the Morse index is one since it is the only possible unstable direction. This proves the three stability criteria.
\end{appendixproof}
\begin{appendixproof}[Proof of \Cref{prop:second_variation_splitting}]\label{pf:second_variation_splitting}
    Let \(d_\varepsilon\) be an admissible variation of \(d_\theta\), and set
    \[
        \eta
        :=
        \left.\frac{d}{d\varepsilon}\right|_{\varepsilon=0}d_\varepsilon,
        \qquad
        \zeta
        :=
        \left.\frac{d^2}{d\varepsilon^2}\right|_{\varepsilon=0}d_\varepsilon.
    \]
    Differentiating \(|d_\varepsilon|^2=1\) gives
    \[
        d_\theta\cdot\eta=0,
        \qquad
        d_\theta\cdot\zeta=-|\eta|^2.
    \]
    Since \((n_\theta,e_2)\) is an orthonormal basis of
    the tangential variation at the point \(d_\theta\), every admissible variation has the unique form
    \[
        \eta=a\,n_\theta+b\,e_2,
        \qquad
        a,b\in V.
    \]

    Using the Lagrange multiplier form of the Euler--Lagrange equation for \(d_\theta\), we compute that it must satisfy for some lagrange multiplier \(\lambda_\theta\)
    \[
        -d_\theta''
        +\mathsf G(1-s)e_3
        +\mathsf M g_\theta
        =
        \lambda_\theta d_\theta.
    \]
    Taking its scalar product with \(d_\theta\) yields
    \begin{align}
        \lambda_\theta
        &=
        -d_\theta''\cdot d_\theta
        +\mathsf G(1-s)e_3\cdot d_\theta
        +\mathsf M g_\theta\cdot d_\theta \notag\\
        &=
        \theta'^2
        +\mathsf G(1-s)\cos\theta
        +\mathsf M g_\theta\cdot d_\theta.
        \label{eq:lambda_theta_def}
    \end{align}

    We now compute the second variation of the full energy.
    Since
    \[
        \mathcal E_{\mathrm{perm}}[d_\varepsilon]
        =
        \frac12\int_0^1|d_\varepsilon'|^2\,ds
        +
        \mathsf G\int_0^1(1-s)e_3\cdot d_\varepsilon\,ds
        +
        \mathsf M\Phi_\alpha(X_{d_\varepsilon}),
    \]
    we differentiate each term separately. For the bending energy,
    \begin{equation*}
        \left.\frac{d^2}{d\varepsilon^2}\right|_{\varepsilon=0}
        \frac12\int_0^1|d_\varepsilon'|^2\,ds
        =
        \int_0^1|\eta'|^2\,ds
        +
        \int_0^1d_\theta'\cdot\zeta'\,ds.
    \end{equation*}
    Similarly, the gravitational term gives
    \begin{align}
        \left.\frac{d^2}{d\varepsilon^2}\right|_{\varepsilon=0}
        \mathsf G\int_0^1(1-s)e_3\cdot d_\varepsilon\,ds
        &=
        \mathsf G\int_0^1(1-s)e_3\cdot\zeta\,ds.
    \end{align}

    For the magnetic term, we first compute the first and second derivatives of the endpoint \(X_{d_\varepsilon}\):
    \begin{align}
        \left.\frac{d}{d\varepsilon}\right|_{\varepsilon=0}
        X_{d_\varepsilon}
        &=
        \int_0^1\eta\,ds,
        \\
        \left.\frac{d^2}{d\varepsilon^2}\right|_{\varepsilon=0}
        X_{d_\varepsilon}
        &=
        \int_0^1\zeta\,ds.
    \end{align}
    Thus, by the chain rule,
    \begin{equation*}
        \left.\frac{d^2}{d\varepsilon^2}\right|_{\varepsilon=0}
        \Phi_\alpha(X_{d_\varepsilon})
        =
        g_\theta\cdot\int_0^1\zeta\,ds
        +
        \nabla^2\Phi_\alpha(X_\theta)
        \left[
            \int_0^1\eta\,ds,
            \int_0^1\eta\,ds
            \right].
    \end{equation*}

    Combining the three terms gives
    \begin{align}
        \delta^2\mathcal E_{\mathrm{perm}}[d_\theta](\eta,\eta)
        ={}&
        \int_0^1|\eta'|^2\,ds
        +
        \int_0^1d_\theta'\cdot\zeta'\,ds
        +
        \mathsf G\int_0^1(1-s)e_3\cdot\zeta\,ds
        \notag\\
        &+
        \mathsf M g_\theta\cdot\int_0^1\zeta\,ds
        +
        \mathsf M\nabla^2\Phi_\alpha(X_\theta)
        \left[
            \int_0^1\eta\,ds,
            \int_0^1\eta\,ds
            \right].
    \end{align}
    We now integrate the second term by parts. Since the clamp gives
    \(\zeta(0)=0\) and the natural boundary condition gives
    \(d_\theta'(1)=0\),
    \begin{align}
        \int_0^1d_\theta'\cdot\zeta'\,ds
        &=
        \left[d_\theta'\cdot\zeta\right]_0^1
        -
        \int_0^1d_\theta''\cdot\zeta\,ds
        \notag\\
        &=
        -\int_0^1d_\theta''\cdot\zeta\,ds.
    \end{align}

    Therefore,
    \begin{align}
        \delta^2\mathcal E_{\mathrm{perm}}[d_\theta](\eta,\eta)
        ={}&
        \int_0^1|\eta'|^2\,ds
        +
        \int_0^1
        \left[
            -d_\theta''
            +\mathsf G(1-s)e_3
            +\mathsf M g_\theta
            \right]\cdot\zeta\,ds
        \notag\\
        &+
        \mathsf M\nabla^2\Phi_\alpha(X_\theta)
        \left[
            \int_0^1\eta\,ds,
            \int_0^1\eta\,ds
            \right].
    \end{align}

    Now we use the Euler--Lagrange equation for \(d_\theta\) and the relation between \(\xi\) and \(\zeta\) to replace the second term and derive
    \begin{align}
        \int_0^1
        \left[
            -d_\theta''
            +\mathsf G(1-s)e_3
            +\mathsf M g_\theta
            \right]\cdot\zeta\,ds
        &=
        \int_0^1\lambda_\theta d_\theta\cdot\zeta\,ds
        \notag\\
        &=
        -\int_0^1\lambda_\theta|\eta|^2\,ds.
    \end{align}
    Thus,
    \begin{equation}
        \label{eq:full_second_variation_general}
        \delta^2\mathcal E_{\mathrm{perm}}[d_\theta](\eta,\eta)
        =
        \int_0^1
        \left(
        |\eta'|^2-\lambda_\theta|\eta|^2
        \right)ds
        +
        \mathsf M\nabla^2\Phi_\alpha(X_\theta)
        \left[
            \int_0^1\eta\,ds,
            \int_0^1\eta\,ds
            \right].
    \end{equation}

    Now we write the second variation in terms of the decomposition \(\eta=a\,n_\theta+b\,e_2\).
    Since \(n_\theta'=-\theta'd_\theta\), we have
    \begin{align}
        \eta'
        &=
        a'n_\theta-a\theta'd_\theta+b'e_2, \notag\\
        |\eta'|^2
        &=
        (a')^2+\theta'^2a^2+(b')^2,
        \qquad
        |\eta|^2=a^2+b^2,
        \notag\\
        \int_0^1\eta\,ds
        &=
        \int_0^1a\,n_\theta\,ds
        +
        \left(\int_0^1b\,ds\right)e_2.
    \end{align}
    Because \(p(\alpha)\) and \(X_\theta\) lie in the \(e_1e_3\)-plane,
    reflection symmetry in the \(e_2\)-direction gives
    \[
        \nabla^2\Phi_\alpha(X_\theta)[e_2,e_1]
        =
        \nabla^2\Phi_\alpha(X_\theta)[e_2,e_3]
        =
        0.
    \]
    Thus, all mixed \(a\)-\(b\) terms vanish, and
    \begin{align}
        \delta^2\mathcal E_{\mathrm{perm}}[d_\theta](\eta,\eta)
        ={}&
        \int_0^1
        \left[
            (a')^2
            +
            (\theta'^2-\lambda_\theta)a^2
            \right]ds
        +
        \mathsf M\nabla^2\Phi_\alpha(X_\theta)
        \left[
            \int_0^1a\,n_\theta\,ds,
            \int_0^1a\,n_\theta\,ds
            \right]
        \notag\\
        &+
        \int_0^1
        \left[
            (b')^2-\lambda_\theta b^2
            \right]ds
        +
        \mathsf M\partial_{22}\Phi_\alpha(X_\theta)
        \left(\int_0^1b\,ds\right)^2.
    \end{align}
    Using \eqref{eq:lambda_theta_def},
    \begin{align}
        \theta'^2-\lambda_\theta
        &=
        -\mathsf G(1-s)\cos\theta
        -\mathsf M g_\theta\cdot d_\theta,
    \end{align}
    so the first line is precisely \(Q_{\mathrm{pl}}(a,a)\), while the second
    is \(Q_\perp[b]\). Therefore
    \[
        \delta^2\mathcal E_{\mathrm{perm}}[d_\theta](\eta,\eta)
        =
        Q_{\mathrm{pl}}(a,a)+Q_\perp[b].
    \]

    Finally, since \(p(\alpha)\cdot e_2=X_\theta\cdot e_2=0\),
    \begin{align}
        \partial_{22}\Phi_\alpha(X_\theta)
        &=
        -3\frac{p(\alpha)\cdot(X_\theta-\mathcal Dp(\alpha))}
        {|X_\theta-\mathcal Dp(\alpha)|^5}, \notag\\
        p(\alpha)\cdot(X_\theta-\mathcal Dp(\alpha))
        &=
        p(\alpha)\cdot X_\theta-\mathcal D \notag\\
        &\leq
        |X_\theta|-\mathcal D
        \leq
        1-\mathcal D<0.
    \end{align}
    Thus, we get the sign condition
    \[
        \partial_{22}\Phi_\alpha(X_\theta)>0.
    \]
\end{appendixproof}

\begin{appendixproof}[Proof of \Cref{thm:conditional_out_of_plane_stability}]\label{pf:conditional_out_of_plane_stability}
    We wish to prove that
    \begin{equation}
        \label{eq:local_transverse_positive}
        \int_0^1
        \left[
            (b')^2-\lambda_\theta b^2
            \right]ds
        >0
        \qquad
        \text{for every }b\in V\setminus\{0\}.
    \end{equation}

    To prove this, we use Picone's identity to rewrite the integrand. Given any strictly positive function \(u\) and any \(b\in V\setminus\{0\}\), Picone's identity gives
    \begin{align}
        u^2\left(\frac{b}{u}\right)'^2
        &=
        \left(
        b'-\frac{u'}{u}b
        \right)^2 \notag\\
        &=
        (b')^2
        -2\frac{u'}{u}bb'
        +\frac{(u')^2}{u^2}b^2.
    \end{align}
    On the other hand, we can also write
    \begin{align}
        \left(\frac{u'}{u}b^2\right)'
        &=
        2\frac{u'}{u}bb'
        +
        \left(\frac{u''}{u}
        -\frac{(u')^2}{u^2}\right)b^2.
    \end{align}
    Adding these two identities gives
    \begin{align}
        u^2\left(\frac{b}{u}\right)'^2
        +
        \left(\frac{u'}{u}b^2\right)'
        &=
        (b')^2+\frac{u''}{u}b^2.
    \end{align}
    Thus we have the identity
    \begin{equation}
        \label{eq:Picone_identity_rewrite}
        (b')^2-\lambda_\theta b^2
        =
        u^2\left(\frac{b}{u}\right)'^2
        +
        \left(\frac{u'}{u}b^2\right)'+\left(\frac{-u''-\lambda_\theta u}{u}\right)b^2
        .
    \end{equation}

    Note that the numerator of the last term is precisely the Euler-Lagrange operator for the energy \eqref{eq:local_transverse_positive}. So we aim to choose a positive function \(u\) such that the last term is strictly positive. \\

    From here on, we assume that \(\theta(s)\in(0,\pi)\) for every \(s\in(0,1]\). The proof in the other case is similar. We claim that the natural choice is \(u=\sin\theta\) which is positive by our assumption. We now compute the Euler-Lagrange operator applied to \(u\):
    \begin{align*}
        u'' &= \theta''\cos\theta - (\theta')^2\sin\theta,\\
        \lambda_\theta &= \theta'^2` + \mathsf G(1-s)\cos\theta + \mathsf M g_\theta\cdot d_\theta,\\
        \implies -u'' - \lambda_\theta u &= -\theta''\cos\theta + (\theta')^2\sin\theta - \lambda_\theta \sin\theta\\
        &= -\theta''\cos\theta - \mathsf G(1-s)\sin\theta\cos\theta - \mathsf M(g_\theta\cdot d_\theta)\sin\theta
    \end{align*}

    But now using the planar Euler-Lagrange equation, we have
    \[
        -\theta'' = \mathsf G(1-s)\sin\theta - \mathsf M g_\theta\cdot n_\theta,
    \]
    so substituting this in gives
    \begin{align*}
        -u'' - \lambda_\theta u &= -\mathsf M(g_\theta\cdot n_\theta)\cos\theta - \mathsf M(g_\theta\cdot d_\theta)\sin\theta\\
        &= -\mathsf M g_\theta\cdot (\cos\theta n_\theta + \sin\theta d_\theta)\\
        &= -\mathsf M g_\theta\cdot e_1,
    \end{align*}
    where we used the fact that \(\cos\theta n_\theta + \sin\theta d_\theta = e_1\).

    Thus, by the second condition in \eqref{eq:out_of_plane_sign_conditions}, we have
    \[
        -u'' - \lambda_\theta u = -\mathsf M (g_\theta)_1 > 0.
    \]

    Now we can integrate \eqref{eq:Picone_identity_rewrite} over \((\varepsilon,1)\) for some small \(\varepsilon>0\) and we can throw away the nonnegative terms. Then we have that
    \begin{equation}
        \label{eq:integrated_Picone_approx}
        \int_\varepsilon^1 \left[(b')^2 - \lambda_\theta b^2\right] ds > \int_\varepsilon^1 \left(\frac{u'}{u}b^2\right)' ds = \left[\frac{u'}{u}b^2\right]_\varepsilon^1 = -\frac{u'(\varepsilon)}{u(\varepsilon)}b^2(\varepsilon),
    \end{equation}
    where in the last equality we used that \(u'(1) = \theta'(1)\cos\theta(1) = 0\) due to the natural boundary condition.
    Thus, in order to prove our statement it suffices to show that
    \begin{equation}
        \label{eq:boundary_term_limit}
        \lim_{\varepsilon\to 0^+} \frac{u'(\varepsilon)}{u(\varepsilon)}b^2(\varepsilon) = 0.
    \end{equation}

    This is based on a few facts. First, note that since
    \[-\theta'' = \mathsf G(1-s)\sin\theta - \mathsf M g_\theta\cdot n_\theta\]
    and \(\theta \in C^0\), we can bootstrap the regularity of \(\theta\) to be \(C^\infty\) on \([0,1]\). Thus, \(u = \sin\theta\) is also \(C^\infty\) on \([0,1]\).
    Furthermore, since \(u>0\) and \(u(0)=0\), we claim that \(u'(0) > 0\). Indeed, if \(u'(0) = 0\), then we would have \(-u''(0) = -\mathsf M (g_\theta)_1 > 0\), which by Taylor expansion would imply that \(u(s) < 0\) for sufficiently small \(s>0\), contradicting the fact that \(u>0\).
    Then by continuity, we may find an interval \((0,\delta)\) such that
    \[
        \frac{1}{2}u'(0) < u'(s) < \frac{3}{2}u'(0),
        \qquad
        \forall s\in(0,\delta).
    \]
    Thus, by choosing \(\varepsilon\) sufficiently small, we have that
    \[
        u(\varepsilon) = \int_0^\varepsilon u'(s) ds \geq \frac{1}{2}u'(0)\varepsilon> \frac{1}{3}u'(\varepsilon)\varepsilon.
    \]
    This gives that
    \begin{align*}
        \frac{u'(\varepsilon)}{u(\varepsilon)}b^2(\varepsilon) \leq \frac{3}{\varepsilon}b^2(\varepsilon) = \frac{3}{\varepsilon}\left|\int_0^\varepsilon b'(s) ds\right|^2 \leq 3\int_0^\varepsilon |b'(s)|^2 ds \to 0 \quad \text{as } \varepsilon \to 0^+,
    \end{align*}
    where the limit goes to 0 because \(b'\in L^2((0,1))\). This proves \eqref{eq:boundary_term_limit}.
\end{appendixproof}

\begin{appendixproof}[Proof of \Cref{thm:sign_preservation}] \label{pf:sign_preservation}
    We prove the result for the positive branch; the negative branch follows
    by reflection symmetry.\\
    \emph{Step 1: a planar equilibrium that satisfies \(0\leq \theta \leq \pi \) with \(\alpha>0\) cannot also satisfy
        \((g_\theta)_1=0\).}\\
    Suppose by contradiction that
    \[
        (g_\theta)_1=0.
    \]
    First, we claim that
    \begin{equation}
        \label{eq:g3_crossing_bound}
        ((g_\theta)_3)_+
        \leq
        \frac{\sqrt6}{9\delta^3}.
    \end{equation}
    Set
    \[
        y:=X_\theta-\mathcal Dp(\alpha),
        \qquad
        r:=|y|.
    \]
    The magnetic dipole field is
    \[
        g_\theta
        =
        \frac{p(\alpha)}{r^3}
        -
        3\frac{(p(\alpha)\cdot y)y}{r^5}.
    \]
    If \((g_\theta)_3\leq0\), then
    \eqref{eq:g3_crossing_bound} is immediate. Suppose therefore that
    \[
        (g_\theta)_3>0.
    \]
    Since \(g_\theta\) is planar and \((g_\theta)_1=0\), for \(0\leq\alpha\leq\pi/2\)
    \[
        g_\theta=(g_\theta)_3e_3\quad \text{and}\quad
        p(\alpha)\cdot g_\theta
        =
        \cos\alpha\;(g_\theta)_3\geq0
    \]
    On the other hand, a direct computation by taking the scalar product of the dipole formula with
    \(p(\alpha)\) gives
    \begin{align}
        p(\alpha)\cdot g_\theta
        & =
        \frac1{r^3}
        -
        3\frac{(p(\alpha)\cdot y)^2}{r^5} \notag \\
        & =
        \frac1{r^3}
        \left[
            1
            -
            3\frac{(p(\alpha)\cdot y)^2}{r^2}
            \right].
    \end{align}
    Thus, nonnegativity implies that
    \begin{equation}
        \label{eq:crossing_angular_bound}
        \frac{|p(\alpha)\cdot y|}{r}
        \leq
        \frac1{\sqrt3}.
    \end{equation}
    However, we can show that \(p(\alpha)\cdot y\) is strictly negative using the bound
    \begin{align}
        p(\alpha)\cdot y
        & =
        p(\alpha)\cdot X_\theta-\mathcal D \notag \\
        & \leq
        |X_\theta|-\mathcal D \notag              \\
        & \leq
        1-\mathcal D
        =
        -\delta\notag                             \\
        \implies |p(\alpha)\cdot y| & \geq\delta.
    \end{align}

    Combining this with \eqref{eq:crossing_angular_bound} yields a lower bound on \(r\):
    \begin{equation}
        \label{eq:crossing_distance_bound}
        r\geq\sqrt3\,\delta.
    \end{equation}

    Now we are ready to bound the magnetic term more finely.
    Writing
    \[
        \nu:=\frac{y}{r},
    \]
    we have
    \begin{align}
        |g_\theta|^2
        & =
        \frac{
            |p(\alpha)-3(p(\alpha)\cdot\nu)\nu|^2
        }{r^6} \notag          \\
        & =
        \frac{
            1+3(p(\alpha)\cdot\nu)^2
        }{r^6},\notag          \\
        & \leq \frac{2}{r^6},
    \end{align}
    where in the last step we used \eqref{eq:crossing_angular_bound}. We conclude by using \eqref{eq:crossing_distance_bound}, to obtain
    \begin{align}
        (g_\theta)_3
        & =
        |g_\theta| \notag \\
        & \leq
        \frac{\sqrt2}{(\sqrt3\,\delta)^3}=
        \frac{\sqrt6}{9\delta^3},
    \end{align}
    which proves \eqref{eq:g3_crossing_bound}.\\

    Now we can conclude the proof of Step 1. Since \((g_\theta)_1=0\),
    \[
        g_\theta\cdot n_\theta
        =
        -(g_\theta)_3\sin\theta,
    \]
    and the equilibrium equation becomes
    \[
        -\theta''
        =
        \left[
            \mathsf G(1-s)
            +
            \mathsf M(g_\theta)_3
            \right]\sin\theta.
    \]
    Testing this equation with \(\theta\), and using the boundary conditions,
    gives
    \begin{align}
        \int_0^1|\theta'|^2\,ds
        & =
        \mathsf G
        \int_0^1(1-s)\theta\sin\theta\,ds
        +
        \mathsf M(g_\theta)_3
        \int_0^1\theta\sin\theta\,ds.
    \end{align}
    Since
    \[
        0\leq\theta\sin\theta\leq\theta^2
        \qquad
        \text{for }0\leq\theta\leq\pi,
    \]
    we obtain using \eqref{eq:g3_crossing_bound},
    \begin{align}
        \int_0^1|\theta'|^2\,ds -\mathsf G\int_0^1(1-s)\theta^2\,ds & \leq \mathsf M((g_\theta)_3)_+ \int_0^1\theta^2\,ds,\\
        \int_0^1|\theta'|^2\,ds -\mathsf G\int_0^1(1-s)\theta^2\,ds
        &\leq \frac{\sqrt6\,\mathsf M}{9\delta^3}\int_0^1|\theta|^2\,ds.
    \end{align}

    Now recalling the definition of \(\beta_0^{\mathrm{cl}}(\mathsf G)\) in \Cref{thm:uniform_straight_stability}, we have
    \[
        \beta_0^{\mathrm{cl}}(\mathsf G) \int_0^1|\theta|^2\,ds \leq \frac{\sqrt6\,\mathsf M}{9\delta^3}\int_0^1|\theta|^2\,ds.
    \]

    This gives a contradiction since the straight line solution \(\theta = 0\) is not an equilirbrium for \(0<\alpha\leq\pi/2\).\\

    \emph{Step 2: If \( 0\leq\theta\leq\pi\) and \((g_\theta)_1<0\), then
        \( 0<\theta<\pi\) for every \(s>0\).}

    Suppose
    \[
        0\leq\theta\leq\pi,
        \qquad
        (g_\theta)_1<0.
    \]
    At any zero \(s_0\) of \(\theta\), the equilibrium equation gives
    \begin{align}
        \theta''(s_0)
        & =
        \mathsf M(g_\theta)_1
        <0.
        \label{eq:second_derivative_at_zero}
    \end{align}
    If \(s_0\in(0,1)\), then \(s_0\) is a local minimum of the nonnegative
    function \(\theta\), which contradicts
    \eqref{eq:second_derivative_at_zero}. A zero at \(s=1\) is also impossible. Indeed, if not, then \(\theta\) would satisfy the boundary conditions
    \[
        \theta(1)=0,
        \qquad
        \theta'(1)=0,
        \qquad
        \theta''(1)<0.
    \]
    However, this would contradict non-negativity of \(\theta\) in a neighborhood of \(s=1\). Thus, the only possible zero is at \(s=0\) giving
    \[
        \theta(s)>0
        \qquad
        \text{for every }s\in(0,1].
    \]
    Next we give a similar argument to show that \(\theta(s)<\pi\) for every \(s\in(0,1]\). Suppose by contradiction that \(\theta(s_0)=\pi\) for some \(s_0\in(0,1]\). If \(s_0\in(0,1)\), then \(s_0\) is a local maximum of the function \(\theta\), and the equilibrium equation gives
    \[
        \theta''(s_0)
        =
        -\mathsf M(g_\theta)_1
        >0,
    \]
    which is impossible. If \(s_0=1\), then \(\theta\) would satisfy the boundary conditions
    \[
        \theta(1)=\pi,
        \qquad
        \theta'(1)=0,
        \qquad
        \theta''(1)=-\mathsf M(g_\theta)_1>0
    \]
    which would contradict \(\theta\leq\pi\) in a neighborhood of \(s=1\). Thus
    \[
        \theta(s)<\pi
        \qquad
        \text{for every }s\in(0,1].
    \]

    A similar argument shows that
    \begin{equation}
        \label{eq:positive_clamp_slope}
        \theta'(0)>0.
    \end{equation}
    Indeed, if \(\theta'(0)=0\), then \(\theta\) would satisfy the boundary conditions
    \[
        \theta(0)=0,
        \qquad
        \theta'(0)=0,
        \qquad
        \theta''(0)=\mathsf M(g_\theta)_1<0,
    \]
    which would contradict nonegativity of \(\theta\) in a neighborhood of \(s=0\).\\

    \emph{Step 3: We show that for small angles \(\alpha>0\), we have \(0<\theta<\pi\) and \((g_{\theta_\alpha})_1 < 0\)}\\
    By \Cref{cor:planar_branch}, the equilibrium branch near the
    origin can be written as
    \[
        \theta_\alpha
        =
        q(\alpha)\psi+w(q(\alpha),\alpha),
    \]
    with
    \[
        q(0)=0,
        \qquad
        q'(0)
        =
        \frac{\mathsf M(2\mathcal D+1)}
        {\delta^4B_0(\psi,\psi)}
        >0,
    \]
    and
    \[
        \partial_qw(0,0)
        =
        \partial_\alpha w(0,0)
        =
        0.
    \]
    Therefore
    \[
        \left.
        \frac{d}{d\alpha}\theta_\alpha
        \right|_{\alpha=0}
        =
        q'(0)\psi.
    \]
    Since for every \(s\in(0,1]\), we have
    \[
        \psi(s)>0
        \qquad
        \text{and} \quad
        \psi'(0)>0
    \]
    it follows that
    \[
        \theta_\alpha(s)>0
        \qquad
        \forall s\in(0,1]
    \]
    for all sufficiently small \(\alpha>0\).

    Similarly,
    \[
        \left.
        \frac{d}{d\alpha}
        (g_{\theta_\alpha})_1
        \right|_{\alpha=0}
        =
        \frac{3q'(0)-(2\mathcal D+1)}{\delta^4}.
    \]
    Using
    \[
        B_0(\psi,\psi)>\frac{3\mathsf M}{\delta^4},
    \]
    we obtain
    \begin{align}
        q'(0)
        & =
        \frac{\mathsf M(2\mathcal D+1)}
        {\delta^4B_0(\psi,\psi)}
        <
        \frac{2\mathcal D+1}{3},
    \end{align}
    and therefore
    \[
        \left.
        \frac{d}{d\alpha}
        (g_{\theta_\alpha})_1
        \right|_{\alpha=0}
        <0.
    \]
    Thus
    \[
        (g_{\theta_\alpha})_1<0
    \]
    for all sufficiently small \(\alpha>0\).\\

    We now return to the continuation parameter \(t\). Since the branch
    \[
        t\longmapsto(q(t),\alpha(t)),
        \qquad
        t\in[0,a),
            \]
            is obtained by continuation of the local branch above, there exists
            \(t_0>0\) such that
            \[
            0<\theta_t(s)<\pi
            \qquad
            \text{for every }s\in(0,1],
    \]
    and
    \[
        (g_t)_1<0
    \]
    for every \(t\in(0,t_0]\).

    \emph{Step 4: We show that these inequalities persist along the continued branch from the straight line:}\\
    First, we observe that \(t\mapsto\theta_t\) is continuous in \(C^2([0,1])\). Indeed,
    the smooth dependence of \(w(q,\alpha)\) in \(\mathcal U\) and the
    continuity of \(t\mapsto(q(t),\alpha(t))\) imply that
    \[
        \theta_t
        \longrightarrow
        \theta_{t_*}
        \qquad
        \text{in }H^1(0,1)
    \]
    whenever \(t\to t_*\). Since\(
    H^1(0,1)\hookrightarrow C^0([0,1]),
    \)
    we also have
    \[
        \theta_t
        \longrightarrow
        \theta_{t_*}
        \qquad
        \text{uniformly on }[0,1].
    \]
    \[
        X_{\theta_t}
        =
        \int_0^1d_{\theta_t}\,ds
        \longrightarrow
        X_{\theta_{t_*}} \quad \text{and}\quad g_t
        =
        \nabla\Phi_{\alpha(t)}(X_{\theta_t})
        \longrightarrow
        g_{t_*}.
    \]

    Define
    \[
        f_t(s)
        :=
        \mathsf G(1-s)\sin\theta_t(s)
        -
        \mathsf M g_t\cdot n_{\theta_t}(s).
    \]
    Then, the above convergences imply that
    \[
        f_t
        \longrightarrow
        f_{t_*}
        \qquad
        \text{in }C^0([0,1]).
    \]
    Since the Euler-Lagrange equation is
    \[
        -\theta_t''=f_t,
        \qquad
        \theta_t'(1)=0,
    \]
    we get the required continuity in \(C^2([0,1])\).\\

    Define
    \[
        \mathcal S
        :=
        \left\{
        t\in(0,a):
        0<\theta_t(s)<\pi\ \text{for every }s\in(0,1],
        \;
        (g_t)_1<0
        \right\}.
    \]
    By the local branch analysis, there exists \(t_0>0\) such that
    \[
        (0,t_0)\subset\mathcal S,
    \]
    so \(\mathcal S\neq\varnothing\).\\

    We claim that \(\mathcal S\) is both open and closed in \((0,a)\).
    Since \((0,a)\) is connected, this will imply
    \[
        \mathcal S=(0,a).
    \]

    To prove closedness, let \(t_n\in\mathcal S\) and suppose
    \[
        t_n\longrightarrow t_*\in(0,a).
    \]
    By the \(C^2\)-continuity of \(t\mapsto\theta_t\) and the continuity of
    \(t\mapsto g_t\),
    \[
        0\leq\theta_{t_*}(s)\leq \pi
        \qquad
        \text{for every }s\in[0,1],
    \]
    and
    \[
        (g_{t_*})_1\leq0.
    \]

    Since
    \[
        0<\alpha(t_*)\leq\frac{\pi}{2},
    \]
    Step 1 excludes the possibility \((g_{t_*})_1=0\), resulting in
    \[
        (g_{t_*})_1<0.
    \]
    Step 2 then yields
    \[
        0<\theta_{t_*}(s)<\pi
        \qquad
        \text{for every }s\in(0,1].
    \]
    Therefore \(t_*\in\mathcal S\), and \(\mathcal S\) is closed in
    \((0,a)\).

    To prove openness, let \(t_*\in\mathcal S\). Since
    \[
        (g_{t_*})_1<0,
    \]
    continuity of \(g_t\) gives
    \[
        (g_t)_1<0
    \]
    for all \(t\) sufficiently close to \(t_*\). Moreover, Step 2 gives
    \[
        \theta_{t_*}'(0)>0.
    \]
    Thus, there exist \(\varepsilon>0\) and \(c>0\) such that
    \[
        \theta_{t_*}'(s)\geq2c
        \qquad
        \text{for }0\leq s\leq\varepsilon.
    \]
    By \(C^2\)-continuity in \(t\),
    \[
        \theta_t'(s)\geq c
        \qquad
        \text{for }0\leq s\leq\varepsilon
    \]
    whenever \(t\) is sufficiently close to \(t_*\). Since
    \(\theta_t(0)=0\),
    \[
        \theta_t(s)
        =
        \int_0^s\theta_t'(r)\,dr
        \geq
        cs>0
        \qquad
        \text{for }0<s\leq\varepsilon.
    \]

    On the compact interval \([\varepsilon,1]\), define
    \[
        m
        :=
        \min_{s\in[\varepsilon,1]}
        \theta_{t_*}(s)
        >0.
    \]
    By uniform continuity of \(\theta_t\) with respect to \(t\),
    \[
        \|\theta_t-\theta_{t_*}\|_{C^0}
        <
        \frac{m}{2}
    \]
    for \(t\) sufficiently close to \(t_*\). Thus,
    \[
        \theta_t(s)\geq\frac{m}{2}>0
        \qquad
        \text{for every }s\in[\varepsilon,1].
    \]
    A similar argument shows that we can also make \(\theta_t(s)<\pi\) for every \(s\in[0,1]\) and
    thus \(t\in\mathcal S\) for all \(t\) sufficiently close to \(t_*\),
    and \(\mathcal S\) is open in \((0,a)\).\\

    Therefore
    \[
        \mathcal S=(0,a).
    \]
    resulting in
    \[
        0<\theta_t(s)<\pi
        \qquad
        \text{for every }s\in(0,1],
        \qquad
        (g_t)_1<0
    \]
    for every \(t\in(0,a)\).
\end{appendixproof}
\begin{appendixproof}[Proof of \Cref{thm:planar_fold}]\label{pf:planar_fold}
    Since
    \[
        \partial_\alpha F(q_*,\alpha_*)\neq0,
    \]
    the implicit function theorem gives a unique smooth function
    \[
        \alpha=\alpha(q)
    \]
    whose graph describes the planar equilibria near \((q_*,\alpha_*)\). Differentiating
    \[
        F(q,\alpha(q))=0
    \]
    and using \(\partial_qF(q_*,\alpha_*)=0\) gives
    \[
        \alpha'(q_*)=0.
    \]
    Differentiating once more at \(q=q_*\) yields
    \[
        \partial_{qq}F(q_*,\alpha_*)
        +
        \partial_\alpha F(q_*,\alpha_*)\alpha''(q_*)
        =0,
    \]
    and
    \[
        \alpha''(q_*)
        =
        -\frac{\partial_{qq}F(q_*,\alpha_*)}
        {\partial_\alpha F(q_*,\alpha_*)}.
    \]
    Since \(\partial_{qq}F(q_*,\alpha_*)\neq0\), the fold is quadratic and therefore a nondegenerate saddle-node. The stated expansion follows from Taylor's theorem. \Cref{thm:planar_stability} and \(\partial_qF(q_*,\alpha_*)=0\)
    implies that \(Q_{\mathrm{pl}}\) is nonnegative and
    \[
        \ker Q_{\mathrm{pl}}
        =
        \operatorname{span}\{\xi_{q_*,\alpha_*}\}.
    \]

    For completeness, we record the derivatives used below to verify the nondegeneracy conditions. Write
    \[
        \theta=\theta_{q,\alpha},
        \qquad
        \xi=\xi_{q,\alpha},
    \]
    and define
    \[
        V=\int_0^1n_\theta\xi\,ds,
        \qquad
        D=\int_0^1d_\theta\xi^2\,ds,
        \qquad
        N=\int_0^1n_\theta\xi^3\,ds.
    \]
    Using
    \[
        Q_{\mathrm{pl}}(\xi,v)=0
        \qquad
        \text{for every }v\in Y,
    \]
    and the fact that \(\partial_\alpha\theta,\partial_q\xi\in Y\), differentiation of the reduced scalar gives
    \[
        \partial_\alpha F
        =
        \mathsf M\,
        \partial_\alpha\nabla\Phi_\alpha(X_\theta)\cdot V,
    \]
    and
    \begin{align}
        \partial_{qq}F
        &=
        \mathsf G\int_0^1(1-s)\sin\theta\,\xi^3\,ds
        +\mathsf M\nabla^3\Phi_\alpha(X_\theta)[V,V,V]\\
        &\quad
        -3\mathsf M\nabla^2\Phi_\alpha(X_\theta)[D,V]
        -\mathsf M\nabla\Phi_\alpha(X_\theta)\cdot N.
    \end{align}
    These formulas will be used below to verify the two nondegeneracy conditions at the computed folds.
\end{appendixproof}

\begin{appendixproof}[Proof of \Cref{cor:tip_angle_fold}]\label{pf:tip_angle_fold}
    Let
    \[
        \Theta(q):=\Theta_{\mathrm{tip}}(q,\alpha(q)),
    \]
    where \(\alpha=\alpha(q)\) is the local equilibrium branch given by
    \Cref{thm:planar_fold}. Since
    \[
        \alpha'(q_*)=0,
    \]
    we have
    \begin{align*}
        \frac{d}{dq}\theta_{q,\alpha(q)}\Big|_{q=q_*}&=\partial_q\theta_{q_*,\alpha_*}= \xi_{q_*,\alpha_*},\\
        \frac{d}{dq}X_{\theta_{q,\alpha(q)}}\Big|_{q=q_*}& = \int_0^1 n_{\theta_*}\xi_{q_*,\alpha_*}\,ds=V_*.
    \end{align*}

    Thus, differentiating \(\Theta(q)\) at \(q=q_*\)
    gives
    \[
        \Theta'(q_*)
        =
        \frac{
            (X_*\cdot e_3)(V_*\cdot e_1)
            -
            (X_*\cdot e_1)(V_*\cdot e_3)
        }{
            |X_*|^2
        }
        =
        T_*.
    \]
    By assumption \(T_*\neq0\), so the inverse function theorem shows that
    \(\Theta_{\mathrm{tip}}\) is a valid local coordinate for the equilibrium branch near the fold and
    \[
        q-q_*
        =
        \frac{\Theta_{\mathrm{tip}}-\Theta_{\mathrm{tip},*}}{T_*}
        +
        o\!\left(
        |\Theta_{\mathrm{tip}}-\Theta_{\mathrm{tip},*}|
        \right).
    \]

    Substituting this into \Cref{thm:planar_fold}, we get
    \[
        \alpha
        =
        \alpha_*
        -
        \frac{\partial_{qq}F(q_*,\alpha_*)}
        {2\partial_\alpha F(q_*,\alpha_*)T_*^2}
        \left(
        \Theta_{\mathrm{tip}}-\Theta_{\mathrm{tip},*}
        \right)^2
        +
        o\!\left(
        |\Theta_{\mathrm{tip}}-\Theta_{\mathrm{tip},*}|^2
        \right).
    \]
\end{appendixproof}

\section*{Acknowledgements}
RV gratefully acknowledges support from the National Science Foundation NSF DMS-2407592. RV and LG also acknowledge support from the RTG grant of the Department of Mathematics, at the University of Utah (Award no. NSF DMS- 2136198). MT and VD acknowledge support from the Prime Minister's Research Fellowship (PMRF ID 0203695). MT also thanks Avinash Umashankar and Fajal, for their assistance with 3D printing and experimental setup.\\

\section{Data availability statement}
Experimental data generated in this research is available in the manuscript. 
\printbibliography
\end{document}